\documentclass[11pt]{amsart}
\usepackage[colorlinks,citecolor=red,pagebackref,hypertexnames=false]{hyperref}

\makeatletter
\def\@tocline#1#2#3#4#5#6#7{\relax
  \ifnum #1>\c@tocdepth 
  \else
    \par \addpenalty\@secpenalty\addvspace{#2}%
    \begingroup \hyphenpenalty\@M
    \@ifempty{#4}{%
      \@tempdima\csname r@tocindent\number#1\endcsname\relax
    }{%
      \@tempdima#4\relax
    }%
    \parindent\z@ \leftskip#3\relax \advance\leftskip\@tempdima\relax
    \rightskip\@pnumwidth plus4em \parfillskip-\@pnumwidth
    #5\leavevmode\hskip-\@tempdima
      \ifcase #1
       \or\or \hskip 1em \or \hskip 2em \else \hskip 3em \fi%
      #6\nobreak\relax
    \dotfill\hbox to\@pnumwidth{\@tocpagenum{#7}}\par
    \nobreak
    \endgroup
  \fi}
\makeatother

\usepackage{tikz,amsthm,amsmath,amstext,amssymb,amscd,epsfig,euscript, mathrsfs, dsfont,pspicture,multicol, mathtools,graphpap,graphics,graphicx,times,enumerate,subfig,sidecap,wrapfig,color}

\usepackage{ hyperref}
\usepackage{enumerate}

 \numberwithin{equation}{section}

\def\bR{{\mathbb{R}}}

\def\NN{{\mathbb{N}}}

\def\cH{{\mathcal{H}}}
\def\cF{{\mathcal{F}}}
\def\mR{{\mathcal{R}}}
\def\mS{{\mathcal{S}}}

\def\PP{{\mathcal{P}}}

\def\cF{{\mathscr{F}}}

\def\cM{{\mathscr{M}}}

\def\cP{{\mathscr{P}}}

\def\ve{\varepsilon}
\def\vp{\varphi}

\renewcommand{\d}{{\partial}}

\def\bx{\bar x}
\def\by{\bar y}
\def\bz{\bar z}
\def\bxi{\bar{\xi}}

\def\wt{\widetilde}

\def\Rn1{\mathbb{R}^{n+1}}
\def\rrn{\mathbb{R}^{n+1}}
\def\rn{\mathbb{R}^{n}}
 
\def\oom{\overline \Omega} 					

\DeclareMathOperator{\diam}{diam}
\def\Cap{\textup{Cap}} 					

\def\Tan{\mathop\mathrm{Tan}} 					

\def\BMO{\mathop\mathrm{BMO}} 					
\def\VMO{\mathop\mathrm{VMO}} 					
\def\Lip{\mathop\mathrm{Lip}} 						
 
\def\lipp{\textup{Lip}_{01}}
 	
\def\dim{\mathrm{dim}} 					
\def\dist{\textup{dist}} 						
\def\supp{\mathop\mathrm{supp}}					
\def\loc{\mathop\mathrm{loc}}						
\def\dt{\partial_t} 						
\DeclareMathOperator*{\esssup}{ess\,sup}			
\renewcommand{\div}{\mathop\mathrm{div }}			
\def\warrow{\rightharpoonup}								

\newcommand{\cnj}[1]{\overline{#1}}

\def\Xint#1{\mathchoice
{\XXint\displaystyle\textstyle{#1}}%
{\XXint\textstyle\scriptstyle{#1}}%
{\XXint\scriptstyle\scriptscriptstyle{#1}}%
{\XXint\scriptscriptstyle\scriptscriptstyle{#1}}%
\!\int}
\def\XXint#1#2#3{{\setbox0=\hbox{$#1{#2#3}{\int}$ }
\vcenter{\hbox{$#2#3$ }}\kern-.58\wd0}}

\def\avint{\Xint-}
\theoremstyle{plain}
\newtheorem{theorem}{Theorem}

\newtheorem{corollary}[theorem]{Corollary}

\newtheorem{lemma}[theorem]{Lemma}

\newtheorem{proposition}[theorem]{Proposition}

\theoremstyle{definition}
\newtheorem{example}[theorem]{Example}
\newtheorem{definition}[theorem]{Definition}
\newtheorem{remark}[theorem]{Remark}

\numberwithin{equation}{section}
\numberwithin{theorem}{section}

  \DeclareFontFamily{U}{mathb}{\hyphenchar\font45} 
\DeclareFontShape{U}{mathb}{m}{n}{
      <5> <6> <7> <8> <9> <10> gen * mathb
      <10.95> mathb10 <12> <14.4> <17.28> <20.74> <24.88> mathb12
      }{}
\DeclareSymbolFont{mathb}{U}{mathb}{m}{n}

\DeclareMathSymbol{\toitself}      {3}{mathb}{"FD}  

\def\HH{\mathcal{H}}
\newcommand{\vv}{\vspace{2mm}}
\newcommand{\vvv}{\vspace{4mm}}

\newcommand{\dv}{\mathop{\rm div}}
\def\R{\mathbb{R}}
\def\vphi{\varphi}
\def\om{\Omega}
\def\zz{\bar 0}

\def\hm{\omega}

  \newtheorem{main}{Theorem}

\begin{document}

\title[Blow-ups   of caloric  measure]{Blow-ups  of caloric measure in  time varying domains and applications to  two-phase problems}

\author[Mihalis Mourgoglou]{Mihalis Mourgoglou}
\address{Departamento de Matem\'aticas, Universidad del Pa\' is Vasco, Barrio Sarriena s/n 48940 Leioa, Spain and\\
Ikerbasque, Basque Foundation for Science, Bilbao, Spain.}
\email{michail.mourgoglou@ehu.eus}

\author{Carmelo Puliatti}
\address{Departamento de Matem\'aticas, Universidad del Pa\' is Vasco, Barrio Sarriena s/n 48940 Leioa, Spain. }
\email{carmelo.puliatti@ehu.eus}

\thanks{M.M. was supported  by IKERBASQUE and partially supported by the grant MTM-2017-82160-C2-2-P of the Ministerio de Econom\'ia y Competitividad (Spain), and by  IT-1247-19 (Basque Government). C.P. was supported by IT-1247-19 (Basque Government).}

\newcommand{\mih}[1]{\marginpar{\color{red} \scriptsize \textbf{Mi:} #1}}
\newcommand{\car}[1]{\marginpar{\color{blue} \scriptsize \textbf{Carmelo:} #1}}
\maketitle

\begin{center}
{\it In memory of Professor Michel Marias (1953--2020)}
\end{center}

\begin{abstract}
We develop a method to study the  structure of the common part of the boundaries of   disjoint and possibly non-complementary  time-varying  domains in $\Rn1$, $n \geq 2$, at the  points of mutual absolute continuity of their respective caloric measures. Our set of techniques, which is based on parabolic tangent measures, allows us to tackle the following problems:
\begin{enumerate}
\item Let $\om_1$ and $\om_2$ be disjoint domains in $\Rn1$, $n \geq 2$, which are quasi-regular  for the heat equation and regular for the adjoint heat equation, and their complements satisfy a mild non-degeneracy hypothesis  on the set $E \subset \d\om_1 \cap \d\om_2$  of mutual absolute continuity of the associated caloric measures $\omega_i$ with poles at $\bar p_i=(p_i,t_i)\in\Omega_i$, $i=1,2$. Then, we obtain a parabolic analogue of the results of  Kenig, Preiss, and Toro, i.e.,   we show that the parabolic Hausdorff dimension of $\hm_1|_E$ is $n+1$ and  the tangent measures of $\hm_1$  at $\hm_1$-a.e. point of $E$ are equal to a constant multiple of the  parabolic $(n+1)$-Hausdorff measure restricted to hyperplanes  containing a line parallel to the time-axis.
\item If, additionally, $\hm_1|_E$ and $\hm_2|_E$ are doubling, $\log \frac{d\hm_2|_E}{d\hm_1|_E} \in \VMO(\hm_1|_E)$, and $E$ is relatively open in the support of $\omega_1$, then  their tangent measures at {\it every} point of $E$ are caloric measures associated with adjoint caloric polynomials. As a corollary we obtain that if $\om_1$ is a $\delta$-Reifenberg flat domain for $\delta$ small enough and $\om_2=\rrn \setminus \overline \om_1$, and  $\log \frac{d\hm_2}{d\hm_1} \in \VMO(\hm_1)$, then $\om_1 \cap \{t<t_2\}$ is vanishing Reifenberg flat. This generalizes results of   Kenig and Toro for the Laplacian.
\item Finally, we establish a parabolic version of a theorem of Tsirelson about triple-points for harmonic measure. Assuming that $\om_i$, $1 \leq i \leq 3$, are quasi-regular domains for both the heat  and  the adjoint heat equations, the set of points on $\cap_{i=1}^3  \d \om_i$, where the three caloric measures are mutually absolutely continuous has null caloric measure.
\end{enumerate}
 In the course of  proving our main theorems, we obtain new results on  heat potential  theory, parabolic geometric measure theory,  nodal sets of caloric functions, and  Optimal Transport, that may be of independent interest. 
\end{abstract}

\tableofcontents

\section{Introduction}

In this paper, we  study   two-phase problems for  harmonic measure associated with the heat equation, which is traditionally called  {\it caloric measure}. We consider two disjoint open sets $\om^+$ and $\om^-$   in $\rrn$, and $\hm^+$ and $\hm^-$  the associated caloric measures, so that their common boundary $\partial \Omega^+\cap \partial\Omega^-$ has positive $\hm^+$ measure. Our main goal is to study how mutual absolute continuity of the respective caloric measures provides information about the infinitesimal structure of $\partial \Omega^+\cap \partial\Omega^-$.

In the elliptic case, Bishop, Carleson, Garnett, and Jones studied in \cite{BCGJ89}  a two-phase problem for simply connected planar disjoint domains: they showed that the two harmonic measures are mutually singular if and only if the intersection of the respective sets of inner tangent points has zero one-dimensional Hausdorff measure.
More recently, in \cite{KPT09}, Kenig, Preiss, and Toro studied the higher-dimensional case, assuming that $\Omega^+$ and $\Omega^-=\Rn1\setminus \overline{\Omega^+}$ are both NTA-domains (as defined by Jerison and Kenig in \cite{JK82}). Combining  the blow-up analysis at points of mutual absolute continuity in \cite{KT06} with the theory of tangent measures along with the monotonicity formula of Alt, Caffarelli and Friedman (see \cite{ACF84}), they showed that mutual absolute continuity of the interior and exterior harmonic measure $\omega^+$ and $\omega^-$ implies that the harmonic measure has  Hausdorff dimension $n$ and that as we zoom in at $\hm^+$-a.e. point of the common boundary,   $\partial \Omega^\pm$ looks flatter and flatter. This blow-up technique was further improved by Azzam, the first named author, and Tolsa  \cite{AMT16} and the same authors along with Volberg in \cite{AMTV16}, who eventually showed that, without any further assumption on the domains, the harmonic measure restricted to the set of mutual absolute continuity is $n$-rectifiable.  The connection between  the Riesz transform and  $n$-rectifiability was an  element of major importance in the proofs of  \cite{AMT16} and \cite{AMTV16} that allowed to improve on the results of \cite{KPT09} from this perspective as well (since $n$-rectifiable measures have Hausdorff dimension $n$).

All the aforementioned  results describe {\it a.e.} phenomena, while it is interesting to investigate conditions that  ensure some nice limiting behavior of our blow-ups at  {\it every} point. In \cite{KT06}, Kenig and Toro considered $\om^+$  and $\om^-$  to be complementary $2$-sided NTA  domains and $\log\frac{d\hm^-}{d\hm^+} \in VMO(d\hm^+)$, and they showed that, in a sequence of arbitrarily small scales, the boundary around any point starts resembling  the zero set of a harmonic polynomial (instead of a hyperplane).  If the domains are $\delta$-Reifenberg flat, for $\delta$ small enough, then the conclusion improves to a hyperplane implying that the domains are Reifenberg flat with vanishing constant. They also proved that the same conclusion holds without the $\delta$-Reifenberg flatness assumption if we assume that $\Omega^+$ is two-sided NTA, $\partial \Omega^+$ is Ahlfors-David regular and the logarithms of the interior and exterior Poisson kernels are in $\VMO$ with respect to the surface measure on $\partial \Omega^+$. Badger went a step further in \cite{Bad11}, showing that those harmonic polynomials are always homogeneous, while later, in \cite{Bad13}, he explored the topological properties of sets where the boundary is approximated by zero sets of harmonic polynomials in this way. {Badger, Engelstein and Toro investigated in \cite{BET17} two-phase problems for the harmonic measure below the continuous threshold using \cite{KPT09}, \cite{Bad11}, \cite{Bad13}, and an interesting structure theorem for sets that are locally bilaterally well approximated by zero sets of harmonic polynomials. More recently, the geometry of the singular set of the boundary $\partial \Omega^\pm$ of complementary NTA domains such their caloric measures $\omega^\pm$ are mutually absolutely continuous and $\log (d\omega^-/d\omega^+)$ is H\"older continuous has been studied in the work of McCurdy \cite{McC19}.} In \cite{AM19}, Azzam and the first named author, among others, extended the results in \cite{KPT09}, \cite{KT06}, and \cite{Bad11} to general  domains and elliptic measures associated with uniformly elliptic operators in divergence form with merely bounded coefficients that are also in the Sarason's space of vanishing mean oscillation with respect to $\hm^+$. For relevant results for elliptic operators with  $W^{1,1}$-coefficients see \cite{TZ17}.

It is interesting to understand if similar phenomena occur also in the context of parabolic PDEs and, in particular, the heat equation.
The implementation of blow-up methods in the context of one-phase free boundary problems for the heat equation was initiated by Hofmann, Lewis, and Nystr\"om in \cite{HLN04}. They extended the work of Kenig and Toro showing that, in parabolic chord arc domains with vanishing constant, the logarithm of the density of caloric measure with respect to a certain projective Lebesgue measure (or else the Poisson kernel) is of vanishing mean oscillation, and  also obtained  a partial converse, which amounts to a one-phase problem in Reifenberg flat domains with parabolic uniformly rectifiable boundaries.  The full converse was proved by Engelstein in  \cite{Eng17}, where a key fact of his proof was a delicate classification of ``flat" blow-ups, which was an open problem in the parabolic setting. He also examined free boundary problems with conditions above the continuous threshold. This problem had already been solved by Nystr\"om  (see \cite{Nys06B} and \cite{Nys12}) in graph domains under the assumption that  the Green function is comparable with the distance function from the boundary.

Results analogous to those in \cite{KT06} for the heat equation were considered  by Nystr\"om in \cite{Nys06}. He proved that, if $\Omega$ and $\rrn \setminus \overline \om$ are  parabolic NTA domains with parabolic Ahlfors regular boundary and the logarithms of  the associated Poisson kernels  are of vanishing mean oscillation, then a portion of $\partial \Omega$ suitably away from the poles of the caloric measures is Reifenberg flat with vanishing constant. Furthermore, it was  shown in \cite{Nys06C} that if one drops the Ahlfors regularity hypothesis and, instead of parabolic NTA, they ask for the domains to be $\delta$-Reifenberg flat for $\delta>0$ small enough, then the same conclusion holds if one substitutes the assumption on the Poisson kernels with a vanishing mean oscillation hypothesis on the logarithm of the density $d\omega^-/d\omega^+$ of the caloric measures associated with the two domains. We remark that the proof  uses the construction of the Green function and caloric measure with pole at infinity. Finally, we refer to \cite{Eng17B} for an interesting geometric result for planar NTA domains  with an application to the previously discussed  one-phase  problem.


To the best of our knowledge, the present paper is the first one  that studies free boundary problems in so general domains.

\vvv

A proficient way to address the kind of questions we referred above is to analyze the fine structure of the boundary by ``zooming-in'' on  boundary points. There are two ways to do that. The first one is by looking at the Hausdorff convergence of rescaled copies of the support of a measure as we zoom in, for which we follow the framework of Badger and Lewis \cite{BL15}.

We define the  ball with respect to the parabolic norm \(
\|\by\| \coloneqq \max \{|y|, |s|^{1/2}\}
\)  centered at  $\bx \in \rrn$ with radius $r>0$  as
$$
C_r(\bx)=\{\by \in \rrn:\|\by-\bx\|<r\}.
$$ 
Note that $C_r(\bar x)$ is  a euclidean right circular cylinder centered at $\bx$ of height $2r^2$ and radius $r$. If $\bx=\bar 0$, we simply write  $C_r(\bar x)=C_r$ unless stated otherwise. We also need the following {\it time-backwards} and {\it time-forwards} versions of the parabolic ball:
\begin{equation*}
\begin{split}
C^-_r(x,t)&=\{(y,s) \in \rrn: y \in  B_r(x),\,\, t-r^2<s<t\} \\
C^+_r(x,t)&=\{(y,s) \in \rrn: y \in  B_r(x),\,\, t<s<t+r^2\}.
\end{split}
\end{equation*}

{A point   $\bx \in \rrn$ is denoted by 
$\bx=(x,t)=(x',x_n,t)\coloneqq(x_1, \dots, x_n ,t)$
and we also use the notation $\bar 0=(0,\ldots,0)$.
For $r>0$ and $\bx,\by\in \rrn$, we set
\[
\delta_r(\bx)\coloneqq (r x, r^2 t), \quad\textup{and}\quad
T_{\by, r}(\bx)\coloneqq \delta_{1/r}(\bx-\by).
\]
}
Let $A\subset\Rn1$ be a set and let $\mathscr S$ be a collection of subsets of $\Rn1$. Given $\bar x \in A$ and $r>0$ we define
\[
\Theta^{\mathscr S}_A(\bar x, r)=\inf_{S\in\mathscr S}\, \max\biggl\{\sup_{\bar a\in A\cap C_r(\bar x)}\frac{\dist(\bar a, \bar x+ S)}{r}, \sup_{\bar z\in(\bar x+ S)\cap C_r(\bar x)}\frac{\dist(\bar z, A)}{r}\biggr\},
\]
where ``$\dist$'' denotes the distance with respect to the parabolic norm $\|\cdot\|$.
{Let $\mathscr S$ be a \textit{local appoximation class}, i.e. let us assume that for all $S\in \mathscr S$ and $\lambda>0$, we have that $\delta_\lambda S\in\mathscr S$.} We say that $A$ is  \textit{locally bilaterally well approximated by $\mathscr{S}$}, also abbreviated $LBWA(\mathscr S)$, if for all $\varepsilon>0$ and all $K\subset A$ compact, there is $r_{\varepsilon, K}>0$ such that $\Theta^{\mathscr S}_A(\bar x,r)<\varepsilon$ fro all $\bar x\in K$ and $0<r<r_{\varepsilon, K}$. {We say that $\bar x\in A$ is an $\mathscr S$-point of $A$ if $\lim_{r\to 0}\Theta^{\mathscr S}_A(\bar x,r)= 0$.} In other words, for $\bx$ to be a $\mathscr{S}$ point means that, as we zoom in on A at the point $\bx$, the set $A$ resembles more and more an element of $\mathscr{S}$. Remark that this element may change as we move from one scale to the next.

The second way is by investigating the weak convergence of rescaled copies of the measure itself.  In  \cite{Pr87}, Preiss developed the theory of \textit{tangent measures}, which turned out to be vital in the study of  two-phase problems for harmonic and elliptic measure; his results are at the core of  \cite{KPT09}, \cite{Bad11}, \cite{AMT16}, \cite{AMTV16}, and \cite{AM19}. The definition of a tangent measure can  be readily adapted to the parabolic context, in which $\Rn1$ is naturally endowed with a set of non-isotropic dilations. {Furthermore, we remark that although the theory of local set approximation was originally developed by Badger and Lewis in the euclidean setting, it can be generalized to $\mathbb R^{n+1}$ endowed with the parabolic metric (see \cite[Remark 1.8]{BL15} and the references therein).}

We say that $\nu$ is a {\it tangent measure} of $\mu$ at a point $\bx \in\rrn$ and write $\nu\in \Tan(\mu,\bx)$ if
$\nu$ is a non-zero Radon measure on $\rrn$ and there are sequences $c_{i}>0$ and $r_{i}\searrow 0$ so that $c_i\,T_{\bx,r_{i}}[\mu]$ converges weakly to $\nu$ as $i\to\infty$. 

Let us record here that a parabolic version of Besicovitch covering theorem is available for parabolic balls (see \cite{It18}), which  allows for   the  basic properties of tangent measures to hold in the parabolic setting by the same proofs as in the Euclidean case.  For further details, we refer to Section \ref{sec:GMT}.

\vv

The geometry of the blow-ups at a point of mutual absolute continuity of caloric measures is intimately related with that of nodal sets of caloric functions.  We define the heat and adjoint heat equations by
\begin{equation}\label{eq:Heat operator-intr}
Hu  :=  \Delta  u - \dt u=0 , \quad \textup{and }\quad H^* u :=  \Delta u + \dt u =0, 
\end{equation}
and  a $C^{2,1}$ function satisfying $Hu=0$ (resp. $H^*u=0$) is called (resp. adjoint) caloric function.

{As in the harmonic case, polynomials play a special role in the theory of caloric functions. Indeed, as harmonic functions of polynomial growth are polynomials, it has been long known that the so-called \textit{ancient solutions} of the heat equation of polynomial growth are polynomials. For a more in-depth discussion, we refer to \cite{LZ19} and the references therein.}
Let $\Theta$ denote the set of caloric functions vanishing at $\bar 0$, $P (d)$  the set of the caloric polynomials {of degree $d$} vanishing at $\bar 0$, and $F(d)$ the set of homogeneous caloric polynomials of degree $d$.
In Lemma \ref{lem:cal-measure-assoc-h}, we show  that,  for any caloric function $h$ on $\Rn1$, there exists a unique Radon measure $\omega_h$ such that
\[
\int \varphi \, d\omega_h = \frac{1}{2}\int |h|(\Delta + \partial_t)\varphi, \qquad \text{ for all }\varphi\in C^\infty_c(\Rn1).
\]
The measure $\hm_h$  is called  \textit{caloric measure} $\omega_h$  \textit{associated to $h$}. In fact, this  is the unique adjoint caloric measure with pole at infinity  in $\Omega^\pm=\{ h^\pm \neq 0\}$ and $h^\pm$ is  the Green function with pole at infinity.
If $h$ is an adjoint caloric function in $\rrn$, we define the \textit{adjoint caloric measure}  analogously.

We say that a function $f(x,t)$ is (parabolic) \textit{homogeneous} of degree $m\in \R$ if for any $\lambda>0$ and $(x,t)\in \rrn\setminus \{\bar 0\},$ we have
\[
f\bigl(\delta_\lambda(x,t)\bigr)=f(\lambda x, \lambda^2 t)=\lambda^mf(x,t).
\]
We remark that every caloric polynomial can be written as the sum of homogeneous caloric polynomials. Indeed, assuming that $h(x,t)=\sum_{\alpha, \ell}c_{\alpha, \ell}x^\alpha t^\ell$ for $c_{\alpha, \ell}\in\R$ and that $c_{\alpha, \ell}=0$ if $|\alpha|+2\ell> d$, we can write
\[
h(x,t)=\sum_{j=0}^d \Bigl(\sum_{|\alpha|+2\ell =j}c_{\alpha, \ell}x^\alpha t^\ell\Big)\eqqcolon \sum_{j=0}^d h_j(X,t).
\]
Observe that $h_j$ is a homogeneous polynomial of degree $j$, which implies that $Hh_j$ is a homogeneous polynomial of degree $j-2$. So 
\[
Hh= \sum_j Hh_j =0
\]
if and only if $Hh_j=0$ for every $j=0,\ldots, d$.
Hence, it is consistent to say that a caloric function  $h$ is a caloric polynomial of degree $d$ if it can be written as the sum of homogeneous caloric polynomials of degree at most $d$. Thus, we define 
\[
\mathscr{H}_\Theta \coloneqq \{\omega_h:\, h\in \Theta\},
\quad
\mathscr P(d)\coloneqq \{\omega_h:\, h\in P(d)\}
\quad \text{ and } \quad 
\mathscr F(d)\coloneqq \{\omega_h:\, h\in F(d)\}.
\]
Given a caloric function $h$, we use the notation $\Sigma^h\coloneqq \{h=0\}$.
Furthermore, we indicate 
$$
\mathscr H_\Sigma=\{\Sigma^h: h\in \Theta\},
$$
\[
\mathscr P_\Sigma(d)=\{\Sigma^h: h\in P(1) \cup \dots \cup P(d)\}
\qquad \text{ and } \qquad 
\mathscr F_\Sigma(d)=\{\Sigma^h: h\in F(d)\}.
\]
The families $\Theta^*$, $P^*(d)$, $F^*(d)$,$ \mathscr T^*$, $\mathscr P^*(d)$, $\mathscr F^* (d)$,  $\mathscr T_\Sigma^*$, $\mathscr P^*_\Sigma(d)$, and $\mathscr F^*_\Sigma (d)$ are defined analogously for adjoint caloric functions, polynomials, and homogeneous polynomials. Set, moreover,
\[
\mathscr F_\Sigma\coloneqq \{V\subset \Rn1: V \text{ is an }n\text{-plane through } \bar 0 \text{ containing a line parallel to the }t\text{-axis}\}
\]
and 
\[
\mathscr F\coloneqq\{\mathcal H^{n+1}_p|_V: V \in \mathscr F_\Sigma\},
\]
and observe that $\mathscr F=\mathscr F(1)=\mathscr F^*(1)$ and $\mathscr F_\Sigma=\mathscr F_\Sigma(1)=\mathscr F^*_\Sigma(1)$.

\vvv

{Potential theory is also very important in  problems relating the metric properties of caloric measure to the geometry of the boundary but, in order to  take full advantage of it,  we need to ensure that the boundary is regular enough so that the Perron-Wiener-Brelot with continuous boundary data is continuous all the way to the boundary in that domain (sometimes up to a polar set of the boundary).  One can find several geometric conditions in the literature that imply the desired regularity of the boundary; e.g.,  the parabolic  Reifenberg flatness  (see the definition before Corollary \ref{cor:reifenberg}), the non-tangential accessibility (see e.g. \cite[pp. 263-264]{Nys06}}), and the  time-backwards Ahlfors-David regularity (see \cite{GH20}). A well-known  condition which implies regularity of the boundary in the Euclidean setting and has been used in the harmonic measure framework is the so-called Capacity-Density Condition (see e.g. \cite{AMT16}). The CDC  ensures that the capacity of the complement of  $\Omega$ in each ball centered at the boundary is large   in a scale invariant way. We will now introduce two similar notions of thickness that, in the parabolic context,  serve the same purpose as the CDC.

Given $\rho>0$ and $\bar x=(x,t)\in \Rn1$ we  define the \textit{heat ball} centered at $\bar x$ with radius $\rho$ to be the set
\begin{equation}\label{eq:def_heat_ball}
\begin{split}
E(\bar x;\rho)&\coloneqq  \Bigl\{(y,s)\in \Rn1 : |x-y|<\sqrt{2n(t-s)\log\Bigl(\frac{\rho}{t-s}\Bigr)}, \, t-\rho<s<t\Bigr\}\\
&=\bigl\{\bar y\in \Rn1 : \Gamma(\bar x- \bar y)>(4\pi\rho)^{-n/2}\bigr\},
\end{split}
\end{equation}
where $\Gamma(\cdot, \cdot)$ stands for the fundamental solution of the heat equation (see Section \ref{sec:potential theory}).
We indicate by $E^*(\bar x; r)\coloneqq\{(\bar y,s) \in \Rn1:(\bar y,t-s) \in E(\bar x,r)\}$ the \textit{adjoint heat ball}.

We say that  $\bx \in \d \om$  belongs to the \textit{parabolic boundary} of $\Omega$ (resp. adjoint parabolic boundary) and write $\bx \in \mathcal P\Omega$ (resp. $\mathcal P^*\Omega$), if  $C_r^-(\bar x)\cap \Omega^c\neq \varnothing$ (resp. $C_r^+(\bar x)\cap \Omega^c\neq \varnothing$ ) for every $r>0$. If  $\bx \in \mathcal P\Omega$ (resp. $\mathcal P^*\Omega$),  we  say that $\bx$ is in the \textit{bottom boundary} of $\Omega$ (resp. adjoint bottom boundary) and write $\bx \in \mathcal B\Omega$ (resp. $\mathcal B^*\Omega$) if  there exists $\ve>0$ such that $C_\ve^+(\bar x)\subset \Omega$ (resp. $C_\ve^-(\bar x)\subset \Omega$).  Moreover, we define the {\it lateral boundary} (resp. {\it adjoint lateral boundary}) as $\mathcal S' \om=\mathcal P \om \setminus \mathcal B \om$ (resp. $\mathcal{S'}^*\om=\mathcal P^* \om \setminus \mathcal B^*\om$) of $\om$ and denote by $\mathcal S\om$ the {\it quasi-lateral boundary} of $\om$ (for the precise definition see Section \ref{sec:preliminaries}).

\vv

We also refer to Section \ref{sec:potential theory} for the definition of $\Cap( \cdot)$, the thermal capacity of a set.

\begin{definition}
Let $\Omega \subset \rrn$  be an open set and  $F \subset  \mathcal P\Omega$ be compact. We say that a $\Omega$ satisfies the \textit{Time-Backwards Capacity Porosity Condition} (TBCPC)  {along} $F$  if there exists a constant $c>0$ and two sequences $\{\bar \xi_j\}_{j \geq 1} \subset F$ and $r_j\to 0$,  such that
\begin{equation}\label{eq:CPC}
\Cap\bigl(\overline{E{(\bar \xi_j; r^2_j)}} \cap \Omega^c \bigr) \geq c\, r_j^n.
\end{equation}
If $F= \{\bxi\}$, then we say that $\Omega$ satisfies the TBCPC at $\bxi$. We remark that  $c$, $\bxi_j$, and $r_j$ may depend on $F$.

Let $\om^+$ and $\om^-$ be two disjoint open sets in $\rrn$ with $\mathcal{S}' \om^+\cap \mathcal{S}' \om^- \neq  \emptyset$ and assume that  $F \subset \mathcal{S}' \om^+ \cap \mathcal{S}' \om^-$ is compact. Then we say that $\om^+$ and $\om^-$ satisfy the {\it joint} TBCPC along $F$ if there exists a sequence  $(\bxi_j,r_j)_{j\geq 1} \subset F \times (0,1)$ so that $r_j \to 0$ and  \eqref{eq:CPC} holds for $\om^\pm$, then there exists a subsequence of $(\bxi_j,r_j)_{j\geq 1} $ so that \eqref{eq:CPC} holds for $\om^\mp$ (with possibly different constant $c$).  We may define the time-forwards CPC (TFCPC) by replacing the heat balls by adjoint heat balls in the definitions above.
\end{definition}

{A subset $E$ of $\rrn$ is (geometrically) porous if for any ball $B$ centered on $E$ with radius small enough, there exists a ball of comparable radius that is contained in $B \cap E^c$. Porosity of a  domain $\om \subset \rrn$  is directly related to boundary regularity of solutions of elliptic and prabolic PDE defined in $\om$ as it quantifies the non-degeneracy of the exterior of the domain. There are different notions of scale-invariant non-degeracy of the exterior of $\om$ such as the geometric porosity of $\Omega$ defined above, (Lebesgue) measure porosity or capacitary porosity of $\Omega$. This is the reason why we use  the term ``porosity''. For instance, we refer to \cite[Corollary 11.25]{BB11}.}

Note that since $\om^+$ and $\om^-$ are disjoint, it holds that $\mathcal{P} \om^\pm \cap \mathcal{B} \om^\mp=\mathcal{B} \om^+ \cap \mathcal{B} \om^- =  \emptyset$ and thus, we could have assumed  $\mathcal{P} \om^+\cap \mathcal{P} \om^- \neq  \emptyset$ in the definition above.
For $\Omega\subset \Rn1$, we denote
\begin{align*}
T_{\min}&= \inf\{T \in \R : \Omega \cap \{t=T\} \neq \emptyset \}, \\
T_{\max}&= \sup\{T \in \R : \Omega \cap \{t=T\} \neq \emptyset \}.
\end{align*}
 Note that it is possible that $T_{\min}=-\infty$ and $T_{\max}=+\infty$. We also define 
$$
E_s=\{(x,t)\in E: t=s\}
$$
 to be the \textit{time-slice} of $E$ for $t=s$.

\begin{definition}
We say that  $\Omega$ satisfies the \textit{Time-Backwards Capacity Density Condition} (TBCDC) at $\bar \xi_0=(\xi_0,t_0)\in\mathcal S\Omega$ if there exists $\wt c>0$ such that
\begin{equation}\label{eq:HCDCintro}
\Cap\bigl(\overline{E(\bxi_0; r^2)}\cap \Omega^c\bigr) \geq \wt c\,r^n, \quad \textup{for all}\,\, 0<r<  \sqrt{{t_0-T_{\min}}/{2}}.
\end{equation}
Analogously, we say that $\Omega$ satisfies the \textit{Time-Forwards Capacity Density Condition} (TFCDC) if, for sime $\wt c>0$,
\begin{equation}\label{eq:HCDCintro_forw}
\Cap\bigl(\overline{E^*(\bxi_0; r^2)}\cap \Omega^c\bigr) \geq \wt c\,r^n, \quad \textup{for all}\,\, 0<r<  \sqrt{{T_{\max}-t_0}/{2}}.
\end{equation}
\end{definition}

\vv

The joint TBCPC  guarantees regularity of the parabolic boundary for the heat equation and is  a much weaker condition than  TBCDC, which, in turn,  is a  natural setting  for our problems (see e.g. \cite{AMT16}). 
The joint TBCPC assumption is a sufficient condition for our blow-up arguments to work and, at the moment, it is not clear how to remove it from our hypotheses. In the case of  harmonic and elliptic measures, it has been proved in \cite[Lemma 5.3]{AMTV16} and \cite[Lemma 4.13]{AM19} that mutual absolute continuity of the interior and the exterior measures implies an even stronger version of \eqref{eq:CPC} in terms of the $(n+1)$-dimensional Lebesgue measure of euclidean balls. This method, however, is based on a dichotomy argument and cannot be generalized directly to the caloric measure because it is not enough to obtain that the exterior of the domain is ``large" in a sequence of  parabolic balls; we need to know that this is true in a sequence of  time-backwards cylinders, which complicates things  significantly. 

Criteria for \textit{parabolic Wiener regularity} have been extensively studied in the literature (see Section \ref{sec:potential theory} for the definition and more references). This property is particularly important to our purposes because we need to extend  the Green function    by zero  to the complement of the domain in a continuous way. In order to work in more generals domains, the intrinsic difficulties of heat potential theory constitute a challenging obstacle to overcome. For instance, we remark that Harnack inequality is time-directed (see e.g. \cite[Corollary 1.33]{Watson}) and that it is possible to construct domains so that the capacity of the set of the irregular points for the heat equation is positive (see \cite[p.336]{TW85}). The most general class of domains in which (some of) our arguments work seem to be the so-called \textit{quasi-regular} domains for the heat operator (resp. adjoint heat operator), which amounts to domains for which the set of irregular  points of the essential boundary for the heat equation (resp. adjoint heat equation) is \text{polar} (see Section \ref{sec:potential theory}). This is in accordance to the elliptic theory where, indeed, the set of irregular points is always polar (which is not the case in the parabolic context).

\vv

Before we state our first main theorem, which is a generalization of \cite{KPT09} and \cite[Theorem I]{AM19}, we need the notion of parabolic Hausdorff dimension of a Borel measure $\omega$. This is defined as
\[
\dim_{\mathcal H_p}(\omega)\coloneqq \inf\bigl\{\dim_{\mathcal H_p}(Z):\, \omega(\Rn1\setminus Z)=0\bigr\},
\]
where ``$\dim_{\mathcal H_p}$'' stands for the parabolic Hausdorff dimension (see Section \ref{sec:preliminaries}), and it can be linked to a pointwise notion (see Section \ref{sec:GMT}). Moreover, given $A\subset\partial \Omega$ we denote by $A^\circ$ the relative interior $A$ with respect to the topology of $\partial \Omega$.

\vv

\begin{main}\label{theorem:blowup1}
Let $\Omega^+$ and $\Omega^-$ be two disjoint domains in $\Rn1$ which are quasi-regular for $H$ and regular for $H^*$, such that $\bigl(\mathcal P^*\Omega^\pm\bigr)^\circ\neq \varnothing$, and let $\omega^\pm$ be the caloric measures associated with $\Omega^\pm$ with poles $\bar p_\pm \in \om^\pm$.
Let also $E\subset\bigl(\mathcal P^*\Omega^+\bigr)^\circ\cap \bigl(\mathcal P^* \Omega^-\bigr)^\circ\cap \supp \omega^+$ be such that  $\hm^+(E)>0$ and $\omega^+|_E\ll\omega^-|_E\ll\omega^+|_E$, and assume that $\Omega^+$ and $\Omega^-$ satisfy the joint TBCPC at all points of $E$. Then, 
$\Tan(\omega^\pm,\bar \xi)\subset \mathscr F$ for $\omega^+$-a.e. $\bar \xi \in E$ and ${\dim_{\mathcal H_p}}\,  \omega^\pm|_E=n+1$.

Moreover, if $\Omega^\pm$ also  satisfy the TFCDC, 
\begin{equation}\label{eq:theta-main1}
\lim_{r\to 0}\Theta^{\mathscr F_\Sigma}_{\partial \Omega^\pm}(\bar \xi, r)=0\, \, \text{for }\omega^+\text{-a.e. }\bar \xi\in E.
\end{equation}
\end{main}

\vv

We remark that if $\partial^*_a\Omega^\pm\cap \mathcal S^*\Omega^\pm=\varnothing$, then $ (\mathcal P^*\Omega^\pm)^0= \mathcal P^*\Omega^\pm$. Any domain that satisfies the TFCDC has this property (see Remark \ref{rem:cdc}).

The theorem above (and similarly all the theorems we prove) can be formulated for adjoint caloric measures if we assume the joint TFCPC and the TBCDC in place of the joint TBCPC and TFCDC and let $E$ be in $\bigl(\mathcal P\Omega^+\bigr)^\circ\cap \bigl(\mathcal P \Omega^-\bigr)^\circ$. We point out that caloric measure is not necessarily doubling in domains such as the ones considered  in Theorem \ref{theorem:blowup1} and also that it is natural to study mutual absolute continuity of $\omega^+$ and $\omega^-$ on the lateral part of the boundary because it cannot occur elsewhere. For more details on this matter, we refer to Section \ref{sec9}. 

\vv

Secondly, we prove the caloric equivalent of Tsirelson's theorem  \cite{Ts97} about triple-points for harmonic measure. Tsirelson's method is based on the fine analysis of filtrations for Brownian and Walsh-Brownian motions, although,  more recently, Tolsa and Volberg \cite{TV18} obtained the same result using analytical tools and, in particular, blow-up arguments, which is exactly the method we follow as well.

\begin{main}\label{theorem:tsirelson}
Let $\Omega_i \subset \rrn$, $1 \leq i \leq 3$, be three disjoint open sets which are quasi-regular for both $H$ and $H^*$ and such that $\bigl(\mathcal P^*\Omega^i)^\circ\neq \varnothing$, $i=1,2,3$. Let  also $\omega^i$ be the caloric measure in  $\Omega^i$ with pole at $\bar p_i \in \Omega^i$   and  assume that $E\subset\mathcal \cap_{i=1}^3 \bigl(\mathcal P^*\Omega_i\bigr)^\circ\cap   \supp \omega^1 $ is  such that  $\omega^i|_E\ll\omega^j|_E\ll\omega^i|_E$, for $1 \leq i,j \leq 3$. Then $\hm^i(E)=0$ for $i=1,2,3$. 
\end{main}

Note  that Theorem \ref{theorem:tsirelson} needs neither the joint TBCPC assumption on the domains nor regularity for $H^*$. Quasi-regularity is all we need to assume since, by  an interior approximation argument that we prove in Section \ref{sec:potential theory}, we show that it suffices to  study the problem in  regular domains.

\vv

Under a pointwise $\VMO$-type condition on  $d\omega^-/d\omega^+$ on a particular subset of the boundary $E$, we prove a local analogue of \cite[Theorem II]{AM19}.

\begin{main}\label{theorem:main_theorem2}
Let $\Omega^+$ and $\Omega^-$ be two disjoint domains in $\Rn1$ which are quasi-regular for $H$ and regular for $H^*$, such that $\bigl(\mathcal P^*\Omega^\pm)^\circ\neq \varnothing$, and let $\omega^\pm$ be the caloric measures associated with $\Omega^\pm$ with poles $\bar p_\pm \in \om^\pm$. 
Let also $E\subset\bigl(\mathcal P^*\Omega^+\bigr)^\circ\cap \bigl(\mathcal P^* \Omega^-\bigr)^\circ\cap \supp \omega^+$ be a relatively open set in $\supp \omega^+$ such that  $\hm^+(E)>0$ and $\omega^-|_E\ll\omega^+|_E$. Assume that $\Omega^+$ and $\Omega^-$ satisfy the joint TBCPC at all points of $E$ and, if we set $f=\tfrac{d\omega^-|_E}{d\omega^+|_E}$ to be the Radon-Nikodym derivative of $\omega^-|_E$ with respect to $\omega^+|_E$, we assume that for $\bar \xi \in E$, 
\begin{equation}\label{eq:main2_vmo}
\lim_{r\to 0}\biggl(\avint_{C_r(\bar \xi)} f\,d\omega^+|_E\biggr)\exp\biggl(-\avint_{C_r(\bar \xi)} \log f\,d\omega^+|_E\biggr)=1,
\end{equation}
and $\Tan(\omega^+,\bar \xi)\neq\varnothing$, then there is $k\geq 1$ such that $\Tan(\omega^+,\bar \xi)\subset\mathscr F(k)$ and
\[
\limsup_{r\to 0}\frac{\omega^+(C_{2r}(\bar \xi))}{\omega
^+(C_r(\bar \xi))}<\infty.
\]
Furthermore, if $\Omega^+$ and $\Omega^-$  also satisfy the TFCDC, then
\begin{equation}\label{eq:theta-main2}
\Theta^{\mathscr F_\Sigma(k)}_{\partial\Omega^\pm}(\bar \xi, r)\to 0,\qquad \text{as } r\to 0.
\end{equation}
\end{main}

\vvv


Before we go any further, let us introduce the space of Vanishing Mean Oscillation. Given a Radon measure $\omega$ in $\Rn1$, we denote
\[
f_{C_r(\bar x)}\coloneqq\avint_{C_r(\bar x)}f\, d\omega\coloneqq\frac{1}{\omega(C_r(\bar x))}\int_{C_r(\bar x)} f\, d\omega,\qquad \bar x\in \supp\omega,
\]
and we say that $f\in \VMO(\omega)$ if $f\in L^1_{\loc}(\omega
)$, and 
\[
\lim_{r\to 0}\sup_{\bar x\in\supp\omega}\avint_{C_r(\bar x)}\bigl|f(\bar y)-f_{C_r(\bar x)}\bigr|\,d\omega(\bar y)=0.
\]

The condition \eqref{eq:main2_vmo} implies a vanishing mean oscillation assumption on $d\omega^-/d\omega^+$ on $E$. However, in general, these assumptions are not equivalent. For more details, we refer to \cite[Section 7]{AM19}. 

\vv

The next theorem is a global version of Theorem \ref{theorem:main_theorem2} under the additional assumption that  $\omega^\pm$ is doubling. This is the analogue of \cite[Theorem III]{AM19}, which,  in turn generalized, the main results of  \cite{KT06} and \cite{Bad11}. We prove that at sufficiently small scales, the support of $\omega^+$ resembles the zero set of an adjoint caloric polynomial  uniformly on compact subsets of the set of mutual absolute continuity. This assertion can be formulated both in terms of $\Theta^{\mathscr P_\Sigma(d)}_{\partial \Omega^\pm}$ and  the functional $d_1(\cdot, \mathscr P(d))$, which is essentially a distance between measures and the set  $\mathscr P(d)$ (see \eqref{eq:def_dist_dcones} for the exact definition).

\begin{main}\label{theorem:main3}
Let $\Omega^+,\Omega^-, \hm^+, \hm^-, f$ and $E$ be as in the statement of Theorem \ref{theorem:main_theorem2}. If $\log f \in \VMO(\omega^{+}|_E)$ and $\omega^{+}|_E$ is $C$-doubling, i.e.
\[
\omega^+|_E(C_{2r}(\bar x))\leq C \omega^+|_E(C_r(\bar x)),\qquad \bar x\in\Rn1, r> 0,
\]
 then, there exists an integer $d=d(n,C)>0$, so that for every compact set $F \subset E$, it holds
\begin{equation}\label{eq:d1_lim_r_to_0}
\lim_{r\rightarrow 0} \sup_{\bar \xi\in F} \, d_{1}\bigl(T_{\bar \xi,r}[\omega^{+}],\mathscr{P}(d)\bigr) =0.
\end{equation}
If  $\Omega^+$ and $\Omega^-$ also satisfy the TFCDC, then
\begin{equation}\label{eq:beta_tilde_d}
\lim_{r\to 0}\sup_{\bar \xi\in F}\, \Theta^{\mathscr P_\Sigma(d)}_{\partial \Omega^\pm}(\bar \xi, r)=0.
\end{equation}
That is, $E \in LBWA(\mathscr P_\Sigma(d))$.
\end{main}


\vv

Theorem \ref{theorem:main3}  also applies directly to the study of (parabolic) Reifenberg flat domains, and gives an alternative proof of a result that Nystr\"om proved with different techniques in \cite{Nys06C}.
We recall that $\Omega\subset \Rn1$ is $\delta$-\textit{Reifenberg flat}, $\delta>0$, if for $R>0$ and $\bar \xi\in \partial \Omega$ there exists a $n$-plane $L_{\bar \xi, R}$ through $\bar \xi$ containing a line parallel to the time-axis and such that
\begin{align*}
&\bigl\{\bar y+ r \hat n\in C_R(\bar \xi): r>\delta R, \, \bar y\in L_{\bar \xi, R}\bigr\}\subset \Omega,\\
&\bigl\{\bar y- r \hat n\in C_R(\bar \xi): r>\delta R, \, \bar y\in L_{\bar \xi, R}\bigr\}\subset \Rn1 \setminus \overline{\Omega},
\end{align*}
where $\hat n$ is the normal vector to $L_{\bar \xi, R}$ pointing into $\Omega$ at $\bar \xi$.

For $t_0\in \mathbb R$, we denote $\Omega^{t_0}\coloneqq \Omega\cap \{t<t_0\}.$
The domain $\Omega$ (resp. $\Omega^{t_0}$) is \textit{vanishing Reifenberg flat} if it is $\delta$-Reifenberg flat for some $\delta>0$ and, for any compact set $K$ (resp. $K\subset\subset \{t<t_0\}$),
\[
\lim_{r\to 0} \sup_{\bar \xi\in K\cap \partial \Omega}\Theta^{\mathscr F_\Sigma}_{\partial \Omega} (\bar \xi, r)=0.
\]
Let us also observe that, if $\Omega$ is $\delta$-Reifenberg flat, it readily follows that $\partial_a\Omega=\partial^*_a\Omega=\varnothing$. Additionally, it is  easy to see that a $\delta$-Reifenberg flat domain satisfies the TBCDC and the TFCDC (with parameters depending on $\delta$) and, thus, is  regular for $H$ and $H^*$.

\begin{corollary}\label{cor:reifenberg}
Let also $\Omega^+\subset \Rn1$ be a $\delta$-Reifenberg flat domain and set $\Omega^-=\Rn1\setminus \overline{\Omega^+}$. Let $\omega^{\pm}$ be the caloric measures of $\Omega^{\pm}$ with poles at $\bar p_\pm=(p_\pm, t_\pm)\in \Omega^{\pm}$. If $\omega^-\ll\omega^+$, $\log f=\log \tfrac{d\omega^{-}}{d\omega^{+}}\in \VMO(\omega^{+})$,  and  $\delta$ is small enough (depending on $n$), then $(\Omega^+)^{t_{-}}$ is vanishing Reifenberg flat.
\end{corollary}
The analogous result for  harmonic measure can be found in  \cite[Corollary 4.1]{KT06}, while for the original theorem for  caloric measure we refer to  \cite[Theorem 1]{Nys06C}.

\vv

\subsection*{Discussion of the proofs} In the current paper, we follow the same strategy as in \cite{AM19}, although there are many significant challenges along the way that we   overcome to  adapt this method successfully to the parabolic setting. In Section \ref{sec:potential theory}, we exhibit various properties of thermal  capacity, the most important of which is the backwards in time self-improvement of a pointwise time-backwards capacity density condition on a particular scale. This is the building block of the proofs of several  PDE estimates around the boundary that are absolutely necessary for our blow-up arguments in Section \ref{sec8}. Those results are new and interesting in their own right, and, judging from their elliptic analogues,  we expect them to  be used widely  in future works. In Section \ref{sec:GMT}, we develop the required parabolic GMT framework and  confirm that, due to a parabolic Besicovitch covering theorem, the theory of tangent measures translates almost unchanged to the parabolic setting. Moreover, we show that  the blow-ups of the ``parabolic surface measure" on euclidean rectifiable sets are parabolic flat, meaning that there is a plane that contains a line parallel to the time axis such that the blow-up measure is the parabolic Hausdorff measure restricted to that plane.  Han and Lin \cite{HL94} had proven that if $h$ is caloric, then in any ball centered on $\Sigma^h=\{h=0\}$,  it holds that $\Sigma^h$ can be written as the union of its regular set, which is a smooth $n$-submanifold,  and its singular set, which is an $(n-1)$-rectifiable set. Unlike the harmonic case, this is not enough for our purposes  and so, in Section \ref{sec:nodal_sets}, we investigate the finer structure of the regular set separating space and time-singularities.    We demonstrate that  any  $\bx \in \mathcal R_x=\Sigma^h \cap \{|\nabla h|\neq 0\}$ has a neighborhood so that $\Sigma^h$ is given by an {\it admissible} $n$-dimensional smooth  graph, while  any $\bx \in \mathcal R_t =\Sigma^h \cap \{|\nabla h|=0\} \cap \{|\dt h| \neq 0\}$,  has a neighborhood in which  $\mathcal R_t$ is contained in a smooth $(n-1)$-graph.  Those results are fundamental to several  measure-theoretic arguments and are used  repeatedly in Sections  \ref{sec:green functions}-\ref{sec8}. To the best of our knowledge, this description of the time-regular set  is novel in the literature. In Section \ref{sec:green functions}, we prove the existence and uniqueness of  the caloric measure associated with a caloric function of the form $d \hm_h = - \d_{\nu_t}h d \sigma_h$, where $\sigma_h$ is the ``surface measure" on $\Sigma^h$. In the elliptic setting, this is a straightforward application of the Gauss-Green formula on sets of locally finite perimeter, which is far from being the case here. The aim of {Section \ref{sec:caloric_polynomials}} is to prove that, if the parabolic tangent measures $\Tan(\omega,\bar \xi)$ to a given Radon measure $\omega$ are caloric measures associated with caloric polynomials and there is a measure corresponding to a \textit{homogeneous} caloric polynomial of degree $k$, then all the elements of $\Tan(\omega,\bar \xi)$ are caloric measures associated with a homogeneous caloric polynomial of the same degree. The connectivity arguments for tangent measures  translate unchanged from \cite{AM19}, although, this is not true for the analogues of the main lemmas from \cite{Bad11}. In fact, it is not clear to us how to adapt Badger's proofs directly to our case, although, his ideas  inspired us to come up with new ones, which we find simpler even for harmonic functions.  In Section \ref{sec:optimal_transport}, we introduce some basic notions of Optimal Transport. We consider the  space of signed Borel measures in a  bounded open subset of $\rn$ equipped with the adapted Kantorovich-Rubinshtein  norm and prove that its Banach dual can be identified with the space $\lipp$ of $1$-Lispchitz functions on $\oom$ vanishing at the boundary. We  also identify its completion as a subspace of the space of first order distributions in the dual of $\lipp$ and show that $L^2$ functions are embedded continuously in that completion. As we deal with signed measures so that $\mu(\om)$ is not necessarily zero, the results are less standard and not broadly used in the literature and so we have to adapt the existing proofs to our setting. Although, those modifications are not trivial and require  some care.  This functional analytic approach is novel in our area of research and it is an essential component of a compactness argument for the blow-ups of the Green's function. Section \ref{sec8}  groups the main one- and two-phase blow-up lemmas, for which we follow the approach in \cite{AMT16}, \cite{ AMTV16}, and \cite{AM19}. However, there are substantial obstacles such as the compact embedding of a ``parabolic" Sobolev space  in $L^2$, the lack of analyticity in time, the unique continuation arguments, the global H\"older continuity of the rescaled Green functions, and the locality of Lemma \ref{lem:blow_up_CDC_two_phase}. We highlight that for the blow-up arguments to work, we prove a novel boundary Caccioppoli-type inequality for the $t$-derivative of the square of a positive caloric function vanishing on a ``boundary ball" and introduce a new mixed-norm Sobolev space based on the aforementioned space of finite signed measures with the Kantorovich-Rubinshtein norm. Finally, in  {Section \ref{sec9}} we provide the proofs of the four main theorems  using the results from  the previous sections.

\vvv

\subsection*{Acknowledgements}

We would like to thank Jonas Azzam, Max Engelstein, Luis Escauriaza, Steve Hofmann, Tuomo Kuusi, Andrea Merlo, and Xavier Tolsa for several fruitful discussions on topics related to the current paper. The first named author would like to dedicate this paper to the memory of his teacher, mentor, and dear friend, Professor Michel Marias.

\vvv

\section{Preliminaries and notation}\label{sec:preliminaries}

We denote points in  $\rrn$ by 
$\bx=(x,t)=(x',x_n,t)\coloneqq(x_1, \dots, x_n ,t)$ and define the (parabolic) norm
\[
\|\bx\| \coloneqq \max \{|x|, |t|^{1/2}\},
\]
where $| \bx|= \sqrt{ x_1^2+ \dots+x_n^2}$ stands for the standard euclidean norm. We also use the notation $\bar 0=(0,\ldots,0)$.

 If $E \subset \rrn$, we define the {\it parabolic distance} to $E$ as $\dist_p(\bx,E)= \inf\{ \|\bx-\by\|: \by \in E\}$  and the {\it parabolic diameter} as $\diam_p(E)= \sup\{\|\bx-\by\|: (\bx ,\by) \in E \times E\}$.

Given an open set $\om$, and a point $\bxi \in \om$, we denote by $\Lambda(\bxi, \om)$ (resp. $\Lambda^*(\bxi, \om)$) the set of points $\bx \in \om $ that are lower (resp. higher) than $\bxi$ relative to $\d\om$, in the sense that there is a polygonal path $\gamma \subset \om$ joining $\bxi$ to $\bx$, along which the time variable is strictly decreasing (resp. increasing). By a polygonal path, we mean a path which is a union of finitely many line segments.

Following  \cite[Definition 8.1]{Watson}, we define the {\it normal boundary} 
$$
\d_n \om =
\begin{cases}
\mathcal{P}\om    &,\mbox{if} \,\om\,\, \mbox{is bounded}\\
\mathcal{P}\om \cup \{\infty\}   &,\mbox{if}\, \om\,\, \mbox{is unbounded}
\end{cases}
$$
and the {\it abnormal boundary} $
\d_a \om = \left\{\bx \in \d\om : \exists\,\, \ve>0 \,\, \textup{such that}\,\, C^-_\ve( \bx) \subset  \om \right\}.
$
The abnormal boundary is further decomposed into $\d_a \om=\d_s \om \cup \d_{ss} \om$, where $\d_s \om$ stands for the {\it singular boundary} and $\d_{ss} \om$ for the {\it semi-singular boundary}, which are defined respectively by
$$
\d_s \om =\left\{ \bx \in \d_a\om:  \exists\,\, \ve>0 \,\, \textup{such that}\,\, C^+_\ve( \bx) \cap \om= \emptyset \right\}
$$
and 
$$
\d_{ss} \om=\left\{ \bx \in\d_a\om: C^+_r( \bx) \cap  \om \neq \emptyset, \,\,\textup{for all}\,\, r>0 \right\}.
$$
By \cite[Theorem 8.40]{Watson}, $\d_a \om$ is contained in a sequence of hyperplanes of the form $\R^n \times \{t\}$. The {\it essential boundary} is defined as 
$
\d_e\om= \d_n \om \cup \d_{ss}\om = \d \om \setminus \d_{s}\om,
$
replacing $\d \om$ by $\d \om \cup \{\infty\}$ if $\om$ is unbounded.  Finally, following \cite{GH20}, we define the {\it quasi-lateral boundary}  to be  
 $$
\mathcal{S}\om=
\begin{cases}
\d\om    &,\mbox{if} \,\, T_{\min}=-\infty \,\, \textup{and}\,\,T_{\max}=\infty\\
\d\om \setminus (\mathcal{B}\om)_{T_{\min}}    &,\mbox{if} \,\, T_{\min}>-\infty \,\, \textup{and}\,\,T_{\max}=\infty\\
\d\om  \setminus (\d_s \om)_{T_{\max}}     &,\mbox{if} \,\, T_{\min}=-\infty \,\, \textup{and}\,\,T_{\max}<\infty\\
\d\om \setminus \left( (\mathcal{B}\om)_{T_{\min}}  \cup (\d_s \om)_{T_{\max}}\right)   &,\mbox{if} \,\, -\infty<T_{\min} <T_{\max}<\infty,
\end{cases}
$$
where $(\mathcal{B}\om)_{T_{\min}}$ is the time-slice of $\mathcal{B}\om$ with $t=T_{\min}$ and $(\d_s \om)_{T_{\max}}  $ is the time-slice of $\d_s\om$ with $t=T_{\max}$. Observe that for a cylindrical domain $U \times (T_{\min}, T_{\max})$, where $U\subset \R^n$ is a domain in the spatial variables, the quasi-lateral boundary coincides with the lateral boundary. By \cite[Lemma 1.14]{GH20},  both $\d_e \om$ and $\mathcal S\Omega$ are closed sets.

\vv
Given $f\colon\Rn1\to\R$ we write $D f =(\nabla f, \partial_t f)$ for the (full) gradient of the function $f$ and, $D^{\alpha, \ell}f=\partial^{\alpha_1}_{x_1}\cdots\partial^{\alpha_n}_{x_n}\partial^\ell_tf $ for higher order derivatives, where $\alpha\in\mathbb Z^n_+$ and $\ell\in\mathbb Z_+$.  If $E\subset \rrn$ and $f$ is a continuous function with compact support in $E \subset \rrn$, then we write $f \in C_c(E)$. If $\om$ is an open set, we  denote by $C^{m,\frac{m}{2}}(\om)$ the class of functions such that  $f(\cdot, t) \in C^m(\om_t)$ for any fixed $t \in (T_{\min}, T_{\max} )$ and $f(x, \cdot) \in C^{\frac{ m}{2}}(T_{\min}, T_{\max} )$  for any fixed $\bx \in \rrn$ such that $\bx\in \om$. If $f$ is $C^m$ in both space and time variables, we will simply write that $f \in C^m(\om)$. Finally, we say that  $f \in C_c^{m,\frac{ m}{2}}(E)$ if $f$ has compact support in $E \subset \rrn$ and $f \in C^{m,\frac{ m}{2}}(\rrn)$. 

\vv

We write $a\lesssim b$ if there is $C>0$ so that $a\leq Cb$ and $a\lesssim_{t} b$ if the constant $C$ depends on the parameter $t$. We write $a\approx b$ to mean $a\lesssim b\lesssim a$ and define $a\approx_{t}b$ similarly. Sometimes we also use the notation $\avint_{F} \,d\mu$ for the average $\mu(F)^{-1}\int_{F} \,d\mu$ over a set $F \subset \R^{n+1}$ with respect to the measure $\mu$.

\vv

If  $E \subset\rrn$ is a Borel set and $\cH^d$ stands for the Euclidean $d$-Hausdorff measure in $\rrn$ for $d \leq n-1$, we define the $d$-``\textit{surface measure}'' on $E$ as the measure $d\sigma=d\sigma_t dt$, where $\sigma_t=\cH^{d}|_{E_t}.$

If $s \in [2, n+2]$ and $0<\delta < \infty$, we  set 
\[
\cH^s_{p, \delta}(E)= \inf \Bigl\{\sum \diam_p (E_j)^s: E\subset \bigcup_j E_j, \diam_p(E_i)\leq \delta\Bigr\},
\]
for $E \subset \rrn$, and, as in the Euclidean case, define the {\it  parabolic $s$-Hausdorff measure} by
\[
\cH^s_{p} (A)=\lim_{\delta\to 0}\cH^s_{p,\delta}(E),
\]
which is a Borel measure by the Carath\'eodory criterion. We also define the {\it parabolic Hausdorff content} of $E$ by
\[
\cH^s_{p, \infty}(E)= \inf \Bigl\{\sum \diam_p (E_j)^s: E\subset \bigcup_j E_j\Bigr\}.
\]
We say that the {\it parabolic Hausdorff dimension} of $E$ is
\[
\dim_{\mathcal H_p}(E)=\sup \{s:\mathcal H^s_p (E)>0\}=\inf \{t:\mathcal H_p^t(E)=0\}.
\]

\vvv

\section{Heat Potential Theory and PDE estimates}\label{sec:potential theory}

Let  $\om \subset \rrn$ be an open set. Define the parabolic operator
\begin{equation}\label{eq:Heat operator}
H_a u := a \Delta  - \dt, \quad \textup{and }\,\, H^*_a u := a \Delta  + \dt, \quad \textup{for}\,\,a>0.
\end{equation}
When $a=1$, we simply write $H_1=H$ and $H_1^*=H^*$ for the {\it heat} and the {\it adjoint heat operator} respectively.

For $\bx =(x,t) \in \rrn$, we denote by 
\[
\Gamma_a(\bx)=(4\pi a t)^{-n/2}\exp\bigl(-|x|^2/4a t\bigr)\chi_{\{t> 0\}}(t)
\]
the {Gaussian kernel}. Note that $\Gamma_a$ is the  {\it fundamental solution associated with $H_a$}, i.e., it satisfies 
\begin{enumerate}
\item $H_a \Gamma_a(x,t) = \delta_{\bar 0}(x,t)$, in the distributional sense, for $\{t>0\}$, where $\delta_{\bar 0}$ stands for the Dirac mass at $\bar 0$;
\item $\Gamma_a(x,t) = \delta_{0}(x)$, for $t=0$;
\end{enumerate}

Moreover, $\Gamma_a \in C^\infty(\rrn \setminus \{\bar 0\})$ and  
\begin{equation}\label{eq:poitwiseHeat kernel}
\Gamma_a(\bx) \leq C_h  \pi^{-n/2}  \|\bx\|^{-n},
\end{equation}
for some  constant  $C_h>0$ depending on $n$ and $a$.

 We say that a  function $u \in C^{2,1}(\om)$ is {\it caloric}  (resp. {\it adjoint caloric}) if it satisfies the ({\it adjoint}) {\it heat equation} $Hu =0$ (resp. $H^*u=0$) 
 in a pointwise sense. If $u \in C^{2,1}(\om)$ and $H u \geq 0$ (resp.  $H u \leq 0$) in $\om$, we  say that  $u$ is {\it  subcaloric} (resp. {\it   supercaloric}). 
 
A distribution $u \in \mathcal{D}'(\om)$ is called  {\it weakly} {\it caloric} (resp {\it weakly subcaloric} or {\it weakly supercaloric}), if it satisfies
 $$
 \int_\om u \,H^*\vphi = 0, \quad (\textup{resp.} \,\geq 0 \,\, \textup{or} \, \,\leq 0),
 $$
 for all $\vphi \in  C^{2,1}(\om)$ (resp. $0 \leq \vphi \in  C^{2,1}(\om)$). If $u$ is a {\it supertemperature} (resp. {\it subtemperature}), then it is  weakly supercaloric (resp. weakly subcaloric) (see \cite[Definition 3.7, Theorem 6.28]{Watson}). Moreover, by  \cite[Theorem 6.28]{Watson}, if $u$ is a subcaloric in $\om$, then there exists a unique non-negative Radon measure $\mu_u$, which is called the {\it Riesz measure associated with} $u$, such that 
\begin{equation}\label{eq:Riesz measure}
 \int_\om u \,H^*\vphi = \int_\om \vphi \,d\mu_u, \quad \textup{for all}\,\,\vphi \in  C^{2,1}(\om).
\end{equation}
By making the obvious adjustments  we obtain the dual definitions for $H^*$.

 \vv

 \begin{lemma}\label{lem:weak_caloric_functions}
If $\om \subset \rrn$ is an open set, then the following are equivalent:
 \begin{enumerate}
 \item $u \in C^{2,1}(\om)$ is caloric in $\om$,
 \item $u \in \mathcal{D}'(\om)$ is weakly caloric in $\om$,
 \item $u \in C^\infty(\om)$ and $Hu=0$ in $\om$.
 \end{enumerate}
 \end{lemma}
 
 \begin{proof}
That (1) implies (2)  follows from integration by parts, while (3) implies (1) is trivial. Remark that the heat operator is hypoelliptic (see \cite[15.6, p. 210]{Vl02}) since the heat kernel is the fundamental solution for the heat equation. Thus, the remaining implication is true because of \cite[Theorem 1, p. 210]{Vl02}.
 \end{proof}

\vv

Let  $f$ be an extended real-valued function defined in $\d_e \om$.  The \textit{upper class} $\mathfrak{U}_f $ of $f$ consists of lower bounded hypertemperatures $w$ in $\Omega$  (see \cite[Definition 3.10]{Watson}) such that $w\geq f$ on $\d_e \Omega$ in the sense that 
it holds 
$$
\liminf_{(y,s) \to (\xi,t)} w(y,s) \geq  f(\xi,t), \quad \textup{for all}\,\, ( \xi,t) \in \d_n\om,
$$
and
$$
\liminf_{(y,s) \to (\xi,t^+)} w(y,s) \geq  f(\xi,t), \quad \textup{for all}\,\, (\xi,t) \in \d_{ss}\om.
$$ 
 The \textit{lower class} $\mathfrak{L}_f $ of $f$ consists of upper bounded hypotemperatures $v$ in $\Omega$  such that $v\leq f$ on $\d_e \Omega$ in the sense that 
it holds 
$$
\limsup_{(y,s) \to (\xi,t)} v(y,s) \leq  f(\xi,t), \quad \textup{for all}\,\, ( \xi,t) \in \d_n\om,
$$
and
$$
\limsup_{(y,s) \to (\xi,t^+)} v(y,s) \leq  f(\xi,t), \quad \textup{for all}\,\, (\xi,t) \in \d_{ss}\om.
$$

The function $U_f(x) = \inf\{w \in \om: w \in \mathfrak{U}_f  \}$ is called the {\it upper solution} for $f$ in $\om$, and  $L_f(x) = \sup\{v \in \om: v \in \mathfrak{L}_f  \}$ is called the {\it lower solution} for $f$ in $\om$. We say that $f$ is {\it resolutive} if $U_f=L_f\eqqcolon S_f$ and $S_f$ is caloric in $\om$. In this case, $S_f$ is called the PWB {\it solution} for $f$ in $\om$. By \cite[Theorem 8.26]{Watson}, any $f \in C(\d_e \om)$ is resolutive. Therefore, for any $\bar x \in \om$, the  map $f \mapsto u_f(\bar x)$ is linear and, by the Riesz representation theorem, there exists a unique probability measure $\hm^{\bar x}$ on $\d_e\om$, which is called {\it caloric measure}, such that
$$
u_f(\bx) = \int_{\d_e \om} f\,d\hm^{\bar x}.
$$
More generally, if $f$ is an extended real-valued function and $\bx \in \om$ and  $\int_{\d_e \om} f \,d\hm^{\bx}$ exists, then 
\begin{equation}\label{eq:Uf=Lf=caloric}
U_f(\bx)=L_f(\bx)=\int_{\d_e \om} f \,d\hm^{\bx}.
\end{equation}
Conversely, if $U_f(\bx)=L_f(\bx)$ and is finite, then $f$ is $\hm^{\bx}$-integrable and  \eqref{eq:Uf=Lf=caloric} holds (see \cite[Theorem 8.32]{Watson}). In the last statement, if we also assume that $f$ is Borel, then we  obtain that $f$ is resolutive. In particular, if $E \subset \d_e \om$ is $\hm^{\bx}$-measurable for all $\bx \in \om$, then $\chi_E$ is resolutive and $S_{\chi_E}(\bx) = \hm^{\bx}(E)$ (see \cite[Corollary 8.33]{Watson}). Therefore, if $\hm^{\bx}(E)=0$ for some $\bx \in \om$, then, by minimum principle, $\hm^{\by}(E)=0$ for all $\by \in \Lambda(\bx, \om)$.

\vv

 If $f \in C(\d_e\om)$, we say that  $u$ is a  {\it solution of the classical Dirichlet problem with data} $f$ if $u$ is caloric and it holds that
\begin{equation}\label{eq:limit-normalbdry}
\lim_{(y,s) \to (\xi,t)} u(y,s) = f(\xi,t), \quad \textup{for all}\,\, ( \xi,t) \in \d_n\om,
\end{equation}
and
\begin{equation}\label{eq:limit-ss-bdry}
\lim_{(y,s) \to (\xi,t^+)} u(y,s) = f(\xi,t), \quad \textup{for all}\,\, (\xi,t) \in \d_{ss}\om.
\end{equation}
If there exists a solution to the classical Dirichlet problem  in $\om$ with data $f$, then $f$ is resolutive and $u_f$ is the PWB solution for $f$ in $\om$ (see \cite[Theorem 8.26]{Watson}).

\begin{definition}\label{def:regular}
A point $\bar \xi_0  \in \d_n \om$ (resp. $\d_{ss} \om$) is called {\it regular} for $H$  or $H$-{\it regular} if for any $f \in C(\d_e \om)$, the PWB solution $S_f$ satisfies \eqref{eq:limit-normalbdry} (resp. \eqref{eq:limit-ss-bdry}). A point $\bar \xi_0  \in \d_e \om$ that is not $H$-regular is called $H$-{\it irregular}. 
\end{definition}

\vv

Let $G_\om(\cdot, \cdot)$  be the non-negative real-valued function  defined in $\om \times \om$ as
$$
G_\om(\bx, \bar p) = \Gamma(\bx- \bar p) - h_\om(\bx, \bar p),
$$
where $ h_\om(\cdot, \bar p)$ is the greatest thermic minorant of $\Gamma(\cdot- \bar p)$, which is non-negative (see \cite[Definition 3.65]{Watson} and the paragraph before this definition). We call $G(\cdot, \bar p)$ the {\it Green function} in $\om$ with pole at $\bar p \in \om$ and, by \cite[Theorem 8.53]{Watson}, if $\psi_{\bar p}=\Gamma(\cdot- \bar p)|_{\d_e \om}$, we also have the representation
$$
G_\om(\bx, \bar p) = \Gamma(\bx - \bar p) -S_{\psi_{\bar p}}(\bx)=\Gamma(\bx - \bar p) -\int_{\d_e \om} \Gamma(\by-\bar p) \, d\hm^{\bx}(\by),
$$
where the last equality follows from the fact that  $\Gamma(\cdot, \bar p)$ is continuous on $\d_e \om$ and thus resolutive.  It is straightforward to see that 
\begin{align}\label{eq:Green-upperbound}
G_\om(\bx, \by) \leq 2C_h \pi^{-n/2}  \|\bx- \by\|^{-n}.
\end{align}
In fact, more is true: if $C_r:=B_r \times (0,r)$ is a cylinder of radius $r>0$, then there exists $c_1>0$ and $c_2>0$ (independent of $r$) such that
\begin{align}\label{eq:Green-two-sidedbound}
c_1 \, m(\bx,\by) \,\Gamma_{c_2}(\bx-\by) \leq G_{C_r}(\bx, \by) \leq c_1^{-1}\, m(\bx,\by)\, \Gamma_{\frac{1}{c_2}}(\bx-\by) ,\quad \textup{for}\,\, \bx \neq \by \in C_r,
\end{align}
where 
$$
m(\bx,\by) := \min\left(1, \frac{\dist(x,B_r)}{\sqrt{t-s}}\right)  \min\left(1, \frac{\dist(y,B_r)}{\sqrt{t-s}}\right).
$$
See e.g. \cite[Theorem 1.2]{Zh02}\footnote{The constants depend on the geometry but not the diameter of the domain. For cylinders, the constants are only dimensional.} and \cite[Theorem 3.3]{Cho06}.

Note that $G_\om(\cdot,\cdot)$ is lower semi-continuous on the diagonal $\{(\bar p, \bar p): p \in  \om\}$ and continuous everywhere else in $\om$. In fact, it is a supertemperature in $\om$ and a temperature in $\om \setminus\{\bar p\}$.  Recall that $\Lambda^*(\bar p, \om)$ is the set of points $\bx \in \om$ for which there is a polygonal path $\gamma \subset \om$ joining  $\bar p$ to $\bx$, along which the time variable is strictly increasing. By \cite[Theorem 6.7]{Watson}, it holds that $G_\om(\cdot, \bar p)>0$ in $\Lambda^*(\bar p, \om)$ and $G_\om(\cdot, \bar p)=0$ in $\om \setminus\Lambda^*(\bar p, \om)$, for any $\bar p \in \om$.  If $\bar \xi \in \d_n \om$ (resp. $\d_{ss}\om$) is a regular point for $H$, we have that $\lim_{\bx \to \bar \xi} G_\om(\bx, \bar p) = 0$ (resp. $\lim_{\bx \to (\xi, t^+)} G_\om(\bx, \bar p) = 0$).  When the domain $\om$ where the Green function is defined is clear from the context, we will drop the subscript and just write $G(\cdot, \cdot)$.

Let us remark that the Riesz measure associated with the Green function $G_\om(\bx, \by)$ in $\om$ is the caloric measure $\hm_\om^{\bx}$ in $\om$ (see \cite{Doob}), i.e.,
\begin{equation}\label{eq:Green-caloric-Riez}
\int G_\om(\bx, \by) H\vphi(\by)\,d\by = \int \vphi \,d\hm^{\bx},\quad \text{for any}\,\, \vphi \in C^\infty_c(\rrn).
\end{equation}

Given a Borel measure $\mu$ on $\Rn1$ and an open set $\om\subset \rrn$, 
\[
G_\om\mu(\bx)\coloneqq \int G_\om(\bx, \by)\, d\mu(\by)
\]
defines a non-negative supertemperature in $\om$ provided the integral is finite in a dense subset of $\om$ and  is called the \textit{heat potential} of $\mu$ in $\om$. If $\om=\rrn$, we replace  $G_\om$ by  $\Gamma$.

\vv

Given a compact set $K\subset \Rn1$, we define its \textit{thermal capacity} as
\begin{equation}\label{def:capacity-compact}
\Cap(K, \om)\coloneqq \sup\bigl\{\mu(K):G_\om\mu\leq 1 \,\,\textup{in}\,\,\om, \mu\geq 0 \textup{ and }\supp\mu\subset K\bigr\}.
\end{equation}
Equivalently, if $u=\widehat{R}_{1}(K,\om)$ is the smoothed reduction $1$ over $K$ in $\om$ (see \cite[Definition 7.23]{Watson}) and $\mu_u$ is the associated Riesz measure, then $\Cap(K, \om)=\mu_u(K)$ (see \cite[Definition 7.33]{Watson}). By \cite[Theorem 7.28]{Watson}, 
$$
u=\widehat{R}_{1}(K,\om) = G_\om \mu_u.
$$
It is easy to see that if $K \subset \om_1 \subset \om_2 \subset \rrn$ are open sets, then 
\begin{equation}\label{eq:cap-inclusion}
\Cap(K,\rrn) \leq  \Cap(K, \om_2) \leq \Cap(K, \om_1).
\end{equation}

This definition can be extended to arbitrary sets $S\subset \Rn1$: the \textit{inner thermal capacity} of $S$ is given by
\begin{equation}\label{def:capacity-minus}
\textup{Cap}_{-}(S, \om)=\sup \{\textup{Cap}(K,\om):S \supset K   \text{ compact}\},
\end{equation}
and its \textit{outer thermal capacity} by
\begin{equation}\label{def:capacity-plus}
\textup{Cap}_+(S,\om)=\inf \{\textup{Cap}_-(D,\om):S \subset D \text{ open}\}.
\end{equation}
If the inner and the outer thermal capacities of $S$ coincide, we say that $S$ is \textit{capacitable} and we denote $\Cap(S,\om)\coloneqq \Cap_-(S,\om)=\Cap_+(S,\om)$. In the case $\om=\rrn$, we simply write $\Cap(S)$. If $E$  is  a Borel set, then, by \cite[Theorems 7.15 and 7.49]{Watson}, it is capacitable. 


\vv

A set $Z\subset \Rn1$ is called \textit{polar} if there is an open set $\om \supset Z$ and a supertemperature  $w$ on $\om$  such that $Z\subset\{\bx\in \om: w(\bx)=+\infty\}$. 
Observe that if $Z_1\subset Z$ and $Z$ is polar, then $Z_1$ is polar itself, and a set $S\subset\Rn1$ is polar if and only if $\Cap(S)=0$ (see  \cite[Theorem 7.46]{Watson}). Moreover, if a set is contained in a horizontal hyperplane its capacity is equal to its $n$-dimensional Lebesgue measure (see \cite[Theorem 7.55]{Watson}). For more properties of polar sets we refer the reader to \cite[Chapter 7]{Watson}. 

\vv

An open set $\Omega$ is \textit{quasi-regular} (resp. {\it regular}) if the set of irregular points of $\partial_e\Omega$ is polar (resp. empty). We remark here that there are domains so that the capacity of the set of the irregular points of $\partial_e\Omega$ for the heat equation is positive (see \cite[p. 336]{TW85}). This comes in contrast to the corresponding result in elliptic potential theory, where the set of irregular boundary points has zero capacity and thus, is polar.

\vv

\begin{lemma}\label{lem:capacity-upperbound}
If $C_r$ is a cylinder of radius $r$, there exists a dimensional constant $c_1>0$ such that 
\begin{equation}\label{eq:capacity-upperbound}
\Cap(\overline{C_r}, C_{2r}) \leq c_1 r^n.
\end{equation}
\end{lemma}

\begin{proof}
Let $\vphi\in C^\infty_c(C_{2r})$ such that $\vphi=1$ in $\overline{C_r}$, $\vphi \geq 0$, and $|H^*\vphi | \leq c_n' r^{-2}$. Then, if $\mu_u$ is the Riesz measure associated with  $u=\widehat{R}_1(\overline{C_r}, C_{2r})$, since $0 \leq u \leq 1$,  it holds 
\begin{equation*}
\Cap(\overline{C_r}, C_{2r})=\mu_u(\overline{C_r}) \leq \int \vphi \,d\mu_u = \int u \,H^* \vphi \leq c_n' r^{-2} |C_{2r}| = c_n r^n.\qedhere
\end{equation*}
\end{proof}

\vv

\begin{lemma}\label{lem:cdc-content}
Let $C_r$ be a cylinder of radius $r$ centered at $\bar \xi_0 \in \rrn$ and let $K \subset \overline{C_{r}}$ be a  compact set such that $\bar \xi_0 \in K$. If  $s \in (0,2]$, then there exists $c_2>0$ (independent of $K$) such that 
\begin{equation}\label{eq:cdc-content}
\Cap(K, C_{2r}) \geq c_2\, \frac{\HH_{p,\infty}^{n+s}(K)}{\min(\diam(K),r)^s}.
\end{equation}
\end{lemma}

\begin{proof}
If we set $\rho = \min(\diam(K), r)$,  it holds that $K \subset C_{\rho} \cap C_r$. {Without loss of generality we assume $\mathcal H^{n+s}_{p,\infty}(K)>0$, otherwise \eqref{eq:cdc-content} is trivially verified.} By the parabolic version of  Frostman's lemma, whose proof is analogous to the Euclidean one and is omitted (see e.g. \cite[Lemma 8.8]{Mattila}), we can find a Borel measure $\mu'$ supported on $K$, such that 
\begin{itemize}
\item$\mu'(C_r(\by))\lesssim r^{n+s}$, for any $r>0$ and $\by \in \rrn$, and
\item$\HH_{p,\infty}^{n+s}(K ) \geq \mu'(K) \geq c\,\HH_{p,\infty}^{n+s}(K)$,
\end{itemize}
where $c_1$ depends on $n$.  

Let $G_{2r}$ be the Green function in $C_{2r}$ and  define  $\mu=\mu' \rho^{-s}$. We  claim that the corresponding heat potential $G_{2r}\mu(\bx) \lesssim 1$ for any $\bx \in C_{2r}$. Indeed, notice that
\begin{align*}
G\mu(\bx) = \int_{\|\by-\bx\| \leq \rho} & G_{2r}(\bx,\by) \,d\mu(\by) \\
&+ \int_{\|\by-\bx\| > \rho} G_{2r}(\bx,\by) \,d\mu(\by)\eqqcolon I_1+I_2.
\end{align*}
Using  \eqref{eq:Green-two-sidedbound} in conjunction with  the fact that $\mu'$   is supported on $K$ and has $(n+s)$-growth, we have that
\begin{align*}
I_2 \lesssim \int_{\|\by-\bx\| > \rho}& \|\by-\bx\|^{-n} \,d\mu(\by)\\
& \lesssim \rho^{-n} \mu'(K)  \rho^{-s} \leq  \mu'(C_\rho) \rho^{-n-s} \lesssim 1.
\end{align*}
If $A_k(\bx)\coloneqq C_{2^{-k} \rho}(\bx) \setminus C_{2^{-k-1} \rho}(\bx)$, arguing as above, we infer that
\begin{align*}
I_1 \lesssim \sum_{k \geq 0} \int_{A_k(\bx)} (2^{-k-1}\rho)^{-n} \,d\mu(\by) \lesssim \sum_{k \geq 0} (2^{-k-1}\rho)^{-n} \rho^{-s} (2^{-k}\rho)^{n+s} \lesssim 1,
\end{align*}
concluding the proof of the claim. 

 If we normalize $\mu$ so that $G_{2r}\mu(\bx) \leq 1$, we deduce that $\mu$ is an admissible measure for the definition of thermal capacity on compact sets. Therefore, 
$$
\Cap(K, C_{2r}) \geq \mu(K) \gtrsim \rho^{-s} \HH_{p,\infty}^{n+s}(K),
$$
and  the proof of the lemma is now complete.
\end{proof}

\vv

\begin{corollary}\label{cor:capacity_rectangle}
Let $0<b<a \leq 1$ and $r>0$, and for $\bar \xi_0 =(\xi_0, t_0) \in \rrn$ we set 
\begin{equation*}
R^-_{a,b}(\bar \xi_0 ; r) = B(\xi_0, r) \times \left(t_0-(ar)^2, t_0-(br)^2\right).
\end{equation*}
Then, we have that 
\begin{equation}\label{eq:cap-Rabr}
\Cap \big(\overline{R^-_{a,b}(\bar \xi_0;r )}, C_{2r}(\bar \xi_0 ) \big) \approx r^n,
\end{equation}
where the implicit constant depends on $n$, $a$, and $b$.
\end{corollary}

\begin{proof}
The upper bound follows from Lemma \ref{lem:capacity-upperbound}. For the lower bound, if $s=n+2$ and $K=\overline{R^-_{a,b}(\bar \xi_0 ;r )}$,  the proof of Lemma  \ref{lem:cdc-content} goes through unchanged if instead of the Frostman's measure we use $\HH_p^{n+2}$, which, in turn,  is  equal to a constant multiple of the $(n+1)$-dimensional Lebesgue measure $\mathcal{L}^{n+1}$, showing that 
\begin{equation}\label{eq:boundcap12}
\Cap(K,  C_{2r}(\bar \xi_0 )) \gtrsim \mathcal{L}^{n+1}(K)/r^2 \approx_{a,b,n} r^n.
\end{equation}
As $a$ and $b$ are  arbitrary, we may approximate $R^-_{a,b} (\bar \xi_0;r )$  by a sequence of compact subsets of the form $\overline{R^-_{a_k, b_k }(\bar \xi_0;r)}$ and obtain our result. 
\end{proof}

\vvv

For $\bar x=(x,t)\in\Rn1$, we  define the \textit{heat ball} centered at $\bar x$ with radius $\rho>0$ to be the set
\begin{align}\label{eq:def_heat_ball}
E&(x,t;\rho)\coloneqq \bigl\{\bar y\in \Rn1 : \Gamma(\bar x- \bar y)>(4\pi\rho)^{-n/2}\bigr\}\\
&= \Bigl\{(y,s)\in \Rn1 : |x-y|<\sqrt{2n(t-s)\log\Bigl(\frac{\rho}{t-s}\Bigr)}, \, t-\rho<s<t\Bigr\}.\notag
\end{align}
It is not hard to see that $E(x,t;\rho)$ is a convex body, axially symmetric about the line $\{x\} \times \R$, and 
\[E(\bar x;\rho)\subset B\Bigl(x,\sqrt\frac{2n\rho}{e}\Bigr)\times (t-\rho,t).
\]

Let us set
$$
A(\bx; \rho/2,\rho)=\overline{E(\bx ; \rho) \setminus E(\bx;\rho/2)}
$$
to be the closed heat annulus of radius $\rho$. It is clear that if $\by \in A(\bar \xi_0; \rho/2,\rho)$ satisfies  $ 0 <  t -s  < \rho/2$, then
it holds that
\begin{align}\label{eq:heat-annulus-small}
2n(t-s) \log \left(\frac{\rho}{t-s} \right) -[2n \log 2] (t-s) &\leq |x-y|^2 \leq 2n (t-s) \log \left(\frac{\rho}{t-s} \right).
\end{align}

\begin{remark}\label{rem:annulus-cover}
Note that if  $\rho=r^2$, $\alpha \in (0,1)$, and ${t-s}= \alpha r^2$, then the region given by  \eqref{eq:heat-annulus-small} is an annulus in the spatial variable which can be covered by at most $c_n'$ many (spatial) cubes of sidelength $\sqrt{\alpha}  r$. 
\end{remark}

There are  parabolic analogues of the Wiener criterion that determine whether  a  point $\bar \xi \in \d_n\om$  is regular in  terms of the capacity of the complement of $\om$ that lies either in heat balls centered at $\bar \xi_0$ or in time-backwards  parabolic cylinders $R^-_{ar}(\bar\xi_0)$. To be precise,  $\bar \xi_0\in\partial_n \Omega$ is regular if and only if
\begin{equation}\label{eq:wiener_test-annulus}
\int^1_0\frac{\Cap( A(\bar \xi_0; \rho/2, \rho)\cap \Omega^c)}{\rho^{n/2}}\, \frac{d\rho}{\rho}=\infty.
\end{equation}
The necessity was proved by Evans and Gariepy  \cite[Theorem 1]{EG82}  and  the sufficiency by Lanconelli \cite{Lanc73}\footnote{A different necessary and sufficent condition is given by  \cite{Land69}.} Equivalently,
$\bar \xi_0\in\partial_n \Omega$ is regular if and only if
\begin{equation}\label{eq:wiener_test}
\int^1_0\frac{\Cap( \overline{E(\bar \xi_0; \rho)}\cap \Omega^c)}{\rho^{n/2}}\, \frac{d\rho}{\rho}=\infty.
\end{equation}
This follows from \eqref{eq:wiener_test-annulus} and  \cite[Theorem 1.13]{Br90}.  Alternatively,  $\bar \xi_0$ is regular if and only if 
\begin{equation}\label{eq:regular-sum}
\sum_{k \geq 0}  \lambda^{-kn} \Cap( A(\bar \xi_0 ;  \lambda^{-k-1},  \lambda^{-k} )\cap \Omega^c)=\infty,\,\, \textup{for some}\,\, \lambda>1.
\end{equation}
In fact if $\bxi_0$ is regular, then \eqref{eq:regular-sum} holds for all $\lambda>1$.

\begin{remark}\label{rem:regular}
In \cite{EG82} and \cite{Lanc73}, the authors do not specify that this criterion only works for points on $\d_n \om$, although it is clear  this is the case. Observe that, by definition, if $\bxi_0 \in \d_a\om$, there exists $\rho>0$ small enough so that  $E(\bxi_0;\rho)\cap \Omega^c=\varnothing$ and thus, the integral in \eqref{eq:wiener_test} clearly converges. Nevertheless, Watson \cite[Theorem 4.1]{Wa14} proved that a point $\bxi_0 \in \d_{ss} \om$ is regular if and only if there exists $r_0>0$ such that $H(\bxi_0,r_0)$ is a connected component of $\om \cap B(\bxi_0,r_0)$, where $B(\bxi_0,r_0)=\{(x,t): |x-\xi_0|^2+|t-t_0|^2<r^2_0\}$ and $H(\bxi_0,r_0)=B(\bxi_0,r_0) \cap \{ t<t_0\}$. 
\end{remark}

\vv

\begin{lemma}\label{lem:tbcdc}
Let $\om \subset \rrn$ be an open set and $\bxi_0=(\xi_0,t_0) \in \d_n \om$, If  
$$
\Cap(\overline{E(\bxi_0;r^2)} \cap \om^c)\geq 2c r^n,
$$
 then there exists $a \in (0,1/2)$ depending on $n$  such that
\begin{equation}\label{eq:tbcdc}
\Cap\big(\overline{E(\bxi_0;r^2)} \cap \om^c \cap \{t<t_0-(ar)^2\} \big)\geq c r^n.
\end{equation}
\end{lemma}

\begin{proof}
If we set 
$$
E_k:= \left\{ \bx=(x,t) \in \overline{E(\bxi_0;r^2)}: 2^{-2(k+1)} r^2 \leq t_0- t \leq 2^{-2k} r^2 \right\},
$$ 
then, for any $\bx \in E_k$, it holds
$$
|x-\xi_0|<2^{-k} r\,\sqrt{2n\log\Bigl(\frac{r^2}{2^{-2(k+1)} r^2}\Bigr)} \leq  \sqrt{2n \log2} \frac{ \sqrt{k}}{2^{k}} r.
$$
Thus, we can cover $E_k$ by at most $c(n) k^{n/2}$ number of cylinders $C_j(k)$ of radius $2^{-k}r$, which, in view of the subadditivity of capacity  and Lemma \ref{lem:capacity-upperbound}, infers that
$$
\Cap(E_k) \leq  \sum_{j \geq 1} \Cap(C_j(k)) \leq c(n) c_1 k^{n/2} 2^{-kn} r^n.
$$
Therefore, if we set $E_M:= \bigcup_{k \geq M} E_k$, we have that 
$$
\Cap(E_M) \leq  \sum_{k \geq M} \Cap(E_k) \leq \tilde{c}(n) \sum_{k \geq M}k^{n/2}2^{-kn} r^n.
$$
Since $\sum_{k \geq 1}k^{n/2}2^{-kn}$ converges, we can find  $M>0$ depending on $n$ so that  $\Cap(E_M) \leq c r^n$, which implies \eqref{eq:tbcdc} for $a=2^{-M}$.
\end{proof}

Similarly we can prove the following lemma.

\begin{lemma}\label{lem:tb-wiener}
If $\om \in \rrn$ is an open set and $\bar \xi_0=(\xi_0,t_0) \in \d_n \om$ is a regular point, then there exists a sequence $M_k>1$ such that  
\begin{equation}\label{eq:tb-wiener}
\sum_{k \geq 0} 2^{-kn} \Cap \big( A(\bx_0;2^{-2(k+1)}, 2^{-2k})\cap \{t<t_0-(2^{-M_k} 2^{-k})^2 \cap \om^c\}\big)=\infty.
\end{equation}
\end{lemma}

\vv

Given $a \in (0,1)$, we define the translated time-backwards cylinder
 \begin{align}\label{eq:R-ar}
 R^-_{a}(\bar x;r) \coloneqq  C^-_r(x, t-(ar)^2) = B(x, r) \times \left(t-r^2, t- (ar)^2 \right).
\end{align}

\vv

\begin{lemma}\label{lem:rect_cont_heat_ball}
If $a\in (0,1)$ and $\bar \xi_0=(\xi_0,t_0)\in\Rn1$, it holds that 
\begin{equation}\label{eq:rect_cont_heat_ball}
R^-_{a}(\bar\xi_0;r)\subset E(\bxi_0;\rho), \quad \text{ for }\rho\geq e^\frac{1}{2n}r^2.
\end{equation}
\end{lemma}

\begin{proof}
Let $R'\coloneqq B(\xi_0,r)\times\{t_0-(ar)^2\}$ and $R''\coloneqq B(\xi_0,r)\times\{t_0-r^2\}.$ By the definition \eqref{eq:def_heat_ball},  we have that $R'\subset E(\bar\xi_0;\rho)$ if and only if 
\[
\rho\geq a^2\exp\Bigl(\frac{1}{2na^2}\Bigr)\,r^2\eqqcolon \tilde c_1 r^2,
\]
and $R''\subset E(\bar\xi_0;\rho)$ if and only if 
\[
\rho\geq \exp\Bigl(\frac{1}{2n}\Bigr)\,r^2\eqqcolon \tilde c_2 r^2.
\]
Remark that, for any fixed $\lambda>0$, the function $y e^{\lambda y}$ is increasing for  $y>-\lambda^{-1}$, and thus, $\tilde c_2> \tilde c_1$. Hence, if $\rho \geq  \tilde c_2 r^2$, we have that $R'\cup R''\subset E(\bar\xi_0;\rho)$ and, by the  convexity of $E(\bar\xi_0;\rho)$, we get \eqref{eq:rect_cont_heat_ball} for $\tilde c=\tilde c_2$. 
\end{proof}

As  backwards cylinders $ R^-_{a}(\bar \xi_0 ; r) $ appear more naturally in  applications, it is interesting to know a version of the Wiener's criterion with $ R^-_{a}(\bar \xi_0;r) $ instead of heat balls. This can be obtained using the following lemma (compare it with  \cite[Theorem 3.1]{GZ82}).

\begin{lemma}\label{cor:Wiener}
Let $\Omega$ be an open set and $\bar \xi_0 \in \partial_n \Omega$. If  there exists a constant $a>0$ such that
\begin{equation}\label{eq:Wiener-Rar}
\int_0^1\frac{\Cap(\overline{R^-_{a}(\bar\xi_0;r)}\cap \Omega^c)}{\Cap(\overline{R^-_{a}(\bar \xi_0;r)})}\, \frac{dr}{r}=\infty,
\end{equation}
then $\bar \xi_0$ is regular. Conversely, if $\bar \xi_0 \in  \partial_n \Omega$ is regular,  then there exists a function $a:(0,1) \to (0,\frac{1}{2})$ so that \eqref{eq:Wiener-Rar} holds for $R^-_{a(r)}(\bar\xi_0;r)$.  
\end{lemma}

\begin{proof}
For $\bar \xi_0\in\partial_n\Omega$, we have that
\begin{equation*}
\begin{split}
\int^{\tilde c}_0\frac{\Cap( \overline{E(\bar \xi_0;\rho)} \cap \Omega^c)}{\rho^{n/2}}\, \frac{d\rho}{\rho} &=\frac{2}{\tilde{c}^{n/2}} \int_0^1\frac{\Cap( \overline{E(\bar \xi_0;\tilde c r^2) } \cap \Omega^c)}{r^{n}}\, \frac{dr}{r}\\
& \overset{\eqref{eq:rect_cont_heat_ball}}{\geq} \int_0^1 \frac{\Cap(\overline{R^-_{a}(\bar\xi_0;r)}\cap \Omega^c)}{r^{n}}\, \frac{dr}{r}=\infty,
\end{split}
\end{equation*}
where $\tilde c$ is given in the proof of Lemma \ref{lem:rect_cont_heat_ball}. The converse direction follows from Lemma \ref{lem:tb-wiener}.
\end{proof}

\vv

\begin{lemma}\label{lem:CPC-regular}
 If $\om$ has the TBCPC at $\bxi_0 \in \d_n \om$, then \eqref{eq:Wiener-Rar} is satisfied and $\bxi_0$ is regular.
\end{lemma}

\begin{proof}
If $\om$ satisfies the TBCPC at $\bxi_0$, by Lemma \ref{lem:tbcdc}, there exists $a \in (0,1/2)$ so that $a \approx_n 1$ and 
\begin{align*}
\int_0^1 \frac{\Cap(\overline{R^-_{a}(\bar\xi_0;r)}\cap \Omega^c)}{r^{n}}\, \frac{dr}{r} &\geq \sum_{j\geq 1} \int_{r_j}^{2 r_j} \frac{\Cap(\overline{R^-_{a}(\bar\xi_0;r)}\cap \Omega^c)}{r^{n}}  \gtrsim  \sum_{j\geq 1} 1= \infty,
\end{align*}
where $r_j \to 0$ is the sequence in the definition of TBCPC. Thus, \eqref{eq:Wiener-Rar} holds and $\bxi_0$ is regular. 
\end{proof}

\vv

We will now introduce a class of regular domains that has played an important role in  (free) boundary value problems for harmonic and elliptic measure.  

\vv

Let $\bar \xi_0 \in \mathcal{S}\om$. If there exists  $a \in (0,1]$ and  $c>0$ such that
\begin{equation}\label{eq:CDC}
\frac{\Cap(\overline{R^-_{a}(\bar \xi_0;r)}\cap \Omega^c)}{\Cap(\overline{R^-_{a}(\bar \xi_0;r)})}  \geq c, \quad \textup{for all}\,\, 0<r<  \sqrt{(t_0-T_{\min})/{2}},
\end{equation}
then, we say that  $\om$ has the {\it time backwards cylindrical capacity density condition at the point $\bar \xi_0$}. Although the TBCDC looks more general, because  of Lemmas \ref{lem:rect_cont_heat_ball} and \ref{lem:tbcdc}, we  can see that the two conditions are in fact equivalent. Note that, by Lemma \ref{cor:Wiener} and Remark \ref{rem:cdc}, if $\om$ has the TBCDC at the point $\bar \xi_0 \in \mathcal{S}\om$, then $\bar \xi_0 $  is a regular point and belongs to $\d_n\om$.

\begin{remark}\label{rem:cdc}
If $\Omega$ satisfies the TBCDC, then it clearly  holds that $\partial_s\Omega\cap \mathcal S\Omega=\partial_{ss}\Omega\cap \mathcal S\Omega=\varnothing $ and 
\begin{equation}\label{eq:lateral_CDC}
\mathcal S\Omega=\partial_n\Omega\cap \mathcal S\Omega = \partial_e\Omega\cap \mathcal S\Omega.
\end{equation}
Moreover, the range of $r$ in \eqref{eq:CDC} is chosen so that, if $\bar x=(x,t)\in \mathcal S\Omega$ and $\bar y\in C_r(\bar x) \cap \om$, then 
\begin{equation}\label{eq_remark_distance}
\dist_p(\bar y,\mathcal S\Omega)=\dist_p(\bar y,\partial_e\Omega).
\end{equation}
Indeed, if we choose $r>0$ so that $r < \sqrt{t-T_{\min}-r^2}$,  we have 
\begin{align*}
\dist_p(\bar y,\mathcal{S}\Omega)\leq\dist_p(\bar y, \bar x)<r &< \sqrt{t-T_{\min}-r^2} \\
&= \dist_p(\mathcal B C_r(\bar x),\mathcal B\Omega|_{T_{\min}} )\leq \dist_p(\bar y,\mathcal B\Omega|_{T_{\min}}).
\end{align*}
 Recalling that $R^-_{ar}(\bar \xi_0) =  C_r(\xi_0, t_0-(ar)^2) $, we obtain the range of $r$ in \eqref{eq:CDC}.
\end{remark}

\vv

A crucial property of our definitions of TBCDC is their invariance under parabolic scaling. 

\begin{lemma}\label{lem:geom_boundary_CDC}
Let $\Omega$ be a domain that satisfies the TBCDC with constants $a,c$ as in \eqref{eq:CDC}. Let $\bar \xi=(\xi, \tau)\in \mathcal S\Omega$, $\rho>0$ and denote $\tilde{\Omega}\coloneqq T_{\xi,\rho}[\Omega]$. Then $\tilde{\Omega}$ satisfies the TBCDC with the same parameters. 
\end{lemma}

\begin{proof}
Let us observe that for  $\bar \zeta \in \mathcal{S} \wt \Omega$ and every $r>0$ we have that
\[C_{r/\rho}^-(\bar \zeta)\cap \widetilde \Omega^c= T_{\xi, \rho}\bigl(C_{r}^-(T_{\xi, \rho}^{-1}  \bar \zeta)\cap \Omega^c\bigr)\]
and $\Gamma (\delta_\rho \bar x)= \rho^{-n} \Gamma (\bar x)$. 
We claim that 
\begin{equation}\label{eq:capacity-rigidmotion}
\Cap (\overline{C_{r/\rho}^-(\bar \zeta)}\cap \widetilde \Omega^c)=\rho^{-n}\Cap(\overline{C_{r}^-(T_{\bxi, \rho}^{-1}  \bar \zeta)}\cap \Omega^c).
\end{equation}
Set $K=\overline{C_{r}^-(T_{\bxi, \rho}^{-1}  \bar \zeta)}\cap \Omega^c$ and $\widetilde K=\overline{C_{r/\rho}^-(\bar \zeta)}\cap \widetilde \Omega^c$, and let $\mu$ be the unique  Radon measure supported in $\widetilde K$ such that $\mu(\widetilde K)=\Cap(\widetilde K)$. If $\widetilde{\mu} = T_{\bxi, \rho\, \sharp}^{-1} \mu$, then it is clear that $\supp \widetilde \mu \subset K$. Moreover, since $G \mu \leq 1$ in $\rrn$, it holds that
\begin{align*}
G\widetilde{\mu}(\bx) = \int \Gamma \big(\bx - T^{-1}_{\bxi, \rho} \by \big)\,d\mu(\by)= \rho^{-n}  \int \Gamma\big(T_{\bxi, \rho}\bx -  \by \big) \,d\mu(\by) \leq \rho^{-n}.
\end{align*}
Therefore, the measure $\rho^{n} \widetilde \mu$ is an admissible measure for   $\Cap(K)$ and thus, 
$$
\rho^{n}  \Cap(\widetilde K) =\rho^{n}  \mu(\widetilde K) = \rho^{n}  \widetilde \mu(K) \leq \Cap(K).
$$
The proof of the converse inequality is similar and we omit it. This proves  \eqref{eq:capacity-rigidmotion}, which, in turn, is valid if we replace the backwards in time cylinders $C_{r/\rho}^-$ by $R^-_{a}(\bar \zeta; r/\rho)$, the truncated cylinders defined in \eqref{eq:R-ar}. Thus, the TBCDC for $\wt \om$ readily follows from the TBCDC for $\om$.
\end{proof}

\vv

The next lemma was proved in \cite[Lemma 2.2]{GH20} in the particular case of domains with Ahlfors-David regular boundaries that, in addition, satisfy the time-backwards Ahlfors-David regular condition. In light of \eqref{eq:cdc-content}, those domains satisfy the TBCDC. 

\begin{lemma}\label{lem:bourgain}
Let $\hm^{\bx}$ be the caloric measure in $\Omega$ with pole at $\bx=(x,t) \in \om$. If $\bar \xi_0=(\xi_0, t_0) \in  \mathcal{P} \om \cap \mathcal{S} \om$, $r>0$, and $a \in (0,1/2)$, we define 
 \begin{align}\label{eq:R+ar}
 \widehat{R}^+_{a}(\bar \xi_0;r) = B( \xi_0, r) \times \bigl(t_0-(ar)^2/2, t_0+r^2 \bigr).
\end{align}
Then there exists $M_0>1$ depending on $n$ and $a$, such that for any  $\bx \in  \overline{\widehat{R}^+_{a}(\bar \xi_0;r)} \cap \om$,
\begin{align}\label{eq:Bourgain}
 \hm^{\bx} \Big(C_{M_0r}(\bar \xi_0) \cap  \{|t-t_0|<r^2\} \Big) \gtrsim  \frac{\Cap(\overline{R^-_{a}(\bar \xi_0;r)}\cap \om^c)}{r^{n}},
\end{align}
 where the implicit constant depends only on $n$  and $a$.
If $\om$ satisfies the TBCDC at $\bar \xi_0$, then  there exists $a \in (0,1/2)$ and $c=c(n,a)>0$ such that
 \begin{equation}\label{eq:bourgain_estimate}
\omega^{\bar x}(C_{r}(\bar \xi_0))\geq c, \quad \textup{for}\,\, \bar x \in \overline{\widehat{R}^+_{a}(\bar \xi_0 ; r/M_0)} \cap \om, \,\,\textup{and}\,\,0<r< M_0\sqrt{(t_0-T_{\min})/{2}}.
\end{equation}
\end{lemma}

\begin{proof}
 Without loss of generality,  we may assume that $\bar \xi_0= \bar 0 \in  \mathcal{S} \om$ and set $K= \overline{R^-_{a}(\bar 0;r)} \cap \om^c$. For notational convenience, we will just write $R^-_{a,r}$ and $\widehat{R}^+_{a,r}$.  We also assume that $\Cap(\overline{R^-_{a,r}}\cap \om^c)>0$ since otherwise \eqref{eq:Bourgain} is trivially true.

By the definition of capacity on compact sets,  there exists a unique positive Radon measure $\mu$ supported on $K$ such that $\| G\mu\|_{L^\infty} \leq 1$ and $\| \mu\| = \Cap(K)$. Since for any $\bx \in \overline{R^+_{ar}} $ and $\by =(y,s) \in  \overline{R^-_{a,r}}$, it holds that 
$$
a^2 r^2/2\leq t-s \leq  2 r^2 \quad \textup{and}\quad |x-y| \leq 2r,
$$
we infer that 
$$
\Gamma(\bx -\by ) \geq (8 \pi r^2)^{-n/2}e^{-2/a^2}=:c_1 r^{-n},
$$
which, in turn, implies that 
$$
G\mu(\bx) \geq c_1 \|\mu\| r^{-n}, \quad \textup{for any} \,\,\bx \in \widehat{R}^+_{a,r}.
$$
Moreover, if $\bx \in \rrn \setminus C_{M_0 r}$ and $(y,s) \in  R^-_{a,r} $, we have that  
$$
\|\bx - \by \| \geq \|\bx\| - \|\by\| \geq M_0r -  r =:( M_0- 1)r,
$$
and thus, by  \eqref{eq:poitwiseHeat kernel}, 
$$
G\mu(\bx) \leq ( (M_0-1) \sqrt{\pi})^{-n} \|\mu\|  r^{-n}, \quad \textup{for any} \,\,\bx \in \rrn \setminus C_{M_0r}.
$$
Define $u = G\mu - ( (M_0-1) \sqrt{\pi})^{-n} \|\mu\|  r^{-n}$ and choose $M_0$ so that $ 2( (M_0-1) \sqrt{\pi})^{-n} \leq c_1$. Then, since $\|G\mu\|_{L^\infty} \leq 1$, the following hold:
\begin{enumerate}
\item $u$ is continuous in $\rrn$ and $Hu=0$ in $\om$,
\item $u \leq 1$ in $\rrn$,
\item $u \leq 0$ in $[\rrn \setminus C_{M_0r}]\cup [\R^n \times \{t \leq - r^2\}]$,
\item $ u \geq \frac{c_1}{2} \|\mu\| r^{-n}$ in $ \overline{\widehat{R}^+_{a,r}}$.
\end{enumerate}
Since $F:=C_{M_0 r}\cap [\R^n \times \{t > - r^2\}] \cap \d_e \om$ is Borel, we have that  $\chi_F$ is resolutive,  $u  \in \mathfrak{L}_{\chi_F}$, and $\omega^{\bx}(F) = S_{\chi_F}(\bx)$.
Thus, $u(\bx) \leq  \omega^{\bx}(F) $ for all $\bx \in \om$ which,   by the  item (4) above and the fact that  $\supp \hm^{\bx} \subset \d_e \om \cap [\R^n \times \{t <  r^2\}]$ for any $\bx \in \overline{\widehat{R}^+_{a,r}}$, proves \eqref{eq:Bourgain}. It is straightforward to see that, by the definition of TBCDC, \eqref{eq:bourgain_estimate} follows from the latter estimate.
\end{proof}

\begin{remark}\label{rem:bourgain-adjoint}
Lemma \ref{lem:bourgain} holds for the adjoint caloric measure if we replace the cylinders $R^-_{a}(\bxi_0, r)$ and  $\widehat{R}^+_{a}(\bxi_0, r)$ with their reflections across the hyperplane passing through their centers orthogonal to the time axis.
\end{remark}

If $\om$ is regular there is  the above lemma has a much easier proof based on the reduction function. Also, the dilation factor $M_0$ need not be taken large enough and just $M_0=2$  does the job. 

\begin{lemma}\label{lem:bourgain-bis}
Let $\om\subset \rrn$  be open, $\bxi_0 \in \mathcal{P} \om \cap\mathcal{S} \om$, $a\in (0,1/2)$, $\beta \geq 2$, and $K= \overline{R^-_{a}(\bar \xi_0;r)}\cap \om^c$. If we set $u=\widehat{R}_1(K,C_{\beta r}(\bxi_0))$, then for any $\bx \in \overline{\widehat{R}^+_{a}(\bar \xi_0;r)} \cap \om$,
\begin{equation}\label{eq:bourgain-reduction}
u(\bx)\gtrsim \Cap\big( K,C_{\beta r}(\bxi_0)\big)/{r^n} \gtrsim {\Cap( K )}/{r^n}.
\end{equation}
If $w$ is a non-negative  subtemperature in $C_{\beta r}(\bxi_0)$ so that $w = 0$ on $K$ and $w \leq 1$ in $C_{\beta r}(\bxi_0)$, there exists $\kappa \in (0,1)$ such that 
\begin{equation}\label{eq:bourgain-subcaloric}
w(\bx)\leq  1- \kappa\,{\Cap\big( K, C_{\beta r}(\bxi_0)\big) }/{r^n} \leq 1-\kappa {\Cap( K)}/{r^n}, \quad \textup{for}\,\, \bx \in \overline{\widehat{R}^+_{a}(\bar \xi_0;r)} \cap \om.
\end{equation}
If $ \mathcal{S}\om \cap C_{\beta r}(\bxi_0) \subset \mathcal{P} \om  $ consists of regular points, then \eqref{eq:bourgain_estimate} holds with $M_0=2$.
\end{lemma}

\begin{proof}
As in the proof of the previous lemma,  we assume that $\bxi_0= \bar 0$ and use the notation $R^-_{a,r}$ and $\widehat{R}^+_{a,r}$. Moreover, we assume that $\Cap(K)>0$ since otherwise \eqref{eq:bourgain-reduction} is trivial. We also denote 
$$
\widetilde{R}^-_{a,r}= B( 0, (1+ 0.25 a)r) \times (-(1+0.25 a^2)r^2, -0.75 a^2 r^2)
$$ 
and let $\vphi \in C^\infty_c(\wt R^-_{a,r})$ such that $\vphi =1 $ in $R^-_{a,r}$ and $|H^*\vphi| \lesssim_{a,n} r^{-2}$. Then, if $\mu$ is the unique Radon measure such that $\mu(K)=\Cap(K)$, it holds that
\begin{align*}
r^{-n}\Cap(K)=r^{-n} \mu(K) \leq r^{-n} \int \vphi \, d\mu =r^{-n} \int u H^* \varphi \lesssim r^{-(n+2)} \int_{\widetilde{R}^-_{a,r} }u \lesssim_{a,n} \inf_{\overline{\widehat{R}^+_{a,r}}} u,
\end{align*}
where in the last step we used the weak Harnack inequality since $u$ is a non-negative supertemperature and thus, a non-negative supercaloric function (see \cite[Corollary 6.24, p. 128]{Lieb96}).

By the definition of $u$, we have that  for any non-negative supertemperature $v$ in $C_{\beta r}(\bxi_0)$ such that $v \geq 1$ on $K$, it holds that $v \geq u$ and thus, \eqref{eq:bourgain-reduction} is true for $v$ as well. Now, if $w$ is a subtemperature in $C_{\beta r}(\bxi_0)$ so that $w = 0$ on $K$ and $w \leq 1$ in $C_{2r}(\bxi_0)$, then $1-w$ is a non-negative supertemperature in $C_{\beta r}(\bxi_0)$  that is identically $1$ on $K$ and \eqref{eq:bourgain-subcaloric} readily  follows. Finally, by the regularity of $\d \om \cap C_{\beta r}(\bxi_0)$, we may extend $\hm^{\bx}(C_{\beta r}(\bxi_0))$ by $1$ in  $C_{\beta r}(\bxi_0) \setminus \overline{\om}$ and use \cite[Theorem 7.20]{Watson} to show that it is a non-negative supertemperature  in $C_{\beta r}(\bxi_0)$ that is $1$ on $K$ and non-neagtive in $ C_{\beta r}(\bxi_0)$ and use \eqref{eq:bourgain-reduction} to show \eqref{eq:bourgain_estimate}  for $M_0=2$.
\end{proof}
\vv

\begin{remark}\label{rem:Bourgain-subcaloric}
One can show that \eqref{eq:bourgain-subcaloric} holds for upper semicontinuous weakly subcaloric functions as well. Indeed,   if we follow the proof of Lemma \ref{lem:bourgain} substituting $\Gamma(\cdot, \cdot)$ with the Green function $G_{C_{\beta r}(\bxi_0)}(\cdot, \cdot)$ and using \eqref{eq:Green-two-sidedbound} together with the weak minimum principle\footnote{The weak minimum principle still holds for  lower semicontinuous (instead of continuous in $\overline{\om}$) supercaloric functions in the form $\liminf_{\bx \to \mathcal{P}(\om)} v \leq \inf_{\om} v$. See e.g. the proof of \cite[Corollary 6.26]{Lieb96}.} (instead of the resolutivity of characteristic functions of Borel sets) we can  show that if $v$ is a non-negative supercaloric function in $\om \cap C_{\beta Mr}(\bxi_0)$ such that $\liminf_{\bx \to \bxi} v \geq 1$ for every $\bxi \in \d_n \om \cap C_{\beta Mr}(\bxi_0)$, then $v$ satisfies  \eqref{eq:bourgain-reduction}. The rest of the proof is the same as before and we skip the  details.
\end{remark}

\vv

\begin{lemma}\label{lem:Modudlus-continuity}
Let $\om\subset \rrn$  be open, $\bxi_0 \in\mathcal{P} \om \cap \mathcal{S} \om$, $a\in (0,1/2)$, $a_1=a/\sqrt{2}$, and $\beta = a_1^{-1}$. Let $w$ be a non-negative subtemperature in $C_{ \beta r}(\bxi_0)$ so that $w = 0$ on $\overline{R^-_a(\bxi_0 ; r) }\cap \om^c$ and $w \leq M_1$ in $C_{ \beta r}(\bxi_0)$.  If $\bx \in C_\rho(\bxi_0) \cap \om$ for some $0 <\rho \leq r$, then 
\begin{equation}\label{eq:Modudlus-continuity}
w(\bx) \leq \exp\left(- c \int_{\rho}^{r} \frac{\Cap(\overline{R^-_a(\bxi_0 ; s ) }\cap \om^c)}{s^n} \frac{ds}{s} \right) \,\sup_{\om \cap C_{\beta r}(\bxi_0)} w,
\end{equation}
for some  constant $c>0$ depending on $n$ and $a$.
\end{lemma}

\begin{proof}
We assume that $\Cap\big(\overline{R^-_a(\bxi_0 ; \rho )} \cap \om^c \big)>0$, since otherwise \eqref{eq:Modudlus-continuity} is trivial.
We set  $a_k=a_1^{k}$ for $k \geq 0$, and
$$
\theta_k:=\frac{\Cap(\overline{R^-_a(\bxi_0 ; a_{2k} r) }\cap \om^c)}{(\beta a_{2k}r)^n},
$$
and note that $a_{k-1}= \beta a_k$.  By \eqref{eq:bourgain-subcaloric},  it holds that 
$$
w(\bx)\leq   (1-\kappa \theta_1) \sup_{C_{\beta r} (\bxi_0)} w \leq \exp(-\kappa \theta_1)  \sup_{C_{\beta r}(\bxi_0)} w, \quad \textup{for}\,\, \bx \in \overline{\widehat{R}^+_{a}(\bar \xi_0; r)}.
$$
In particular, the latter inequality holds in $\overline{C_{a_1 r}}(\bxi_0)$. We apply \eqref{eq:bourgain-subcaloric} once again in $C_{a_2 r}(\bxi_0)$ and get that
$$
w(\bx)\leq \exp(-\kappa \theta_2)  \sup_{C_{a_1 r}(\bxi_0)} w \leq  \exp(-\kappa (\theta_1+\theta_2))  \sup_{C_{\beta r}(\bxi_0)} w, \quad \textup{for}\,\, \bx \in \overline{\widehat{R}^+_{a}(\bar \xi_0;a_2 r)}.
$$
The latter inequality holds in $\overline{C_{a_3 r}}(\bxi_0)$ and we may apply  \eqref{eq:bourgain-subcaloric} in $C_{a_4 r}(\bxi_0)$.
By iteration, if  $M $ is an integer such that $a_{2M+2} \leq \rho \leq a_{2M}$, we get that 
$$
w(\bx)\leq\exp\Big(-\kappa \sum_{k=0}^M \theta_k \Big)  \sup_{C_{\beta r}(\bxi_0)} w, \quad \textup{for}\,\, \bx \in \overline{\widehat{R}^+_{a}(\bar \xi_0;a_{2M} r)}.
$$
Therefore, since 
\begin{align*}
\int_\rho^r &\frac{\Cap(\overline{R^-_a(\bxi_0 ; s ) }\cap \om^c)}{s^n} \frac{ds}{s} \leq \sum_{k =0}^M \int_{a_{2k+2}}^{a_{2k}} \frac{\Cap(\overline{R^-_a(\bxi_0 ; s ) }\cap \om^c)}{s^n} \frac{ds}{s}\\
& \leq \sum_{k =0}^M \int_{a_{2k+2}}^{a_{2k}} \frac{\Cap(\overline{R^-_a(\bxi_0 ; a_{2k} r ) }\cap \om^c)}{a_{2k+2}^n} \frac{ds}{a_{2k+2}} \leq \beta^{3n+2} \sum_{k =0}^M \theta_k,
\end{align*}
if $c=a_1^{3n+2} \kappa$, we obtain that
\begin{equation}\label{eq:Modudlus-continuity}
w(\bx)\leq\ \exp\left(- c \int_\rho^r \frac{\Cap(\overline{R^-_a(\bxi_0 ; s ) }\cap \om^c)}{s^n} \frac{ds}{s}\right)\sup_{C_{\beta r}(\bxi_0)} w, \quad \text{for}\,\, \bx \in C_{\rho}(\bxi_0).
\end{equation}
\end{proof}

\vv

Given $T>0$, we define $E(T)\coloneqq\{(x,t)\in\Rn1: t<T\}.$ Moreover, for an open set $\Omega$, $\bar x\in \partial \Omega$ and $r>0$ we define $\Omega_r \coloneqq\Omega_r(\bar x)\coloneqq \Omega\cap C_r(\bar x)$ and $\Omega_r(T)\coloneqq \Omega_r\cap E(T).$

\begin{lemma}\label{lem:bdry Holder}
Let $\om \subset \rrn$ be an open set satisfying the TBCDC, and let $\bar \xi_0 \in  \mathcal S \om$ and $0<r< \sqrt{(t_0-T_{\min})/{2}}$. Then for any non-negative  function $u$, which is either weakly subcaloric  or subtemperature in $\om_{2r}(T_1)$ vanishing continuously on $C_{2r}(\bar \xi_0) \cap \d_e \om \cap E(T_1)$, it holds that
\begin{equation}\label{eq:boundaryHolder}
u(\bar y) \lesssim \left( \dist(\bar y,\partial_e\Omega) / r\right)^\alpha \sup_{ \om_{3r/2}(T_1)} u, \quad \textup{for every} \,\,\by\in \om_r(T_1),
\end{equation}
where $T_1=T_{\max}(C_r(\bar \xi_0))=t_0+r^2$.
\end{lemma}

\begin{proof}
For subtemperatures, this is a direct consequence of Lemma \ref{lem:Modudlus-continuity} (with a slightly larger cylinder on the right hand-side of \eqref{eq:boundaryHolder}), while for weakly subcaloric functions one can follow the proof of  \cite[Lemma B.2]{GH20} (which still works  for TBCDC domains).
\end{proof}

\vv

\begin{lemma}\label{lem:estim_green_caloric_measure}
Let $\Omega\subset\Rn1$  be a quasi-regular open  set for $H$ and let  $M>1$. If $\bar \xi_0\in\mathcal{S}\Omega$ and $r>0$, for any $\bar x\in\Omega\setminus C_{r}(\bar \xi_0) $ and $\bar y\in \Omega \cap C_{r/4}(\bar \xi_0)$ we have that
\begin{equation}\label{eq:estim_green_function}
\inf_{\bar z\in C_{r}(\bar \xi_0) \cap \om }\omega^{\bar z}(C_{Mr}(\bar\xi_0)) \, G(\bar x, \bar y)r^n\lesssim \omega^{\bar x}(C_{Mr}(\bar\xi_0)).
\end{equation}
Moreover, if $\Omega$ satisfies either the TBCDC  at $\bar\xi_0$, then  if $M_0>1$ is the constant from Lemma \ref{lem:bourgain}, we have that for $\bar x\in\Omega\setminus C_{r}(\bar \xi_0) $ and $\bar y\in \Omega \cap C_{r/4}(\bar \xi_0)$,
\begin{equation}\label{eq:estim_green_function_CDC}
G(\bar x, \bar y)r^n\lesssim \omega^{\bar x}(C_{M_0r}(\bar \xi_0)), \,\,\textup{for}\,\,0<r< M_0\sqrt{(t_0-T_{\min})/{2}}.
\end{equation}
\end{lemma}

\begin{proof}
Fix $\bar y=(y,s)\in \Omega\cap C_{r/4}(\bar\xi_0)$ and recall that $G(\bx,\by)=0$ for any $\bx \not\in\Lambda^*(\by, \om)$. If  $\bar x=(x,t)\in \d_e C_{r}(\bar \xi_0)  \cap \Lambda^*(\by, \om)$,  then there exists $c>0$ such that
\[
G(\bar x, \bar y)r^n\lesssim\frac{r^n}{\|\bar x-\bar y\|^n}\leq c,
\]
and so, for any $\bar x \in \d_e C_{r}(\bar \xi_0)  \cap \Lambda^*(\by, \om)$,
\begin{align}\label{eq:lem_greenf}
\inf_{\bar z\in C_{r}(\bar \xi_0) \cap \om}\omega^{\bar z}(C_{Mr}(\bar \xi_0)) \, G(\bar x, \bar y)r^n &\leq c \inf_{\bar z\in C_{r}(\bar \xi_0)  \cap \om}\,\omega^{\bar z}(C_{Mr}(\bar \xi_0))\\
&\leq c\, \omega^{\bar x}(C_{Mr}(\bar \xi_0)).\notag
\end{align} 

Define now
\[
u(\bar x)\coloneqq c^{-1}\inf_{\bar z\in C_{r}(\bar \xi_0)}\omega^{\bar z}(C_{Mr}(\bar \xi_0)) G(\bar x, \bar y)r^n, \qquad v(\bar x)\coloneqq \omega^{\bar x}(C_{Mr}(\bar \xi_0)).
\]
If we set  $w(\bar x)=u(\bar x)-v(\bar x)$, then it is clear that $w$ is  a caloric  function, hence continuous, in $\Omega \setminus \overline{C_{r/2}(\bar \xi_0)}$. Let $\mathcal{E}(\bar \xi_0, r/2)$ be the ellipsoid that circumscribes $C_{r/2} (\bar \xi_0)$  lying inside $C_{r} (\bar \xi_0)$. If ${\bz}_*=(z_*,t_*)$ is the point on the boundary of the ellipsoid such that $t_*<t$ for every $\bz=(z,t) \in\d{\mathcal{E}(\bar \xi_0, r/2)} \setminus\{{\bz}_*\}$, then 
$$\d \mathcal{E}(\bar \xi_0, r/2) \setminus \{ {\bz}_*\} \subset  \d_n(\rrn \setminus \mathcal{E}(\bar \xi_0, r/2))\quad  \text{and} \quad {\bz}_* \in \d_{ss}(\rrn \setminus \mathcal{E}(\bar \xi_0, r/2)),
$$
while it is clear by \eqref{eq:wiener_test} and Remark \ref{rem:regular}, all the points of $\d \mathcal{E}(\bar \xi_0, r/2)$ are regular points of $\d_e (\rrn \setminus \mathcal{E}(\bar \xi_0, r/2))$ for both $H$ and $H^*$.
 Using \eqref{eq:lem_greenf} we obtain
\[
\lim_{\Omega\setminus {\mathcal{E}(\bar \xi_0, r/2)} \ni \bar x \to \bar \zeta}w(\bar y)=w(\bar\zeta)\leq 0,
\]
for any $\bar \zeta\in\partial \mathcal{E}(\bar \xi_0, r/2) \cap \Omega$. 

We set $Z$ to be  the set of irregular points of $\partial_e\Omega$, which is polar because of the quasi-regularity assumption on $\Omega$. Thus, if $\bar \zeta \in\partial_n\Omega\setminus Z$ (resp. $\partial_{ss}\Omega\setminus Z$), it holds that 
$
w(\bar x)\to 0
$ as ${\Omega\ni \bar x\to \bar \zeta}$ (resp. $\Omega\ni (x,t)\to (\zeta, \tau^+)$).
Moreover, it is clear that 
$$
\limsup_{\bx \to \bar \zeta \in Z} w(\bx) <\infty.
$$
Hence, by the maximum principle \cite[Theorem 8.2]{Watson} in $\Omega\setminus \overline{\mathcal{E}(\bar \xi_0, r/2)}$, we infer that  $w\leq 0$ in $\Omega\setminus {\mathcal{E}(\bar \xi_0, r/2)}$ concluding the proof of the lemma.

If we assume the TBCDC,  \eqref{eq:estim_green_function_CDC} readily follows from \eqref{eq:estim_green_function} using \eqref{eq:bourgain_estimate}.
\end{proof}

\vv

A pretty standard result in elliptic theory is that  polar sets have zero harmonic measure. The same is true for  caloric measure as well. As we were not able to find an appropriate reference, we will write the  proof for completeness.

\begin{proposition}\label{proposition:polar_sets}
Let $\Omega$ be an open set, $\bx\in\Omega$, and  $\omega^{\bx}$ be the caloric measure in $\om$.
If $E\subset\Rn1$ is polar, then  $\omega^{\bx}(E)=0$.
\end{proposition}
\begin{proof}
Fix $\bx \in \om$.  Since $E\cap \partial\Omega$ is polar as a subset of the polar set $E$, without loss of generality,  we may assume that $ E\subset\partial\Omega$. As we have that $\Cap(E)=0$, by \eqref{def:capacity-minus}, there exists a sequence of open sets $S_j \supset E $ such that  $\lim_{j\to \infty}\Cap_{-}(S_j)=0$. Set now $\widetilde{E}\coloneqq \bigcap_j S_j$, and note  that $\widetilde E$ is Borel, $E\subset \widetilde{E}$, and $\Cap(\widetilde{E})=0$.  Therefore, without loss of generality, we may  assume that $E$  is Borel. 

Let us assume that  $\Cap (\om^c)>0$. Then, there exists an open set $\widetilde \om$ such that $\om \cup E \subset \widetilde \om$ and $\Cap(\widetilde \om^c)>0$.   Indeed, as $\om^c \setminus E$ is Borel, by \eqref{def:capacity-minus}, there exists a compact set $K \subset \om^c \setminus E$  so that $\Cap(K)> \Cap( \om^c \setminus E)/2$.  We set $\widetilde \om= K^c$, which is an open set such that  $\om \cup E \subset \widetilde \om$, and, since  $\Cap(E)=0$, it clearly satisfies  $\Cap (\widetilde \om^c)>0$. If $\Cap (\om^c)=0$,   we simply assume that $\widetilde \om =\rrn$.

 Now, we apply \cite[Theorem 7.3]{Watson} to find  $u \geq 0$, a supertemperature in $\widetilde \Omega$, such that $u=\infty$ on $E$ and $u(\bx)<+\infty$.
Hence, $\lambda u|_{\om}$ belongs to the upper class $\mathfrak U_E$ for every $\lambda>0$, which implies that $0\leq \hm^{\bx}(E) \leq \lambda u(\bx)$. We conclude the proof by taking $\lambda \to 0$.
\end{proof}

\vv

\begin{lemma}\label{lem:quasireg-subdomain}
Let $\Omega_{1}$ and $\Omega_{2}$ be two disjoint quasi-regular open sets  in $\bR^{n+1}$ for $H$ and $H^*$ so that $(\mathcal{P}\Omega_{i})^0 \neq \emptyset$, $i=1,2$, with caloric measures $\omega_{i}=\omega_{\Omega_{i}}^{\bar{p}_{i}}$ for some $\bar{p}_{i}\in\Omega_{i}$ and suppose $\omega_{1}\ll \omega_{2}\ll \omega_{1}$ on a Borel set $E \subset (\mathcal{P}\Omega_{1})^0 \cap (\mathcal{P}\Omega_{2})^0$ and $\hm^{\bar p}_{\om}(E)>0$. Then there are  open subsets $\widetilde{\Omega}_{i}\subset \Omega_{i}$ which are regular for $H$ and $H^*$ and contain $\bar{p}_{i}$ for $i=1,2$ so that there exists $G_{0}\subset E \cap (\mathcal{P}\wt \Omega_{1})^0 \cap (\mathcal{P}\wt \Omega_{2})^0$ with $\widetilde{\omega}_{i}(G_{0})>0$ upon which $\widetilde{\omega}_{2}\ll \widetilde{\omega}_{1}\ll \widetilde{\omega}_{2}$. 
\end{lemma}

\begin{proof}

We follow closely  the proof of \cite[Lemma 2.3]{AMTV16}. Let $F_{i}$ and  $F_{i}^*$ be the sets of irregular points for $\Omega_{i}$ for $H$ and $H^*$ respectively, which are polar and co-polar sets in $\bR^{n+1}$. Although, since by \cite[Theorem 7.46]{Watson} polar and co-polar sets coincide, by \cite[Definition 7.1]{Watson}, there is a positive supertemperature  $v_{i}$ on $\Omega_{i}$ so that 
\[
\lim_{\by\rightarrow \bx \atop \by \in \Omega_{i}} v_{i}(\by)=\infty \;\; \mbox{for all }\bx\in F:=F_{1}\cup F_{2}\cup F_{1}^*\cup F_{2}^*.\]


Let $\lambda>0$. Since $v_{i}$ is supertemperature on $\Omega_{i}$, it is lower semicontinuous, and remains so when we extend it by zero to $\Omega_{i}^{c}$. Thus, for each $\bx\in F$ there is a closed cylinder  $C'_{i}(x)$ centered at $\bar x$ (not containing either $\bar p_{i}$) such that $v_{i}\geq \lambda$ on $C'_{i}(\bx)\cap \Omega_{i}$. For each $C'_{i}(x)$ let $C_{i}(x)$ be the cylinder of the same center and half the radius of $C'_{i}(x)$.  Let $\{C_{j}\}$ be a Besicovitch subcovering (see Lemma \ref{lem:Besicovitch}) and if $E_j$ is the  closed ellipsoid of revolution around the axis of the cylinder $C_j'$ that is inscribed in $C_j'$, we define 
$$
\wt \om_i=\om_i \setminus H\subset \om_i, \,\,\textup{ where}\,\, H=\bigcup_{j \geq 1} E_j.
$$

Note that $\wt\Omega_{i}$ is open. Indeed, to show that $\widetilde{\Omega}_{i}^{c}$ is closed consider $\bx_{k}\in \wt\Omega_{i}^{c}$, $k\geq 1$ and $\bx_{k}\rightarrow \bx$. Then we need to show $\bx\in \widetilde{\Omega}_{i}^{c}$. If there is a subsequence contained in $\Omega_{i}^{c}$, we are done. Otherwise, assume that $\bx_{k}\in H\backslash \Omega_{i}^{c}=H\cap \Omega_{i}$. If $\bx_{k}\in E_{j}$ for infinitely many $k$, then $\bx\in E_{j}$ and we are done since $E_{j}$ is closed and $E_{j}\subset \widetilde{\Omega}_{i}^{c}$. Otherwise, suppose $\bx_{k}$ is not in any $E_{j}$ more than finitely many times. By the bounded overlap property, if $j(\bx_k)$ is such that $\bx_k\in E_{j(\bx_k)}$, then $\diam(E_{j(\bx_k)})\downarrow 0$ as $k\to\infty$, and since the ellipsoids are centered on $F\subset \Omega_{i}^{c}$, we have that $\bx\in \Omega_{i}^{c}\subset \widetilde{\Omega}_{i}^{c}$, and we are done. Thus, $\widetilde{\Omega}_{i}$ is open. 


Set $V_j=\rrn \setminus E_j$ and note that if $\bz_j^1=(z_1,t_1)$ and $\bz_j^2=(z_2,t_2)$, are the points on  $\d E_j \cap \ell_j$, where $\ell_j$ is the line containing the axis of $C_j$, so that $t_1<t_2$,  it holds that $\bz_1 \in \d_{ss} V_j$ and every $\bx \in \d E_j\setminus \{\bz_j^1\}  \subset \PP V_j$ is in $\d_n V_j$. Moreover, it is clear by \eqref{eq:wiener_test} and Remark \ref{rem:regular} that  every $\bx \in \d E_j$ is regular for $V_j$. Therefore, by \cite[Corollary 8.47]{Watson}, it holds that $\d E_j \cap \om_i$ consists of regular points of $\d_e \wt \om_i$. As the points of $\d_e  \om_i \cap \d_e \wt \om_i$ are regular for $\om_i$ then, by another application of \cite[Corollary 8.47]{Watson}, they are regular for $\wt \om_i$ as well. Therefore,  $\wt \om_i$ is   regular for $H$ and $H^*$, and $E \subset \d_e \wt \om_i$. In fact, since $E \subset (\mathcal{P}\Omega_{1})^0 \cap (\mathcal{P}\Omega_{2})^0$ and $\d_a \wt \om_i  \cap \d \om_i\subset \d_a  \om_i$,  then it holds that 
$$
E \cap \d_e \wt \om_i \subset  \mathcal{P}\wt \Omega_{1} \cap \mathcal{P}\wt \Omega_{2} \cap (\mathcal{P}\Omega_{1})^0 \cap (\mathcal{P}\Omega_{2})^0.
$$

Let $\widetilde{\omega}_{i} = \omega_{\widetilde{\Omega}_{i}}^{ \bar{p}_{i}}$ and  $G=E   \backslash H$.  As the cylinders are closed, it is not hard to see that $G \subset (\mathcal{P}\wt \Omega_{1})^0 \cap (\mathcal{P}\wt \Omega_{2})^0$. Since $v_i$ is positive in $\Omega_i$ and $\omega_{i}^x(H) \leq 1 \leq v_i(\bx)/\lambda$ for all $\bx \in H$, by the maximum principle on $\widetilde{\Omega}_{i}$, we have that $\omega_{i}(H)\leq \lambda^{-1} v_{i}( \bar{p}_{i})$. Picking
\[
\lambda^{-1}  < \frac{1}{2} \min_{i=1,2}\{\omega_{i}(E)/v_{i}( \bar{p}_{i})\},\] 
this gives $\omega_{i}(H)\leq \frac{1}{2} \omega_{i}(E)$ and hence $\omega_{i}(G)>0$. 
Similarly,  by the maximum principle, since $v_i/\lambda \geq 1$ on $H$, $\widetilde{\omega}_{i}(H)\leq \lambda^{-1}  v_{i}( \bar{p}_{i})$. Picking 
\[
\lambda^{-1}  < \frac{1}{2} \min_{i=1,2}\{\omega_{i}(G)/v_{i}( \bar{p}_{i})\},\] 
this gives 
\[
\widetilde{\omega}_{i}(H)\leq \frac{1}{2} \omega_{i}(G).\]

Moreover, by the maximum principle, and since $\widetilde{\Omega}_{i}$ is a regular domain,
\[
\widetilde{\omega}_{i}(H^{c}\cap G^{c})\leq {\omega}_{i}(H^{c}\cap G^{c}).\]
Thus,
\begin{align*}
\widetilde{\omega}_{i}(G)
& =1- \widetilde{\omega}_{i}(G^{c})
=1-\widetilde{\omega}_{i}(H\cap G^{c})-\widetilde{\omega}_{i}(H^{c}\cap G^{c}) \\
& \geq 1-\frac{1}{2} \omega_{i}(G)-{\omega}_{i}(H^{c}\cap G^{c})
\geq \omega_{i}(G)-\frac{1}{2} \omega_{i}(G)= \frac{1}{2} \omega_{i}(G)>0.
\end{align*}

Note that $\widetilde{\omega}_{1}\ll \omega_{1}$ on $G$ by the maximum principle and since $\widetilde{\omega}_{1}(G)>0$, it is not hard to show using the Lebesgue decomposition theorem that there is $G_{1}\subset G$ of full $\widetilde{\omega}_{1}$-measure upon which we also have $ \omega_{1} \ll \widetilde{\omega}_{1}$. Hence $\omega_{1}(G_{1})>0$, which implies $\omega_{2}(G_{1})>0$. The same reasoning gives us a set $G_{2}\subset G_{1}$ upon which $\widetilde{\omega}_{2} \ll \omega_{2} \ll \widetilde{\omega}_{2}$. Thus, $\widetilde{\omega}_{2}\ll \widetilde{\omega}_{1}\ll \widetilde{\omega}_{2}$ on $G_{2}$. 

\end{proof}

\vv

The  notion of halving metric space was first introduced by Korey  \cite{Kor98} and plays an important role in the proofs of our theorems.
\begin{definition}
A probability space $(X,\omega)$ is called \textit{halving} if for every subset $E\subset X$ such that $\hm(E)>0$, there exists $F\subset E$ such that $\omega(F)=\omega(E)/2$. 
\end{definition}

We conclude this section by proving   the caloric analogue of  \cite[Lemma 7.9]{AM19}.
\begin{lemma}
Let $\Omega\subset\Rn1$ be an open  set and  $\bx\in\Omega$. Then, the probability space $(\d_e\om, \omega^{\bx})$   is halving.
\end{lemma}
\begin{proof}
Set $\omega\coloneqq\omega^{\bx}_\Omega$ and let us assume that there exists $E\subset \partial\Omega$ such that $\omega(E)>0$ and $\omega(F)\neq \omega(E)/2$ for every $F\subset E$. Given $s\in \R$ and $v\in \mathbb S^n$, we define the half-space $H_{s,v}\coloneqq\{\bar\xi\in\Rn1: \bar\xi\cdot v\geq s \}$.
By the mean value theorem and our assumption, the map $s\mapsto \omega(H_{s,v}\cap E)$ is not continuous for any $v \in \mathbb{S}^n$. In particular, for any $v \in \mathbb{S}^n$, there exists $s_v$, such that $\omega(\partial H_{{s_v},v}\cap E)>0$. Set now $V_v\coloneqq \partial H_{s_v,v}$, which is an  $n$-dimensional plane. Define
\[
S\coloneqq \mathbb \{(y',y_n,\tau)\in\mathbb S^n:y'=0\},
\]
and observe that, since $S$ is uncountable, there is $\varepsilon>0$ such that $\omega(V_v\cap E)>\varepsilon$ for all $v$ in some uncountable subset $A$ of $S$. Let us consider $A'\subset A$ countable. For every $v,v'\in A'$  we have that $V_v\cap V_{v'}$ is an $(n-1)$-plane orthogonal to the $t$ axis, which is a polar set since it is contained in a horizontal hyperplane and has $n$-dimensional Lebesgue measure zero (see \cite[Theorem 7.55]{Watson}). Therefore,  by Proposition \ref{proposition:polar_sets}, it holds that $\omega(V_v\cap V_{v'})=0$. Set  now 
$$
W_u\coloneqq V_u\setminus \bigcup_{v\in A', v\neq u}V_v,
$$
and notice that $\omega(W_u\cap E)=\omega(V_u\cap E)>\varepsilon$ and  $W_u\cap W_{u'}=\varnothing$ for any $u\neq u' \in A'$. Thus, since $A'$ is countable, 
\[
1\geq \omega (E) \geq \sum_{u\in A'}\omega(W_u\cap E)=+ \infty,
\]
which is a contradiction.
\end{proof}

\vvv

Every result related to the  heat equation we have stated  so far has a dual for the adjoint heat equation $H^* u=0$ obtained by reversing the sign of the time variable. Therefore, we can define the associated fundamental solution $\Gamma^*$, the Green function $G_\om^*$, the adjoint caloric measure  denoted by $\hm_*^{\bar p}$, the associated parabolic capacity $\Cap_*$ and so forth. A solution of $H^* u =0$ is called adjoint caloric. Remark that, by \cite[Theorem 6.10]{Watson}, we have that
$$
\Gamma(\bx, \by)= \Gamma^*(\by, \bx) \quad \textup{and}\quad G_\om(\bx, \by)= G_\om^*(\by, \bx).
$$
For the regularity of the $\d_e^*\om$ for $H^*$ and the corresponding capacity density conditions, it is pretty clear that we should just take time-forwards cylinders and the so-called co-heat balls, which are defined as the heat balls using the adjoint heat kernel.

\vvv

\section{Hausdorff and tangent measures}\label{sec:GMT}

If $d \leq n+1$ is an integrer and  $\cH^d(E)=0$ (resp. $\cH^d(E)<\infty$), then  $\cH^{d+1}_{p} (E)=0$ (resp, $\cH^{d+1}_{p}(E)<\infty$). For a proof see  \cite[Lemma 3.2]{He17}. Note that this was originally stated under the additional hypothesis that $E$ is Euclidean $d$-rectifiable but   it is easy to see it is redundant. Moreover, by  \cite[Lemma 3.8]{He17}, it holds that there exists $c_1>0$ and $c_2>0$ depending only on $d$, such that on $\R^{d-1} \times \R \subset \R^n \times \R$, 
\[
\cH^{d+1}_{p}= c_2 \, \mathcal{L}^{d}=c_1 \cH^{d-1} \times \cH^2_{p}.
\]

%
We recall that the Hausdorff dimension of a Borel measure $\omega$  is defined by 
\[
\dim_{\mathcal H_p}(\omega)= \inf\bigl\{\dim_{\mathcal H_p}(Z):\, \omega(\Rn1\setminus Z)=0\bigr\}.
\]
This definition is related with the concept of \textit{pointwise dimension} of $\omega$ at $\bar x\in\supp\omega$. More specifically,  if
\[
\underline{ d}_\mu(\bar x)=\liminf_{r\to 0}\frac{\log\mu(C_r(\bar x))}{\log r}\, \, \text{ and }\, \, \overline d_\mu (\bar x)=\limsup_{r\to 0} \frac{\log \mu(C_r(\bar x))}{\log r}
\]
denote the \textit{lower} and \textit{upper pointwise dimension} respectively, we can argue as in \cite{BW} to show that
\[
\dim_{\mathcal H_p}(\omega)=\esssup\{\underline{d}_\mu(\bar x): \bar x\in\supp \mu\}.
\]
If $\underline d_\mu(\bar x)=\overline d_\mu (\bar x)$, we denote the common value by $d_\mu(\bar x)$.
\vv

\begin{definition}\label{def:Lip1,l}
If $\ell \in [\frac{1}{2},1]$, we say that a function $\psi\colon\R^d\to\R$ is {\it $(1,\ell)$-Lipschitz} and  write $ f \in \Lip_{1,\ell}(\R^d)$ if there exists a constant $L>0$ such that
\begin{equation}\label{eq:def_lip112}
|\psi(x',t)-\psi(y',s)|\leq L \max\big(|x'-y'|,|t-s|^{\ell}\big), \quad \textup{ for}\,\, (x',t)\neq (y',s)\in\R^{d-1}\times \R.
\end{equation}
We call $L$ the {\it Lipschitz constant} of $f $ and write $\Lip(f)=L$. 
\end{definition}
Observe that for $\ell=1/2$ these functions correspond to Lipschitz functions with respect to the parabolic norm $\|\cdot\|$ or the equivalent  norm $ |x'-y'|+|t-s|^{1/2}$.  For $\ell=1$, those are just the usual Lipschitz functions with respect to the Euclidean norm.

\begin{definition}\label{def:admissibility}
We say that $\Gamma\subset \Rn1$ is an \textit{admissible $d$-dimensional graph} if there exists a vector field $f\colon \R^d \to\R^{n-d+1}$ such that, possibly after a rotation in  space  and a translation,
\[
\Gamma = \Gamma_f=\bigl\{\bigl(x',f(x',t),t)\bigr):x\in\R^{d-1},t\in\R\bigr\}.
\]
If $f\in C^m(\R^d ; \R^{n-d+1})$, $1\leq m\leq \infty$, we say that $\Gamma_f$ is an \textit{admissible} $d$-dimensional $C^m$-graph. If $f$ is  an affine map, then  $\Gamma_f = V $ is a plane that contains a line parallel to the time-axis and we call it an  \textit{admissible} $d$-dimensional plane. In fact, after rotation in space and a translation, we can always assume $V=\R^{d-1}\times \{ \vec 0 \} \times \R$, where $\vec 0 = (0,\dots,0) \in \R^{n-d+1}$.
\end{definition}

\begin{definition}\label{def:rectifiable}
A closed set $E \subset \rrn$ is {\it Euclidean $d$-rectifiable} if there exists $E_i\subset \R^d$ and $\Lip_{(1,1)}$-functions $f_i\colon E_i\to\rrn$ so that $\cH^d\bigl(E\setminus \cup^\infty_{i=1} f_i(E_i)\bigr)=0$. Equivalently, $E$ is Euclidean $d$-rectifiable if  there exists a countable collection of  $d$-dimensional $\Lip_{(1,1)}$-graphs (or $d$-dimensional $C^1$-manifolds) $\{\Gamma_i\}_{i \geq 1}$ such that $\cH^d\bigl(E\setminus \cup^\infty_{i=1} \Gamma_i\bigr)=0$. 
\end{definition}

By  the co-area formula \cite[Theorem D]{He17},  if $E$ is Euclidean $d$-rectifiable, { $\HH^d|_E$ is locally finite}, and $\sigma$ {indicates its associated surface measure as  defined in Section \ref{sec:preliminaries}}, then 
\begin{equation}\label{eq:sigma=H}
\sigma= c \HH^{d+1}_p|_E.
\end{equation}
 In fact,  using the Rademacher theorem in \cite{Or20}, one can show that this theorem is true for parabollically rectifiable sets as well, that is, for sets that are exhausted, up to a set of $\HH^{d+1}_p$-measure zero, by admissible  $\Lip_{1,\frac{1}{2}}$-graphs with $\frac{1}{2}$-derivative in time in $\BMO$.\footnote{We omit the detailed definitions and the proof since we will not be dealing with such general sets in the present paper.}

\vv

If  $\mu$ and $\nu$ are Radon measures on $\Rn1$, {following \cite{Pr87}} we define the {\it distance between $\mu$ and $\nu$} in the parabolic ball $C_r$ by
\begin{equation}\label{eq:funtional_F_def}
d_{C_r}(\mu,\nu)=\sup_{f} \int f\, d(\mu-\nu),
\end{equation}
where  the supremum is taken over all functions $ f \in \Lip_{(1,\frac{1}{2})}(\rrn)$ which are supported in $C_r$ and satisfy $\Lip(f) \leq 1$.
If a sequence of Radon measures $\mu_j$ converges weakly to a Radon measure $\mu$, we use the notation $\mu_j\rightharpoonup \mu$.

For $r>0$ and $\bx,\by\in \rrn$, set
\[
\delta_r(\bx)\coloneqq (r x, r^2 t), \quad\textup{and}\quad
T_{\by, r}(\bx)\coloneqq \delta_{1/r}(\bx-\by).
\]
If $\mu$ and $\nu$ are  Radon measures in $\rrn$, we define
\[
T_{\by,r}[\mu](A)\coloneqq \mu(\delta_r(A)+\by)=\mu\bigl(T^{-1}_{\by,r}(A)\bigr), \quad A\subset\rrn.
\]
and 
\[
F_r(\mu,\nu)\coloneqq \sup_f \int f\, d(\mu-\nu), \quad \textup{for}\,\, r>0,
\]
where  the supremum is taken over all functions $ f \in \Lip_{(1,\frac{1}{2})}(\rrn)$ which are supported in $C_r$ and satisfy $\Lip(f) \leq 1$.
By density, it is enough to consider the supremum in the class of $C^\infty_c(C_r)$ functions such that $\Lip(f) \leq 1$.

We also define $F_r(\mu)\coloneqq F_r(\mu,\bar 0).$ A standard argument shows that 
\[
F_r(\mu)=\int \dist_p(\bx, \rrn\setminus C_r)\,d\mu(\bx) =\int^r_0 \mu(C_s)\, ds.
\]
As it is easy to see that
\[
\dist_p\bigl(\bx, \rrn\setminus C_r\bigr)=\bigl(r-\|\bx\|\bigr)_+,
\]
where $(\cdot)_+$ stands for the positive part of a function, we infer that
\[
F_r(\mu)=\int \bigl(r-\|\bx\|\bigr)_+\, d\mu(\bx).
\]

\begin{definition}[$d$-cone]
A set of Radon measures $\mathcal M$ is a $d$-cone if $c\,T_{\bar 0,r}[\mu]\in \mathcal M$ for all $\mu\in\mathcal M$, $c, r>0$.
Given a $d$-cone $\mathcal M$, the set $\{\mu\in\mathcal M:F_1(\mu)=1\}$ is referred to as its \textit{basis}.
We say that $\mathcal M$ has \textit{closed} (resp. \textit{compact}) \textit{basis} if its
basis is closed (resp. compact) with respect to the weak topology of the space of Radon measures.
\end{definition}
For a $d$-cone $\mathcal M$, $r>0,$ and a Radon measure $\mu$ such that $F_r(\mu)\in(0,\infty)$, we define
\begin{equation}\label{eq:def_dist_dcones}
d_r(\mu,\mathcal M)\coloneqq \inf\Bigl\{F_r\Bigl(\frac{\mu}{F_r(\mu)},\nu\Bigr):\nu\in \mathcal M, \,F_r(\nu)=1\Bigr\}.
\end{equation}
The next lemma collects some of the relevant properties of $F_r$ and $d_r(\cdot, \mathcal M).$ For more details, see \cite[Section 2]{KPT09} and the references therein.

\begin{lemma}\label{lem:prop_Fr}
Let $\mu, \nu$ be Radon measures in $\Rn1$, $\bar \xi \in \Rn1$ and $r>0$. The following properties hold:
\begin{enumerate}
\item $\label{eq:FrF1}
F_r(\mu)=rF_1\bigl(T_{\bar 0,r}[\mu]\bigr)$.
\item $\label{eq:Fr-mu(B)}
\tfrac{r}{2} \mu(C_{r/2}) \leq F_r(\mu) \leq r \mu(C_{r}).$
\item $\mu_j\rightharpoonup \mu$ if and only if $F_r(\mu_j,\mu)\to 0$ for all $r>0$.
\item $d_r(\mu,\mathcal M)\leq 1$ and $d_r(\mu,\mathcal M)=d_1(T_{\bar 0,r}[\mu],\mathcal M).$
\item if $\mu_j\rightharpoonup \mu$ and $F_r(\mu)>0$, then $d_r(\mu_j,\mathcal M)\to d_r(\mu,\mathcal M)$.
\end{enumerate} 
\end{lemma}

\begin{definition}
We say that $\nu$ is a {\it tangent measure} of $\mu$ at a point $\bx \in\rrn$ if
$\nu$ is a non-zero Radon measure on $\rrn$ and there are sequences $c_{i}>0$ and $r_{i}\searrow 0$ so that $c_i\,T_{\bx,r_{i}}[\mu]$ converges weakly to $\nu$ as $i\to\infty$ and write $\nu\in \Tan(\mu,\bx)$.
\end{definition}

\begin{remark}\label{rem:Besicovitch}
A Besicovitch covering theorem for parabolic balls in $\rrn$ was proved in \cite[Theorem 1.1]{It18}. This is an important tool for parabolic geometric measure theory. In particular, one can show that Radon measures in $\rrn$ satisfy the Lebesgue density  theorems and the Lebesgue differentiation theorems with respect to parabolic balls, as reported in the next lemma.
\end{remark}

\begin{lemma}\label{lem:Besicovitch}
Let $\mu$ be a Radon measure on $\Rn1$. 
If $f\colon\mathbb{R}^{n+1}\to\R\cup\{\infty\}$ is locally $\mu$-integrable, then 
\[
f(\bar x)=\lim_{r \to 0}\frac{1}{\mu(C_r(\bar x))}\int_{C_r(\bar x)}f\, d\mu\, \qquad \text{ for }\mu\text{-a.e}\, \bar x\in\Rn1.
\]
Furthermore, if $E\subset\Rn1$ is $\mu$-measurable, then the limit
\[
\lim_{r \to 0} \frac{\mu(E\cap C_r(\bar x))}{\mu(C_r(\bar x))}
\]
exists and equals $1$ for $\mu$-a.e. $\bar x\in E$ and  $0$ for $\mu$-a.e. $\bar x\in \Rn1\setminus E$.
\end{lemma}
\begin{proof}
The proof follows from the argument in \cite[Corollary 2.14]{Mattila} using the Besicovitch theorem for parabolic balls in \cite{It18}.
\end{proof}

\begin{remark}\label{rem:tangent measures}
Once we have made the appropriate modifications in the definitions of the blow-up mappings to reflect the parabolic dilation, all the theorems related to tangent measures  that are required to obtain our results hold with the same proofs as in the Euclidean setting.
\end{remark}

\vv

If $\mu$ is a Radon measure, $s \in [0, \infty)$,  and $\bx\in \Rn1$, we define the \textit{lower} and \textit{upper $s$-density} of $\mu$ at $\bx$ as
\[
\Theta^s_{\mu,*}(\bx)\coloneqq \liminf_{r\to 0}\frac{\mu(C_r(\bx))}{r^s} \qquad \textup{ and }\qquad \Theta^{s,*}_\mu (\bx)\coloneqq \limsup_{r\to 0}\frac{\mu(C_r(\bx))}{r^s}.
\]
A measure $\mu$ is \textit{asymptotically doubling} at $\bx \in \rrn$ if
\[
\limsup_{r\to 0}\frac{\mu(C_{2r}(\bx))}{\mu(C_r(\bx))}<\infty.
\]
Remark that if $0<\Theta^s_{\mu,*}(\bx) \leq \Theta^{s,*}_{\mu}(\bx)< \infty$, then $\mu$ is asymptotically doubling at $\bar x$ since
$$
\limsup_{r\to 0}\frac{\mu(C_{2r}(\bx))}{\mu(C_r(\bx))}= 2^s\, \frac{\Theta^{s,*}_{\mu}(\bx)}{\Theta^s_{\mu,*}(\bx) } < \infty.
$$
\vv

\begin{lemma}\label{lem:tangent_meas_properties}
If $\mu$ is a Radon measure on $\Rn1$ and $\bx \in \rrn$, then the following hold:
\begin{enumerate}
\item If $\nu\in \Tan(\mu,\bx)$,  there exists $\{r_{i} \}_{i \geq 1}$ decreasing to $0$ and $\rho,c>0$ so that 
$$
\frac{T_{\bx, r_{i}}[\mu]}{\mu(B(\bx, r_{i} ))}\rightharpoonup c \, T_{\bar 0,\rho}[\nu] \quad \textup{and}\quad c\,T_{\bar 0,\rho}[\nu](C_1(\bar 0))>0.
$$ 
\item If, additionally, $\mu$ is  asymptotically doubling at $\bx \in \rrn$, then for any $\nu \in \Tan(\mu,\bx)$ there exists $c_1>0$ and a sequence $\{r_i\}_{\i \geq 1}$ decreasing to $0$ such that 
$$
 c_1 \frac{T_{\bar x, r_{i}} [\mu]}{\mu(B( \bar x, r_{i} ) )} \rightharpoonup \nu.
$$
In this case, $ \bar 0 \in \supp \nu$ for all $\nu \in \Tan(\mu,\bx)$.
\item If, additionally, $0<\Theta^s_{\mu,*}(\bx) \leq \Theta^{s,*}_{\mu}(\bx)< \infty$, then for any $\nu \in \Tan(\mu,\bx)$ there exists $c_1>0$ and a sequence $\{r_i\}_{\i \geq 1}$ decreasing to $0$ such that 
$$
 c_2 \frac{T_{\bx, r_{i}} [\mu]}{ r_{i}^s } \rightharpoonup \nu.
$$
\end{enumerate}
\end{lemma}
\begin{proof}
For a proof see  \cite[Theorem 14.3]{Mattila} and  \cite[Remarks 14.4 (1)-(4)]{Mattila}. 
\end{proof}

As a result of Lemma \ref{lem:Besicovitch} we obtain the following localization property of tangent measures.
\begin{lemma}\label{lem:local-tengents}
Let $\mu$ be a Radon measure in $\rrn$ and $f \in L^1(\mu)$ a non-negative Borel function. Then, for $\mu$-a.e. $\bx \in \rrn$, it holds that $\Tan (f \,\mu,\bx)= f(\bx)\Tan(\mu,\bx)$.
\end{lemma}
\begin{proof}
With Lemma  \ref{lem:Besicovitch} at our disposal, we just follow the proof of  \cite[Lemma 3.12]{DeLellis}.
\end{proof}

\begin{lemma}[see \cite{Mattila}, Theorem 14.16]\label{lem:second-tangents}
Let $\mu$ be a Radon measure on $\Rn1$. For $\mu$-a.e. $\bar x\in \Rn1$, if $\nu\in \Tan(\mu, \bar x)$, then  the following properties hold:
\begin{enumerate}
\item $T_{\bar y, r}[\nu]\in \Tan(\mu, \bar x)$ for all $\bar y\in \supp\nu$ and $r>0$.
\item $\Tan(\nu, \bar y)\subset\Tan(\mu, \bar x)$ for all $\bar y\in \supp\nu$.
\end{enumerate}
\end{lemma}

\vv

\begin{lemma}\label{lem:blow-upLipgraph}
Let $\Gamma\subset\Rn1$ be an admissible $d$-dimensional $\Lip_{(1,1)}$-graph\footnote{Lemma  \ref{lem:blow-upLipgraph} can be proved  for parabolic Lipschitz graphs as well. Although, since we will not deal with such general sets in the present manuscript and the proof is more involved, we  skip it.} and let $\mu=\cH^{d+1}_{p}|_\Gamma$. Then, for $\mu$-a.e. ~$\bx\in\Gamma$, there exists a positive constant $c_{\bx}$ and an admissible $d$-plane $V_{\bx}$ passing through the origin such that
\begin{equation}\label{eq:blow-up-Lip}
r^{-d-1}T_{\bx,r}[\mu]\rightharpoonup \cH^{d+1}_{p}|_{V_{\bx}}\qquad \text{ as }r\to 0.
\end{equation}
\end{lemma}

\begin{proof}
By the parabolic Rademacher theorem in \cite{Or20}, one can show that a  parabolic Lipschitz vector field  $\psi: \mathbb{R}^d \to \mathbb{R}^{n+1-d}$ satisfies
\begin{equation}\label{eq:diff_rademacher_par_123}
\frac{|\psi(\by)-\psi(\bx)-A_{\bx}(y-x)|}{\| \by - \bx\|}=\epsilon_{\bx}(|\by - \bx|),
\end{equation}
where $\epsilon_{\bx}(r)\to 0$ as $r \to 0$. In fact, that theorem  is only stated for $d=n$ and globally parabolic Lipschitz functions (see e.g. \cite[Definition 3.1]{Or20}) but  { a} local hypothesis is enough. { To be precise, it is sufficient for the functions to be locally in $\Lip_{(1,1/2)}$ and to satisfy \cite[display (3.4)]{Or20} locally (see \cite[Remark 3.6]{Or20}).} Moreover, one can apply \eqref{eq:diff_rademacher_par_123} to each component of $\psi$ and obtain the result above. Note that $A_{\bx}: \mathbb{R}^d \to \mathbb{R}^{n+1-d}$ is a horizontal linear map. In fact, it is the horizontal differential for the Lipschitz vector field $\psi(\cdot, t):\mathbb{R}^{d-1} \to \mathbb{R}^{n+1-d}$ at $\bar x$. Using this map we can construct the approximating admissible $d$-plane passing through the origin. Namely, since any admissible $\Lip_{(1,1)}$-function {satisfies the aforementioned conditions,}   if we set $V_{\bx}=(y',A_{\bx}(y'), s)$ and follow the proof in \cite[pp. 38-39]{DeLellis} using the parabolic area formula \cite[Theorem 3.64]{He17}, we can prove \eqref{eq:blow-up-Lip}. We skip the details.
\end{proof}

\begin{corollary}\label{cor:tangent-rectifiable}
Let $E \subset \rrn$ { be} such that { $\HH^d|_E$ is locally finite and} $ \sigma(E \setminus \bigcup_{j \geq 1} \Gamma_j ) =0$,  where $\Gamma_j$ are admissible $d$-dimensional $\Lip_{(1,1)}$-graphs and $\sigma$ { is}  the surface measure on $E$. Then, for $\sigma$-a.e. $\bx \in \rrn$, there exists an admissible $d$-dimensional plane $V_{\bx}$ passing through the origin, such that
$$
\Tan(\sigma, \bx ) =  \{ c \HH^{d+1}_p|_{V_{ \bx } } : c>0\}.
$$
\end{corollary}
\begin{proof}
 By Lemma \ref{lem:local-tengents}, we have that for $\HH^{d+1}_{p}$-a.e. $\bx \in \Gamma_j \cap E$, $j \geq 1$, it holds that 
$$
\Tan(\cH^{d+1}_{p}|_E,\bx)= \Tan(\HH^{d+1}_{p}|_{\Gamma_j\cap E},\bx)=\Tan(\HH^{d+1}_{p}|_{\Gamma_j},\bx).
$$
Observe that $\HH^{d+1}_{p}|_{\Gamma_j}(C_r(\bx)) \approx r^{d+1}$ for any $\bar x\in \Gamma_j$ and $r >0$ (the implicit constant depends on the Lipschitz character of $\Gamma_j$). Thus, Lemma \ref{lem:blow-upLipgraph} implies that if $\nu \in \Tan(\HH^{d+1}_{p}|_{\Gamma_j},\bx)$, there exists a positive constant $c_{\bx}$ and  an admissible $d$-dimensional plane $V_{\bx}$ passing through the origin, such that  $ \nu =  c_{\bx}\cH^{d+1}_{p}|_{V_{\bx}}$. Thus,  as $E$ is Euclidean $d$-rectifiable and {$\HH^d|_E$ is locally finite}, it holds that $\sigma=c \cH^{d+1}_{p}|_E$ and the result follows.
\end{proof}

\vvv

\section{Nodal set of caloric functions}\label{sec:nodal_sets}

Given a caloric function $h$, we denote by $\Sigma^h\coloneqq \{h=0\}$ the nodal set of $h$ and by $\sigma_h$ the associated surface measure 
\begin{equation}\label{surface_measure_h}
d\sigma_h = d\mathcal H^{n-1}|_{\Sigma_t} dt.
\end{equation}
In most of the paper the function $h$ is clear from the context, in which case we simply understand $\sigma=\sigma_h$.

If $h$ is a non-zero caloric function on $\Rn1$, by  unique continuation  \cite[Theorem 1.2]{Po96} it cannot vanish in an open set.
Hence, we may apply \cite[Theorem 1.1]{HL94} and  \cite[Proposition 1.2]{HL94} to deduce that  $\Sigma^h$ has locally finite $\cH^n$-measure 
and $\Sigma^h\cap |Dh|^{-1}(0)$ is a (euclidean) $(n-1)$-rectifiable set. In particular, for any cylinder $C_r$ centered on  $\Sigma^h$,  the set $\Sigma^h\cap C_r$ can be  decomposed into a union of an $n$-dimensional $C^1$-submanifold $C_r \cap \Sigma^h\cap \{|Dh|>0\}$ with finite $n$-dimensional Hausdorff measure and a closed set $C_r \cap \Sigma^h\cap \{|Dh|=0\}$ of Hausdorff dimension not larger than $n-1$.  

For our purposes, we need a finer study of this decomposition in terms of admissible graphs and so we define the {\it regular} and the {\it singular set}  of $\Sigma^h$ by 
\begin{align*}
\mR &\coloneqq\bigl\{ \bx \in \Sigma^h :| \partial_t h| + |\nabla h | > 0\bigr\} \\
\mS &\coloneqq\bigl\{ \bx \in \Sigma^h :|\partial_t h| + |\nabla h | = 0\bigr\}.
\end{align*}
Additionally, we set 
\begin{align*}
\mR_x &\coloneqq\bigl\{ \bx \in \mR : |\nabla h | > 0\bigr\} \\
\mR_t& \coloneqq\bigl\{ \bx \in \mR : |\partial_t h|> 0\,\,\textup{and}\,\,|\nabla h | = 0 \bigr\}= \mR\setminus \mR_x,
\end{align*}
to be the {\it space-regular} and the {\it time-regular sets}  respectively.

\vv

\begin{lemma}[Structure of the space-regular set]\label{lem:space_regular_set}
If $h\colon\rrn \to \R$ is a non-zero caloric function, then for every $\by\in \mR_x$ there exists $\rho=\rho(\by)>0$  and an admissible $C^\infty$-graph $\Sigma'_{\by}$ such that $\Sigma^h\cap C_\rho(\by)= \Sigma'_{\by}\cap C_\rho(\by)$.
\end{lemma}

\begin{proof}
The proof is an easy application of the implicit function theorem. 
Indeed, if $|\nabla h(\by)|>0$, we can assume without loss of generality that $\partial_{n} h(\by) \neq  0$. Thus, we can find a cylinder  $C_\rho(\by)$ centered at $\by$ in which $\partial_{n} h \neq 0$ and a smooth function $\vp\colon\R^n\to \R$ such that
\[
\bx=\big(x',\vp(x',t),t\big),\quad  \textup{ for any }\bx\in C_\rho(\bar y)\cap \Sigma^h.
\]
We remark that, despite the implicit function theorem defines $\vp$ just locally, we can extend it to a $C^\infty$-function defined in the whole $\R^n$ (see e.g. \cite[Theorem 5, p. 181]{Stein}). 
\end{proof}

\vv

In general, we cannot expect to express $\mR_t$ locally as an admissible graphs (see also the examples after the next lemma). However, we can still make some general consideration about its geometry and we show that its dimension is lower than that of $\mathcal R_x$.

\begin{lemma}[Structure of the time-regular set]\label{lem:sigma_F}
If $h\colon\rrn \to \R$ is a non-zero caloric function, then for every $\by\in \mR_t$ there exists  $\tilde\rho=\tilde\rho(\by)>0$  and a smooth $(n-1)$-dimensional $C^\infty$-graph $\widetilde \Sigma_{\by}$ such that $\mR_t \cap C_{\tilde\rho}(\by)\subset \widetilde\Sigma_{\by} \cap C_{\tilde\rho}(\by)$. In particular, $\sigma\bigl(\mR_t\bigr)=0$.
\end{lemma}

\begin{proof}
If $\mR_t=\varnothing$ there is nothing to prove.
Fix $\bar y=(y,s) \in \mR_t$, and note that since $\dt h(\bar y) \neq 0$, by the implicit function theorem, we can find $\rho=\rho(\bar y)>0$ such that $C_\rho (\bar y) \cap \Sigma^h$ agrees (possibly after  a rotation in space) with the  graph $(x,\vp(x))$ of a $C^\infty$-function $\vp\colon \R^n\to \R$ inside $C_\rho (\bar y)$. Moreover, since $Hh=0$ the latter implies that $\Delta h(\bar y) \neq 0$. 
Without loss of generality, we assume that $\partial_n^2 h(\bar y) \neq 0$, where $\d_n$ stands for the partial derivative in $x_n$ variable.

If we denote $g\coloneqq \partial_{n} h$ and $\Sigma^g=\{g=0\}$, then $g$ is smooth, $\by=(y',y_n, s) \in \Sigma^g$ and $\partial_{n} g(\by) \neq 0$. Thus, by the implicit function theorem, there exists $\rho' \leq \rho$  such that $C_{\rho'} (\by)\cap \Sigma^g$ is the (possibly rotated) graph $(x',\psi(x',t),t)$ of a smooth function $\psi\colon \R^n\to \R$ and $\by = (y',\psi(y',s), s)$. Let us define
\[
\widetilde h (x',t)\coloneqq h(x',\psi(x',t),t)
\]
 and note that $\widetilde h (y',s)= h(y,s)=0$. So by the chain rule and the fact that $(y,s)\in \mR_t$ (hence $\partial_{n} h(y,s)=0$),
\[
\dt \widetilde h(y',s)=\partial_{n} h(y,s)\dt\psi(y',s)+ \dt h(y,s)=\dt h(y,s) \neq 0.
\]
Hence, as $\wt h$ is clearly  a smooth function,  we can apply the implicit function theorem  at $(y',s)$ and  obtain that there exists a neighborhood of $(y',s)$ and a $C^\infty$ function $\varphi\colon\R^{n-1}\to \R$ such that $\{\widetilde h=0\}$ coincides with the graph $(\cdot, \varphi(\cdot))$ in that neighborhood. More specifically, there exists $\rho''\leq \rho'$ such that $\widetilde\Sigma_{\rho''}\coloneqq\Sigma^h\cap \Sigma^g \cap C_{\rho''}(y,s)$ admits the parametrization
\[
\bigl(x',\psi\bigl(x',\varphi(x')\bigr), \varphi(x')\bigr), \qquad x'\in B_{\rho''}(y').
\]
In particular, $\widetilde\Sigma_{\rho''}$ is euclidean $(n-1)$-rectifiable, so $\mathcal{H}^{n+1}_{p}\bigl(\widetilde{\Sigma}_{\rho''}\bigr)=\mathcal{H}^n\bigl(\widetilde{\Sigma}_{\rho''}\bigr)=0$ and, by the co-area formula \cite[Theorem D]{He17},
\[
\int \cH^{n-1}\bigl(\widetilde\Sigma_{\rho''}|_t\bigr)\, \,d\mathcal H^2_{p}(t)=0,
\]
which, since $\mathcal{H}^{2}_{p}\approx \mathcal L^1,$ proves that $\sigma\bigl(\widetilde \Sigma_{\rho''}\bigr)=0$. It is clear that $\mR_t \cap C_{\rho''}(\by) \subset \widetilde\Sigma_{\rho''}$ and, by a  covering argument, we get that $\sigma\bigl(\mR_t)=0$, as wished.
\end{proof}

\vv

\begin{example} 
a) Let us consider the caloric polynomial $h_1(x_1,x_2,t)=x_1^2+x_2^2+4t.$
The set $\Sigma_{h_1}$ is a rotational paraboloid around the time-axis.
Its time-regular set $\mR_t$ is the singleton $\{\bar 0\}$ and its space-regular set is $\Sigma_{h_1}\setminus\{\bar 0\}.$\\
b) Let $h_2(x_1,x_2,t)=x_1^2+x_2^2-2x_1 x_2+4t,$ whose nodal set is a parabolic cylinder. We have that $\mathcal S=\varnothing$, which gives that $\Sigma_{h_2}=\mR_x\cup \mR_t$ and $\mR_t=\{(x,x,0):x\in\R\}$. So $\mR_t$ is a $1$-dimensional manifold (in particular, a line), in contrast  to what we have for the function $h_1$ of the previous example.
\end{example}

\vvv

\section{Caloric measure associated with a caloric function}\label{sec:green functions}

\begin{lemma}\label{lem:cal-measure-assoc-h}
Let $h$ be a caloric function in $\rrn$ and let $h^+$ and $h^-$ indicate the positive and negative parts of $h$ respectively. There exists a unique Radon measure $\hm_h$ supported on $\{h=0\}$ such that  
\begin{equation}\label{eq:cal-measure-assoc-h}
\int \vp \, d\omega_h = \int  h^+ H^*\vp =  \int h^-H^*\vp= \frac{1}{2}\int |h|H^*\vp, \quad \textup{for}\,\,\vp \in C^\infty_c(\rrn).
\end{equation}
\end{lemma}

\begin{proof}
Let $\tilde{h}^\pm$ be the extension by zero of $h^\pm$ in the complement of $\{h^\pm>0\}$. Then, since $h$ is continuous in $\rrn$, $h^\pm \to 0$ continuously on $\{h=0\}$. Thus,  as $h^\pm$ is caloric in $\{h^\pm>0\}$, we have that   $\tilde{h}^\pm$ is a subcaloric function in $\rrn$ and by  \cite[Theorem 6.28]{Watson}, there exists a unique Radon measure $\omega_{\tilde{h}^\pm}$ such that 
$$
\int \vp \, d\omega_{\tilde{h}^\pm} = \int_{\rrn} \tilde{h}^\pm H^*\vp=\int_{\{h^\pm>0\}} h^\pm\, H^*\vp, \quad \textup{for}\,\,\vp \in C^\infty_c(\rrn).
$$
Now, since $\tilde{h}^\pm$ is caloric in $\{h^\pm>0\}$ and zero in $\{h^\mp>0\}$, we have that $\int \vp \, d\omega_{{h}^\pm} =0$ for $\vp \in C^\infty_c(\{h^\pm>0\})$ and $\vp \in C^\infty_c(\{h^\mp>0\})$,  implying that $\supp \omega_{{h}^\pm} \subset \{h=0\}$. Therefore, since $\mathcal{L}^{n+1}(\Sigma^h)=0$, for any $\vp \in C^\infty_c(\rrn)$, we have that
$$
\int \vp \, d\omega_{\tilde{h}^+}-\int \vp \, d\omega_{\tilde{h}^-}  = \int_{\rrn}  \tilde{h}^+ H^*\vp - \int_{\rrn}   \tilde{h}^- H^*\vp =\int_{\rrn} h\, H^*\vp=0,
$$
and so $\omega_{\tilde{h}^+}=\omega_{\tilde{h}^-}$. Thus, the first two  equalities in  \eqref{eq:cal-measure-assoc-h} hold, while the last one follows by  adding instead of subtracting the two identities above.
\end{proof}

\vv

Given a caloric function $h$, by the discussion  in Section \ref{sec:nodal_sets}, $\Sigma^h$ is smooth away from a euclidean $(n-1)$-rectifiable set, and so the  set $\Omega^\pm=\{h^\pm>0\}$ is a set of locally finite perimeter in $\Rn1$ (see Definition 5.1 and  Theorem 5.23 in \cite{EG92}).   Hence, by {\cite[Theorem 18.11]{Mag12}}, for a.e.~ $t\in\R$, its horizontal section $\Omega^\pm_t$ is a set of locally finite perimeter in $\mathbb R^n$.  In fact, more is true. By Lemma \ref{lem:space_regular_set},  around $\sigma$-a.e. any point, we have that $\Sigma^h$ is  given by an admissible Lipschitz graph, while the rest of the points lie on an $(n-1)$-rectifiable set.  So, for a.e. $t$, $\Sigma_t^h$ is also locally Lipschitz and thus,  its measure theoretic boundary (see \cite[Definition 5.7]{EG92}) coincides with its topological boundary.   Therefore, for a.e.~ $t$ there is a unique measure theoretic outward unit  normal $\nu^\pm_t$ to $\d\Omega^\pm_t$ such that we have the generalized Gauss-Green Theorem 
\begin{equation}\label{gauss-green}
\int_{\Omega^\pm_t} \div \vp\, dx= \int_{\Sigma^h_t}\vp\cdot \nu^\pm_t\, d\cH^{n-1}, \qquad \text{ for all } \vp\in C^1_c(\R^n;\R^n).
\end{equation}
Moreover,  $\nu_t^\pm$ coincides with the usual (geometric) outward unit normal on $\d^* \om^\pm$ (see Definitions 5.4 and 5.6, Theorems 5.15 and 5.16, and Lemma 5.5 in \cite{EG92}) .

\vv

Let us recall the parabolic Cauchy estimates for caloric functions.
\begin{proposition}[see e.g. \cite{HL94}, Proposition 2.1]
\label{proposition:cauchy_estimates}
Let $R>0$ and let $h$ be a caloric function in $C_R$. For $r<R$ and $\alpha\in\mathbb{Z}^n_+$ and any positive integer $\ell$ with $|\alpha|+2\ell=m,$ we have
\begin{equation}\label{eq:cauchy_estimates}
|D^{\alpha,\ell}h(x,t)|	\lesssim_{n,m} (R-r)^{-m} \|h\|_{L^\infty(C_R)}, \qquad \textup{ for all }(x,t)\in C_r.
\end{equation}
\end{proposition}

\vv
\begin{lemma}\label{lem:formula-calmeas-calfunc}
If $h$ is a caloric function in $\rrn$ and let  $\sigma$ be the surface measure on $\Sigma^h$ as defined in \eqref{surface_measure_h}, then $\d_t|h| \in L^\infty_{\loc}(\rrn)$ and for any $\vp \in C^{2,1}_c(\rrn)$, 
\begin{equation} \label{eq:formula-calmeas-calfunc}
\int \vp \, d\hm_h = \int_{\rrn} \d_t  |h|\,\vp\, dxdt -\int_{\Sigma^h} \frac{\partial h}{\partial \nu_t^+}\vp\, d\sigma,
\end{equation}
where $ \frac{\partial h}{\partial \nu_t^+}:=\nu_t^+ \cdot \nabla h \neq 0$ $\sigma$-a.e.  
\end{lemma}

\begin{proof}
If we apply  \eqref{gauss-green}  first to $h(\cdot, t) \nabla \vp(\cdot, t) \in  C^1_c(\R^n;\R^n)$ and then to $\vp(\cdot, t) \nabla h(\cdot, t) \in  C^1_c(\R^n;\R^n)$ in $\Omega_t^\pm$,  for fixed $t$, and integrate in $t$, we obtain 
\begin{equation}\label{eq:IBP-gauss-green}
\begin{split}
\int_{\Omega^\pm}& h H^* \vp\, dxdt -\int_{\Omega^\pm}  h\,\d_t \vp\, dxdt=\int_{\Omega^\pm} h\Delta \vp\, dxdt    \\
&= -\int_{\Omega^\pm} \nabla h\nabla \vp\, dxdt + \int_{\Sigma^h} h \frac{\partial \vp}{\partial \nu_t^\pm}\, d\sigma=\int_{\Omega^\pm} \Delta h\, \vp\, dxdt - \int_{\Sigma^h} \frac{\partial h}{\partial \nu_t^\pm}\vp\, d\sigma\\
&= \int_{\Omega^\pm} \d_t h\, \vp\, dxdt-\int_{\Sigma^h} \frac{\partial h}{\partial \nu_t^\pm}\vp\, d\sigma,
\end{split}
\end{equation}
where we used that $h$ is caloric and $h=0$ on $\Sigma^h$.  As $f(\cdot):=| \cdot|$ is smooth away from $0$ and Lipschitz in $\R$  with $\Lip(f) \leq 1$ and $h \in C^{2,1}(\rrn)$, by Rademacher's theorem, we have that for $\bx \in \rrn \setminus \Sigma^h$, 
\begin{equation}\label{eq:Lipschitz-composition}
\nabla_{x,t} f(h)(\bx)=f'(h(\bx)) \nabla_{x,t} h(\bx),
\end{equation}
which, by \eqref{eq:cauchy_estimates}, implies that $|h|$ is locally Lipschitz (in the euclidean norm) and so $\d_t |h| \in L^\infty_{\loc}(\rrn)$. Thus,  since $h \chi_{\om^\pm} = \pm h^\pm \chi_{\rrn}$ and $\d_t h \chi_{\om^\pm} =\pm \d_th^\pm \chi_{\rrn \setminus \Sigma_h}$, we get that
\begin{align*} 
2 \int &\vp \, d\hm_h \overset{\eqref{eq:cal-measure-assoc-h}}{=} \int_{\Omega^+} h H^* \vp\, dxdt -\int_{\Omega^-} h H^* \vp\, dxdt \\
&\overset{\eqref{eq:IBP-gauss-green}}{=} - \int_{\rrn}  |h|\,\d_t \vp\, dxdt +  \int_{\rrn \setminus \Sigma_h} \d_t |h|\, \vp\, dxdt - \int_{\Sigma^h} \frac{\partial h}{\partial \nu_t^+}\vp\, d\sigma + \int_{\Sigma^h} \frac{\partial h}{\partial \nu_t^-}\vp\, d\sigma\\
&=2\int_{\rrn} \d_t  |h|\, \vp\, dxdt - 2 \int_{\Sigma^h} \frac{\partial h}{\partial \nu_t^+}\vp\, d\sigma,
\end{align*}
where in the last equality  we used $\nu_t^+= -\nu_t^-$ and integrated by parts in $t$ using  that $\mathcal{L}^{n+1}(\Sigma_h)=0$ and $|h|$ is locally Lipschitz in $\rrn$. 
Recall  that the points of $\Sigma^h$ where $\nabla h=0$ are contained in an $(n-1)$-rectifiable set and thus, have $\sigma$-measure zero. As  the tangential component of $\nabla h$ on each slice $\Sigma^h_t$ is zero, we have that $\nabla h = \partial_{\nu_t^+} h$ and so $\partial_{\nu_t^+} h \neq 0$ $\sigma$-almost everywhere.
\end{proof}

\vv

\begin{lemma}\label{lem:Poisson-kernel}
If $h$ is a caloric function  in $\rrn$ and   $\sigma$ is the surface measure on $\Sigma^h$ as defined in \eqref{surface_measure_h}, then we have that
\begin{equation}\label{eq:cal-meas-abs.cont.}
\hm_{h}(A)=- \int_A \frac{\partial h}{\partial \nu^+_t}  \, d\sigma, \quad\textup{for any Borel}\,\, A \subset \Sigma^h,
\end{equation}
where  the associated  Poisson kernel $k_h=-\frac{\partial h}{\partial \nu^+_t}$ is positive $\sigma$-a.e. and in  $L^\infty_{\loc}(\sigma)$. In particular, $\hm_h\ll \sigma$.
\end{lemma}

\begin{proof}
Let $A$ be a compact subset of $\Sigma^h$. Note that since $\sigma$ and $\hm_h$ are Radon measures, it holds that $\hm_h(A)<\infty$ and $\sigma(A)<\infty$. By Urysohn's lemma, we can find a sequence of decreasing functions $\{\vp_j\}_{j=1}^\infty \subset C_c^\infty(\rrn)$, $0\leq \vp_j\leq 1$, so that $\vp_j=1$ on $A$ 
and $\vp_j \to \chi_A$ pointwisely. Let us denote $K_j\coloneqq \supp\varphi_j$ and observe that $K_{j+1} \subset K_j$ for all $j \geq 1$.  Then, by \eqref{eq:formula-calmeas-calfunc}, we have
\begin{equation*}
 \int \vp_j\, d\omega_h =- \int_{\Sigma^h} \frac{\partial h}{\partial \nu^+_t} \,\vp_j \, d\sigma + \int_{\rrn} \d_t |h|\,  \vp_j.
\end{equation*}
As  $\sigma(A)<\infty$ and  $\Sigma^h$ is Euclidean $n$-rectifiable, \eqref{eq:sigma=H} entails $\mathcal H^{n+1}_{p}(A)<\infty$ and consequently $\mathcal L^{n+1}(A)\approx \mathcal H^{n+2}_{p}(A)=0$.  Therefore, since  $\supp \vp_j \subset K_j \subset K_1$ for any $j \geq 1$ and $\dt |h| \in L^\infty(K_1)$, then, by dominated convergence, we have  
\begin{equation}\label{eq:w<<sigma.limit}
 \omega_h(A) = - \lim_{j \to \infty} \int_{\Sigma^h} \frac{\partial h}{\partial \nu^+_t} \, \vp_j \,d\sigma.
\end{equation}
  If $C_\rho$ is a cylinder of radius $\rho \approx \diam K_1$ which is centered at $\Sigma^h$ and satisfies $K_1 \subset C_{\rho/2}$, by \eqref{eq:cauchy_estimates} we have that for $\bx \in \rrn$,
$$
\sup_{j\geq 1}  \left| \frac{\partial h}{\partial \nu^+_t}(\bx) \, \vp_j(\bx)  \right| \lesssim_{ n}  \chi_{K_1}(\bx) \, \sup_{C_{\rho/2}} |\nabla h| \lesssim_{\rho} \chi_{K_1}(\bx)\, \|h\|_{L^\infty(C_\rho)}   \in L^1(\sigma),
$$
since  $ h \in L^\infty(C_\rho)$ and $\sigma$ is locally finite. Thus, by \eqref{eq:w<<sigma.limit} and the dominated convergence theorem, we conclude \eqref{eq:cal-meas-abs.cont.} for compact subsets of $\Sigma^h$. Recall that  $\hm_h$ and $\sigma$ are Radon measures  and so  the  result  for Borel sets follows from  inner regularity (see \cite[Definition 1.5]{Mattila}).
\end{proof}

\begin{remark}
In \cite{Bad11}, \cite{AMT16}, \cite{AMTV16}, and \cite{AM19}, it is stated that if $h$ is harmonic, then the ``Poisson kernel" of $\hm_h$ is $\d_{\nu^+} h^+$. Although, since $h^+$ is not differentiable on $\Sigma^h$, a direct application of  \eqref{gauss-green} is not allowed and one should  obtain $\d_{\nu^+} h$ instead. Nevertheless, this detail does not create any issues in the  proofs of the aforementioned papers. It is only in the proof of \cite[Lemma 4.2]{Bad11} one has to  notice that  $h^\pm$ agrees with $h$ in  $\om^\pm$ and so the conclusion of the lemma is still true.
\end{remark}

\vvv

\section{Caloric polynomial measures}\label{sec:caloric_polynomials}

{We recall that we write $\Theta$ to denote the set of caloric functions on $\Rn1$ vanishing at $\bar 0$}.
\begin{lemma}\label{lem:homogeneity_cal_pol_measures}
If $h\in \Theta$ and  $\hm_h$  the associated caloric measure, then 
\begin{equation}\label{eq:hom_1}
T_{\bar 0,r}[\hm_h]=r^{n} \hm_{h\circ T^{-1}_{\bar 0,r}}.
\end{equation}
Moreover, if $h\in F(k)$,
\begin{equation}\label{eq:hom_2}
T_{\bar 0,r}[\hm_h]=r^{n+k}\hm_h.
\end{equation} 
\end{lemma}

\begin{proof}
Let $\vp\in C^\infty_c(\rrn)$. An application of the chain rule gives
\begin{equation}\label{eq:lem_homog_147}
\begin{split}
&H^*\bigl(\vp\circ T_{\bar 0,r}\bigr)(\bx)
= \frac{1}{r^2}\bigl((H^*\vp)\circ T_{\zz,r}\bigr)(\bx).
\end{split}
\end{equation}
So, given $h\in \Theta$, we have
\[
\begin{split}
\int \vp \, d T_{\zz,r}[\hm_h]&= \int \vp \circ T_{\zz,r} \,d\hm_h\\
&=\frac{1}{2}\int |h|H^*\bigl(\vp\circ T_{\zz,r}\bigr)\overset{\eqref{eq:lem_homog_147}}{=}\frac{1}{2r^2}\int |h|\bigl((H^*\vp)\circ T_{\zz,r}\bigr)\\
&= \frac{r^{n}}{2}\int \bigl|h\circ T^{-1}_{\zz,r} \bigr| H^*\vp =r^n \int \vp\, d\hm_{h\circ T^{-1}_{\zz,r}},
\end{split}
\]
which proves the first part of the statement. If we further assume that $h\in F(k)$, we obtain
\[
h\circ T^{-1}_{\zz,r}(\bx)=h\bigl(\delta_r (\bx)\bigr)=r^k h(\bx), \qquad \bx\in\rrn,
\]
and  \eqref{eq:hom_2} follows.
\end{proof}

\vv

We recall that we use the notation $C_r\coloneqq C_r(\bar 0)$ for cylinders centered at the origin.
A consequence of the previous lemma is the following corollary.

\begin{corollary}
If $h\in F(k)$ and $r_1, r_2>0$ it holds that
\begin{equation}\label{eq:Fr_hom_pol}
F_{r_1}(\omega_h)=\left(\frac{r_1}{r_2}\right)^{n+k+1}F_{r_2}\bigl(\omega_h\bigr).
\end{equation}
In particular, for any $r>0$ and $M \geq 1$,
\begin{equation}\label{eq:F1_hom_pol}
\frac{M}{2}\left(\frac{r}{M}\right)^{n+k+1}\hm_h(C_{M/2}) \leq F_r \bigl(\omega_h\bigr) \leq M  \left(\frac{r}{M}\right)^{n+k+1} \hm_h(C_M).
\end{equation}
\end{corollary}

\begin{proof}
 By Lemma \ref{lem:prop_Fr} and Lemma \ref{lem:homogeneity_cal_pol_measures} we have that
 $$
 F_r(\omega_h)=r F_1\bigl(T_{\zz,r}[\omega_h]\bigr)=r^{n+1}F_1\bigl(\omega_{h\circ T^{-1}_{\zz,r}}\bigr)= r^{n+k+1}F_1\bigl(\omega_h\bigr),
 $$
 and so \eqref{eq:Fr_hom_pol} readily follows.  By Lemma \ref{lem:prop_Fr}-\eqref{eq:Fr-mu(B)} and \eqref{eq:Fr_hom_pol}, it is straightforward to see that \eqref{eq:F1_hom_pol} holds.
\end{proof}

\vv

\begin{lemma}\label{lem:weak_conv_cal_meas}
Let ${h_j}$ be a sequence in $\Theta$ which converges uniformly on compact sets to some $h\in\Theta$. Then $\omega_{h_j}\rightharpoonup \omega_h$.
\end{lemma}

\begin{proof}
The proof is a minor variant of the one of \cite[Lemma 5.4]{AM19} and we omit it.
\end{proof}

\vv

\begin{lemma}\label{zero_in_support}
Let $h\in \Theta$ and let $\omega_h$ be the associated caloric measure. Then $\zz\in\supp\omega_h$.
\end{lemma}

\begin{proof}
Let us  recall that  $\Sigma^h=\mR_x\cup\mR_t\cup \mathcal S$ (see section \ref{sec:nodal_sets}) and  that the Poisson kernel is given by $-\partial_{\nu_t}h$. Therefore, by Lemma \ref{lem:space_regular_set}, for any $\bx \in \mR_x$ there is a sufficiently small neighborhood of $\bx$ in which $\Sigma^h$ agrees with  an admissible smooth graph. Since $h$ vanishes on $\Sigma^h$, the component of $\nabla h$ which is tangential to $\Sigma^h_t$ is the zero vector and we have  $\nabla h=\partial_{\nu_t}h$. So  $\partial_{\nu_t}h(\bx) \neq 0$ for any $\bx \in \mR_x$ and thus $\bx \in \supp \hm_h$ showing that $\mR_x \subset \overline{\mR_x} \subset \supp \hm_h \subset \Sigma^h$.

In light of  \cite[Theorem 1.1]{HL94} and Lemma \ref{lem:sigma_F}, for $\rho>0$ small enough, it holds that $\mathcal H^{n-1}\bigl((\mR_t\cup \mathcal  S)\cap C_\rho\bigr)<\infty$, and so $\mR_t\cup \mathcal  S$ has empty interior in the relative topology of $\Sigma^h$. Hence,
\[
\bar 0 \in\Sigma^h\cap C_\rho = \overline{\Sigma^h \setminus (\mR_t\cup \mathcal  S)}\cap C_\rho = \overline{\mR_x}\cap C_\rho \subset \supp \hm_h,
\]
which finishes the proof.
\end{proof}

\vv

\begin{lemma}\label{lem:compact_basis}
The $d$-cones $\mathscr{F}(k)$ and $\mathscr{P}(k)$ have compact basis for all $k$.
Moreover, for $h\in \mathscr{P}(k)$ and $r>0$ we have
\[
\|h\|_{L^\infty(C_r)}\approx_k r^{-n-1} F_r(\omega_h).
\]
\end{lemma}

\begin{proof}
In light of Lemma \ref{zero_in_support}, the proof is a routine adaptation of Lemmas 5.5 and 5.6, and Corollary 5.7 in \cite{AM19}.
\end{proof}

\vv

In our argument we need the following formula for the expansion of caloric functions by Han and Lin, which we report for the reader's convenience.
\begin{theorem}[see \cite{HL94}, Theorem 2.2]\label{theorem:han_lin}
Let $R>0$ and let $h$ be a caloric function in $C_R$. Then, for any positive integer $d$,  $0<r<R/2$, and $\bx \in C_r$, we have that
\begin{align}
h(x,t)&=\sum_{j=1}^d \sum_{|\alpha|+2\ell=j}\frac{D^{\alpha, \ell}h(\bar 0)}{\alpha!\ell!}x^\alpha t^\ell + R_d(x,t) \eqqcolon \sum_{j=1}^d h_j + R_d(x,t),\label{eq:caloric-Taylor}
\end{align}
where
\begin{align}\label{eq:caloric-Taylor-error}
|R_d(x,t)|\lesssim_{n,d} \frac{r^{d+1}}{(R-r)^{d+1}} \|h\|_{L^\infty(C_R)}, \qquad \textup{ for }(x,t)\in C_r.
\end{align}
\end{theorem}

\vv

\begin{lemma}\label{lem:finite_upper_density}
Let $h\in \Theta$ and $h_j$ be as in \eqref{eq:caloric-Taylor}.  If $m \geq 1$ is the smallest integer such that $h_m \not\equiv 0$, then there exists $\rho_0>0$ and $c_0>0$ such that
\begin{equation}\label{eq:comparison_fh_fhm}
c_0^{-1}F_r(\omega_{h_m})\leq F_r(\omega_h)\leq c_0 F_r(\omega_{h_m}) \qquad\textup{ for all }r\leq \rho_0.
\end{equation}
\end{lemma}

\begin{proof}
Let $r<1/100$. By Theorem \ref{theorem:han_lin} and the fact that  $h_j=0$ for $j<m$, there exists a function $R_m\colon \Rn1\to\R$ such that
\[
h(\bar x)=h_m (\bar x)+ R_m(\bar x), \qquad \bar x\in C_{2r},
\]
and
\begin{equation}\label{eq:lem_R_d_unif_convergence}
|R_m(\bar x)|\lesssim_{n,m}\frac{r^{m+1}}{(1-2r)^{m+1}}\|h\|_{L^\infty(C_1)}, \qquad \bar x\in C_{2r}.
\end{equation}
Let $\psi_r \in C^{2,1}(C_r)$ be so that $\psi_r=1$ in $C_{r/2}$ and $0\leq \psi_r \leq 1$, satisfying
$$
|\nabla \psi_r| \leq C/r,\quad \textup{and}\quad  |\partial_t \psi_r|+|\partial^2_{x_ix_j} \psi_r| \leq C/r^2, \,\,\textup{for}\,\, 1\leq i,j \leq n.
$$
One way to construct such a cut-off is the following: let $\eta \in C^\infty([0,\infty))$ be the standard cut-off so that $\eta=1$ in $[0,1/2]$, $\eta=0$ in $(1,\infty)$, $0\leq \eta \leq 1$, $|\eta'| \leq C$ and $|\eta''| \leq C$.  If $\zeta_ r(x)= \eta(|x|/r)$ and $\xi_r(t)\coloneqq\eta(|t|/r^2)$, we define $\psi_r(x,t)= \zeta_ r(x)\xi_r(t)$.
It is easy to see that the function $\varphi=(2C)^{-1} r\, \psi_r$ is an admissible function for $F_r$.
With this choice of $\varphi$ we have
\[
\Bigl|\int \varphi\, d\omega_h - \int \varphi\, d\omega_{h_m}\Bigr|= \Bigl|\int_{\Rn1}(|h|-|h_m|)H^*\varphi\Bigr|\leq \frac{1}{r}\int_{C_{2r}}|R_m(x,t)|
\]
and by \eqref{eq:lem_R_d_unif_convergence},
\[
\frac{1}{r}\int_{C_{2r}}|R_m(x,t)|\lesssim_{m,n} \frac{r^{m+n+2}}{(1-2r)^{m+1}}\|h\|_{L^\infty(C_1)}.
\]
Hence, since  $r<1/100$, there is $c(n,m)>0$ such that
\begin{equation}\label{eq:lem_upper_bound_omega_h}
\Bigl|\int \varphi\, d\omega_h - \int \varphi\, d\omega_{h_m}\Bigr|\leq c(n,m)r^{n+m+2}\|h\|_{L^\infty(C_1)}
\end{equation}
and so, if for a fixed $\varepsilon>0$ we further assume 
\[r\leq \rho_0\coloneqq \varepsilon F_1(\omega_{h_m})c(n,m)^{-1}\|h\|^{-1}_{L^\infty(C_1)},
\]
the inequality \eqref{eq:lem_upper_bound_omega_h} reads
\begin{equation}\label{eq:lem_789}
\Bigl|\int \varphi\, d\omega_h - \int \varphi\, d\omega_{h_m}\Bigr|\leq \varepsilon r^{n+m+1}F_1(\omega_{h_m})=\varepsilon F_r(\omega_{h_m}),
\end{equation}
where in the last step we used \eqref{eq:Fr_hom_pol}. Let us observe that, by Lemma \ref{zero_in_support}, $\zz\in \supp\omega_{h_m}$, which, in turn, by \eqref{eq:F1_hom_pol},  implies that $0<F_r(\omega_{h_m})<\infty$ for all $r>0$.

 Therefore, since $\varphi$ is admissible for $F_r$, we obtain
\begin{equation}
\label{eq:lem_comp_h_hm2}
\frac{1}{C 2^{m+n+3}} F_r(\hm_{h_m}) \leq \frac{r}{4C} \hm_{h_m}(C_{r/2}) \leq \int \varphi \,d\hm_{h_m} \overset{\eqref{eq:lem_789}}{\leq} F_r(\hm_h)+\ve F_r(\hm_{h_m})  
\end{equation}
and
\begin{equation}\label{eq:lem_comp_h_hm3}
\begin{split}
(4C)^{-1} F_{r/2}(\hm_{h}) \leq \frac{r}{4C} \hm_h(C_{r/2}) \leq \int \varphi \,d\hm_{h} &\overset{\eqref{eq:lem_789}}{\leq} (1+ \ve) F_r(\hm_{h_m})\\
&=2^{m+n+1} (1+ \ve) F_{r/2}(\hm_{h_m}),
\end{split}
\end{equation}
which conclude \eqref{lem:finite_upper_density} by choosing $\ve = C^{-1} 2^{-n-m-4}$.
\end{proof}

\vv

We apply the previous lemma to show that the first non-zero term in the expansion $h_m$ of $h$ determines the density of $\omega_h$ at $\bar 0$.

\begin{lemma}\label{corollary:bounded_densities}
Let $h,h_m$ and $c_0$ be as in Lemma \ref{lem:finite_upper_density}. Then 
\[
\liminf_{r\to 0}\frac{\omega_h(C_r)}{r^{n+m}}\approx\limsup_{r\to 0}\frac{\omega_h(C_r)}{r^{n+m}}\approx \omega_{h_m}(C_1) \in (0, \infty),
\]
where the implicit constants depend on $n, m$ and $c_0$.
\end{lemma}

\begin{proof}
Let $\rho_0$ and $c_0$ be as in Lemma \ref{lem:finite_upper_density} and let $r<\rho_0/2$. We apply Lemma \ref{lem:prop_Fr}-\eqref{eq:Fr-mu(B)}, \eqref{eq:comparison_fh_fhm}, and \eqref{eq:F1_hom_pol}, and  we have that 
\begin{equation*}
\begin{split}
\frac{\omega_h(C_r)}{r^{n+m}}&\overset{}{\leq} {r^{-n-m-1}}F_{2r}(\omega_h)\overset{}{\leq} {c_0}{r^{-n-m-1}}F_{2r}(\omega_{h_m})\overset{}{=} c_0 2^{n+m+1} F_{1}(\omega_{h_m}) \lesssim \omega_{h_m}(C_1).
\end{split}
\end{equation*}
Arguing similarly but using the converse inequalities of the ones we used  above, we infer that
\begin{equation*}
\begin{split}
\frac{\omega_h(C_r)}{r^{n+m}}&\geq {r^{-n-m-1}}F_r(\omega_h)\geq {c_0^{-1} r^{-n-m-1}}F_r(\omega_{h_m})\\
&= c_0^{-1} 2^{n+m+1} F_{1}(\omega_{h_m}) \gtrsim \omega_{h_m}(C_1)>0,
\end{split}
\end{equation*}
where the last  inequality holds because  $\zz\in\supp\omega_{h_m}$.
\end{proof}

\vv

\begin{lemma}\label{lem:tangent_caloric_function}
Let $h\in \Theta$ and $h_j$ be as in \eqref{eq:caloric-Taylor}.  If $m \geq 1$ is the smallest integer such that $h_m \not\equiv 0$,
then $\Tan (\omega_h, \zz)=\{c\omega_{h_m}:c>0\}$.
\end{lemma}

\begin{proof}
Let $R>0$ and $r<1/2$. By Theorem \ref{theorem:han_lin} and $h_j=0$ for $j<m$, we have that
\[
h(\bx)=h_m (\bx)+ R_m(\bx), \qquad \bx\in C_{rR},
\]
where
\begin{equation}\label{eq:lem_R_d_unif_convergence}
|R_m(\bx)|\lesssim_{n,m}\frac{r^{m+1}}{ R^{m+1}(1-r)^{m+1}}\|h\|_{L^\infty(C_R)}, \qquad \bx\in C_{rR}.
\end{equation}
Given $\by\in C_R$, we have that $\delta_r(\by)\in C_{rR}$, so we combine \eqref{eq:lem_R_d_unif_convergence} with the homogeneity of $h_m$ in order to get
\begin{equation*}
\begin{split}
&|r^{-m}h \circ T_{\zz,r}^{-1}(\by) - h_m(\by)|\\
&\qquad=|r^{-m}h (ry,r^2s) - h_m(y,s)|= r^{-m}|h (ry,r^2s) - h_m(ry,r^2s)|\\
&\qquad\lesssim_{n,m}\frac{r}{R^{m+1}(1-r)^{m+1}}\|h\|_{L^\infty(C_R)}\lesssim_m \frac{r}{R^{m+1}} \|h\|_{L^\infty(C_R)},
\end{split}
\end{equation*}
where the last term converges to $0$ as $r\to 0$. In particular, this shows that $r^{-m}h \circ T_{\zz,r}^{-1}$ converges to $h_m$ uniformly on compact subsets of $\Rn1$.

The definition of caloric polynomial measure together with Lemma \eqref{lem:homogeneity_cal_pol_measures} gives us that 
\begin{equation}\label{eq:lem_tangent_meas_2}
\omega_{r^{-m}h \circ T_{\zz,r}^{-1}} = r^{-m}\omega_{h\circ T^{-1}_{\zz,r}}= r^{-n-m}T_{\zz,r}[\omega_h].
\end{equation}
In order to finish the proof of the lemma, it suffices to use \eqref{eq:lem_tangent_meas_2} and the fact that $\omega_{r^{-m}h \circ T_{\zz,r}^{-1}}$ converges weakly to $\omega_{h_m}$ by Lemma \ref{lem:weak_conv_cal_meas}. Indeed, $\omega_h$ has positive lower  and finite upper $(n+m)$-density at $\bar 0$ by Lemma \ref{corollary:bounded_densities}, so we can apply Lemma \ref{lem:tangent_meas_properties} to conclude that every measure in $\Tan(\omega_h, \zz)$ is of the form $c\omega_{h_m}$ for some $c>0$.
\end{proof}

\vv

The next result is an important application of the previous lemma.

\begin{lemma}\label{lem:tangent_measures_inside_Fk}
Let $\omega$ be a Radon measure on $\Rn1$, $\bx\in\supp\omega$ and let $d$ be the minimal integer such that $\Tan(\omega,\bx)\cap \mathscr{P}(d)\neq \varnothing$.	Then $\Tan(\omega,\bx)\cap \mathscr{P}(d)\subset \mathscr{F}(d).$
\end{lemma}

\begin{proof}
It is enough to argue as in \cite[Lemma 5.9]{AM19}\footnote{The proof of the cited lemma refers to \cite[Theorem 14.16]{Mattila} in order to guarantee that tangent measures of tangent measures are tangent measures themselves. However, that result holds at almost every point. The proof can be fixed invoking \cite[Lemma 2.6]{Bad11}.} and  invoke Lemma \ref{lem:tangent_caloric_function}.
\end{proof}

\vv

\begin{lemma}\label{lem:compare_h_hd}
Let $h\in P(d)$ and assume that 
$$
h(x,t)=\sum_{j=1}^d h_j (x,t)\coloneqq \sum_{j=1}^d \sum_{|k| + 2 \ell = j} a_{k\ell} x^k t^\ell.
$$ 
Then, there exists $r_0>0$ and $C_0 \geq 1$ such that
\begin{equation}\label{eq:Frhm-Frhmd}
C_0^{-1} F_r(\omega_{h_d})   \leq  F_r(\omega_h) \leq C_0 F_r(\omega_{h_d}),  \quad \textup{for all}\,\,r \geq r_0,
\end{equation} 
where $r_0$ depends on $d$, $n$, $F_1(\hm_{h_d})$, $\max_{1\leq j \leq d-1} \sum_{|k|+2\ell=j} |a_{k\ell}|$, and $C_0$ depends on $d$ and $n$.
\end{lemma}

\begin{proof}
Let $\psi_r$ be a cut-off function as in Lemma \ref{lem:finite_upper_density}. In particular, $\psi_r\in C^{2,1}(C_r),$ $0\leq \psi_r\leq 1$, $\psi_r=1$ in $C_{r/2}$ and
$$
|\nabla \psi_r| \leq C/r,\quad \textup{and}\quad  |\partial_t \psi_r|+|\partial^2_{x_ix_j} \psi_r| \leq C/r^2, \,\,\textup{for}\,\, 1\leq i,j \leq n.
$$
Hence, the function $\varphi\coloneqq(2C)^{-1}r\psi_r$ is admissible for the functional $F_r$, $|H^*\varphi|\leq r^{-1}$, and we have
\begin{align*}
\left| \int \varphi \,d\hm_h -  \int \varphi \,d\hm_{h_d} \right| &= \left| \int_{\rrn} (|h| - |h_d|) H^* \varphi \right| \\
&\leq   \frac{1}{r} \int_{C_{2r}}  \bigl||h| - |h_d|\bigr| \leq   \frac{1}{r} \sum_{j=1}^{d-1}  \int_{C_{2r}}  |h_j |.
\end{align*}
For any $\bar x \in C_{2r}$, it holds that
$$
|h_j(\bar x)| \leq  (2r)^j \sum_{|k| + 2 \ell = j} |a_{k\ell}|\eqqcolon c_j r^j,
$$
which, in turn, implies that
\begin{align*}
 r^{-1} \sum_{j=1}^{d-1}  \int_{C_{2r}}  |h_j | &\leq   r^{-1} \,(2r)^{n+2} \, \Bigl(\max_{1 \leq j \leq d-1}c_j \Bigr)\, \sum_{j=1}^{d-1} r^j \leq 2^{n+2} \,  \frac{r^{n+d+1}}{r-1} \max_{1 \leq j \leq d-1}c_j.
\end{align*}
Therefore,  if we let 
\begin{align*}
r_0 &> 1 + (\ve F_1(\hm_{h_d}))^{-1} 2^{n+2}  \max_{1 \leq j \leq d-1}c_j,
\end{align*}
for some $\ve>0$ small enough, we infer that
$$
\left| \int \varphi \,d\hm_h -  \int \varphi \,d\hm_{h_d} \right|\leq \ve r^{d+n+1} F_1(\hm_d)= \ve F_r(\hm_d),
$$
where in the last equality we used \eqref{eq:Fr_hom_pol}. Note that $\bar 0 \in \supp(\hm_{h_d})$ and thus, by \eqref{eq:F1_hom_pol}, we have that $0<F_1(\hm_{h_d})<\infty$. 
We conclude the proof arguing as  in \eqref{eq:lem_comp_h_hm2} and \eqref{eq:lem_comp_h_hm3}. 
\end{proof}

\vv

\begin{lemma}\label{lem:connectivity_in_F_k}
Let $h$, $r_0$ and $C_0$ be as in Lemma \ref{lem:compare_h_hd}. There exists $\varepsilon_0>0$ and $r_1>0$ such that if
$d_r\bigl(\omega_h, \cF(k)\bigr)<\varepsilon_0$ for all $r\geq r_1$, then  $k=d$.
\end{lemma}

\begin{proof}
Fix $\tau \geq 2 $ to be chosen and pick $r\geq r_1$ such that $d_{\tau r}\bigl(\omega_h, \cF(k)\bigr)<\varepsilon_0$. In particular, there exists $\psi \in \cF(k)$ such that $F_{\tau r}(\psi)=1$ and 
\[
F_r\Bigl(\frac{\omega}{F_{\tau r}(\omega_h)}, \psi \Bigr)\leq F_{\tau r}\Bigl(\frac{\omega_h}{F_{\tau r}(\omega)}, \psi\Bigr)<\varepsilon_0.
\]
Hence
\begin{equation}\label{eq:pf_lem_dist_1}
F_r(\psi)-\varepsilon_0<\frac{F_r(\omega_h)}{F_{\tau r}(\omega_h)}<F_r(\psi)+\varepsilon_0.
\end{equation}
Also, since $\psi$ is homogeneous, \eqref{eq:Fr_hom_pol} gives
\begin{equation}\label{eq:pf_lem_dist_3}
F_r(\psi)=r^{n+k+1}F_1(\psi)=\tau^{-n-k-1}F_{\tau r}(\psi)=\tau^{-n-k-1}.
\end{equation}
Assuming that $r>r_2\coloneqq\max\{r_0, r_1\}$, by Lemma \ref{lem:compare_h_hd} and \eqref{eq:Fr_hom_pol} we have
\begin{equation}\label{eq:pf_lem_dis_2}
\frac{F_r(\omega_h)}{F_{\tau r}(\omega_h)}\leq C_0^{2} \frac{F_r(\omega_{h_d})}{F_{\tau r}(\omega_{h_d})}\leq C_0^{2}  \tau^{-n-d-1}.
\end{equation}
Therefore, \eqref{eq:pf_lem_dist_1}, \eqref{eq:pf_lem_dist_3}, and \eqref{eq:pf_lem_dis_2} infer that 
$$\tau^{-n-k-1}-\varepsilon_0< C_0^{2}  \tau^{-n-d-1},$$
or equivalently,
$$
\tau^{d-k} (1 -\varepsilon_0 \tau^{n+d+1})<C_0^{2} .
$$
If we pick $\tau/2=C_0^2$ and $\ve_0=\tau^{-n-d-2}$, the latter inequality implies that 
$$
\tau^{d-k} (\tau - 1)<\frac{\tau^{2}}{2},
$$ 
which can only  hold if $d=k$, since $\tau \geq 2$. 
\end{proof}

\vv

The next lemma is crucial to our purposes: it provides a connectivity result for parabolic cones of Radon measures. The proof  translates unchanged to the parabolic setting and  we skip it.

\begin{lemma}[see Lemma 3.10, \cite{AM19}]\label{lem:connectivity_lemma}
Let $\cF$ and $\cM$ be parabolic $d$-cones and assume that $\cF$ has compact basis. Moreover, suppose that there is $\ve_0>0$ such that the following property holds: if $\mu\in\cM$ and there exists $r_0>0$ such that $d_r(\mu,\cF)\leq \ve_0$ for all $r\geq r_0$, then $\mu\in\cF$.
If $\eta$ is a Radon measure and $\bx\in\supp\eta$ are such that $\Tan(\eta, \bx)\subset\cM$ and $\Tan(\eta, \bx)\cap \cF\neq\varnothing$, then $\Tan(\eta, \bx)\subset \cF.$
\end{lemma}

\vv

The following proposition gathers all the results of this section. After proving all the previous lemmas, the proof is analogous to that of \cite[Proposition II]{AM19}. We report it anyways, in order to give the reader the precise references inside this section. 
\begin{proposition}\label{prop:connect714}
Let $\omega$ be a Radon measure in $\Rn1$ and let $\bx\in\Rn1$ be such that $\Tan(\omega,\bx)\subset \cP(k)$ for some $k$. If $\Tan(\omega,\bx)\cap \cF(k)\neq\varnothing$ for some integer $k$, then $\Tan(\omega,\bx)\subset\cF(k).$
\end{proposition}

\begin{proof}
Let us assume that $\Tan(\omega,\bx)\subset \cP(k)$ and let $m\leq k$ be the smallest integer for which $\Tan(\omega,\bx)\cap \cP(m)\neq \varnothing.$ In particular $\Tan(\omega, \bx)\cap \cP(m)\subset \cF(m)$ by Lemma \ref{lem:tangent_measures_inside_Fk}, which gives that $\Tan(\omega, \bx)\cap \cF(m)\neq \varnothing.$ Furthermore, $\cF(k)$ has compact basis by Lemma \ref{lem:compact_basis} and, in light of Lemma \ref{lem:connectivity_in_F_k}, we can apply Lemma \ref{lem:connectivity_lemma} with $\cM=\cP(k)$, $\cF=\cF(k)$ and $\eta=\omega$. Thus, $\Tan(\omega, \bx)\subset \cF(k).$
\end{proof}

\vvv

\section{Elements of the theory of Optimal Transport}\label{sec:optimal_transport}

In this section, we record and develop  some basic theory of Optimal Transport. In particular,  we are interested in the following:
\begin{enumerate}
\item The relationship between the Kantorovich norm $\| \cdot\|_{KR}$ on the space of  signed measures $\mathcal{M}( \om)$  with not necessarily zero mass and the Wasserstein distance between finite non-negative measures on a bounded  open set $\om \subset \rn$ whose masses are not necessarily equal.
\item Kantorovich duality in the class of finite non-negative measures on $ \om$ whose masses are not necessarily equal.
\item Identifying the space of $1$-Lipschitz functions defined  on $\overline \om$ that vanish on $\d \om$ as the dual of $\left( \mathcal{M}( \om), \| \cdot\|_{KR}\right)$.
\item The identification of the completion of $\mathcal{M}( \om)$ or, equivalently, the closure of $\mathcal{M}( \om)$ as a subspace of $\left(  \Lip(\overline{\Omega} )^*, \| \cdot\|_{KR} \right)$.
\end{enumerate}

\vv

Let   $\om \subset \rn$ be a bounded  open set, $\mathcal M( \om)$  be the set of signed Borel measures on $ \om$ such that $|\mu|( \om)<\infty$, and let $\mathcal M^+( \om)$ be the set of non-negative measures in $\mathcal M( \om)$. Set
$$
\lipp(\oom):= \left\{\vp: \oom \to \R: \vp \in \Lip(\oom),  \|\vp\|_{\Lip}\leq 1 \text{ and }  \vp(x)=0 \text{ on } \d \om \right\}
$$
and note that  $\dist( \cdot, \rn \setminus \om) \in  \lipp(\overline \om)$ and $ \sup_{x \in \oom} |\vp(x)|\leq \diam(\om)$, for any $\vp \in \lipp(\oom)$. We define  the  Kantorovich-Rubinshtein norm on $\mathcal M(\om)$  by
\begin{equation}\label{eq:KRnorm}
\| \mu \|_{KR(\om)}:= \sup \left\{ \int \vp \,d\mu : \vp \in \lipp(\oom)\right\}
\end{equation}
and, for $\mu, \nu \in \mathcal M^+( \om)$ the modified $1$-Wasserstein distance
\begin{equation}\label{eq:wbidef}
Wb_1(\mu,\nu):=  \inf \left\{ \int_{\oom \times\oom} |x-y| \,d\gamma(x,y) : \gamma \in \Pi b(\mu,\nu) \right\},
\end{equation}
where
$$
\Pi b(\mu,\nu) := \left\{ \gamma \in \mathcal{M}^+( \oom \times \oom) : \pi^1_{\sharp}( \gamma)|_{\om} = \mu \text{ and } \pi^2_{\sharp}( \gamma)|_{\om} = \nu\right\}.
$$
Here we used the notation $\pi^i: \oom_1 \times \oom_2 \to \oom_i$ for the canonical projection from $\oom_1 \times \oom_2$ onto $\oom_i$, $i=1,2$. The distance $Wb_1$ differs from the classical Wasserstein distance because we allow mass to be ``transported'' also on $\partial \Omega$: in our case the \textit{transport plan} $\gamma$ takes values in $\mathcal{M}^+( \oom \times \oom)$ instead of $\mathcal{M}^+( \om \times \om)$. For $\sigma \in \mathcal{M}(\om)$, we  set
$$
\Gamma b(\sigma) := \left\{ \gamma \in \mathcal{M}^+( \oom \times \oom) : \pi^1_{\sharp}( \gamma)|_{\om}, \pi^2_{\sharp}( \gamma)|_{\om} \in \mathcal{M}^+( \om)  \text{ and }   \pi^1_{\sharp}( \gamma)|_{\om} -  \pi^2_{\sharp}( \gamma)|_{\om}= \sigma \right\}
$$
and define the Wasserstein functional 
\begin{equation}\label{eq:wass-funct-def}
\|\sigma\|_W:=  \inf \left\{ \int_{\oom \times\oom} |x-y| \,d\gamma(x,y) : \gamma \in \Gamma b(\sigma) \right\}.
\end{equation}
It was proved in \cite[Proposition 2.9]{FG10} that the space $\left( \mathcal{M}^+( \om ), Wb_1(\cdot,\cdot)\right)$ is a geodesic metric space. In fact, the result was stated for the $Wb_2$ distance but the same proof works for  $Wb_1$ as well. We also refer to \cite{BBS97} and \cite{BB01}.	 Moreover, as stated in \cite[p. 313]{San15}, the infimum in \eqref{eq:wbidef} is attained  and the following duality formula holds:
\begin{align}\label{eq:dualityWasser-Kanto}
Wb_1(\mu,\nu)= \min \left\{ \int_{\oom \times\oom} |x-y| \,d\gamma(x,y) : \gamma \in \Pi b(\mu,\nu) \right\}=\| \mu-\nu \|_{KR(\om)}.
\end{align}

Another important result is the Kantorovich duality, which holds in this setting as well. Let $C_0(\oom)$ be the set of all continuous functions on $\oom$ vanishing on $\d\om$ and  $\Phi_c(\oom)$ be the set of functions $(\vp,\psi) \in C_0(\oom) \times C_0(\oom)$ satisfying $\vp(x) + \psi(y) \leq |x-y|$ for $\mu$-a.e. $x \in \oom$ and $\nu$-a.e. $y \in \oom$. Then
\begin{equation}\label{duality-kantorov}
Wb_1(\mu,\nu) = \sup_{(\vp,\psi) \in \Phi_c(\oom)} \left\{ \int_{\om} \vp \,d\mu + \int_{\om} \psi \,d\nu \right\}.
\end{equation}
We refer to the proof of \cite[Theorem 1.3]{Vil03} and in particular, \cite[pp. 26-27, case 1]{Vil03}. The only two modifications that are required are the following: \begin{itemize}
\item One should use the class $C_0$ instead of  the one of bounded  continuous functions  denoted by $C_b$.\footnote{This comment should be used every time we modify a proof  and will not be repeated.}
\item  At the top of p. 27, where it is demonstrated that the functional on the right-hand side of \eqref{duality-kantorov} is well-defined, in order to deduce that $\vp= \tilde \vp$ and $\psi= \wt \psi$, one should utilize that the functions in $ \Phi_c(\oom)$  vanish on $\d \om$  instead of the fact that $\mu(\om)$ and $\nu(\om)$ are equal (which does not hold in our case).
\end{itemize}

Let us emphasize that it is not necessary to invoke \cite{FG10} in order to show that $Wb_1$ is a metric since the case $p=1$  is easier. Indeed,  we can follow the proof of  \cite[Theorem 4.1]{Ed11} (making some easy modifications) and use  \eqref{duality-kantorov} to  conclude \eqref{eq:dualityWasser-Kanto}.  Then, combining \cite[Lemma 4.3]{Ed11} and a standard  argument that can be found for instance in \cite[Lemma 1.13]{Sle16}, we can show directly that $\|\cdot\|_{KR}$ is a norm  using the fact that for a closed set $C \subset \om$, the function  $f_n(x)=n^{-1}(1-n \, \dist(x,C))^+ \in \Lip(\oom)$ and vanishes on $\d \om$ for $n \geq  [\dist(C,\d \om)]^{-1}$. Thus, $Wb_1$ becomes a metric and by (the proof of) \cite[Lemma 4.5]{Ed11}, it holds that $\|\cdot\|_W=\|\cdot\|_{KR(\om)}$ on $\mathcal{M}(\om)$.

\vv

For any  $ \vp \in \lipp(\oom)$, we  define
\begin{equation}\label{eq:def-distribution}
 \widehat \vp (\mu):= \int_\om \vp \,d\mu.
\end{equation}
Then  the following theorem holds:
\begin{theorem}\label{thm:Lip is dual of measure} 
The map $\vp \mapsto \widehat \vp$ is a linear bijection of  $  \lipp(\oom) $ onto the Banach dual of $(\mathcal{M}(\om), \|\cdot\|_{KR(\om)} )$ and $\|\widehat \vp\|=\|\vp\|_{\Lip(\oom)}$. Thus, the Banach dual of $(\mathcal{M}(\om), \|\cdot\|_{KR(\om)} )$  is indeed $\textup{Lip}_{01}(\oom) $. 
\end{theorem}

\begin{proof}
We follow the proof of \cite[Theorem 7.3]{Ed11}. One only has to check the results in sections 2, 3, 4, 6, and 7 up to Theorem 7.3. Since most of the arguments  require minor changes we will provide only the roadmap for the proof apart from Lemma 7.2 and Theorem 7.3,  that require some care.
\begin{itemize}
\item One can deal with Section 3 invoking the Kantorovich duality \eqref{duality-kantorov}.
\item In Section 4, one only has to check Theorem 4.5, since Theorems 4.1  and 4.4 are taken care of by \eqref{eq:dualityWasser-Kanto} and  \cite{FG10}. The proof is exactly the same. For future reference, we get that $\| \cdot \|_W = \| \cdot \|_{KR(\om)}$.
\item From Section 6, we only need Theorem 6.1, which is already stated in the desired generality. 
\item Lemma 7.1 should be stated just for points in $\om$ and thus, the second part of the proof of Lemma 7.2 works as it is for points in $\om$. To extend it to the boundary points it is enough to use that $\lipp(\oom)$ functions vanish on the boundary. Indeed, let $\ve_z$ stand for the Dirac mass at $z \in \om$, and let  $x\in \om$ and $y\in \d\om$. Then, for any $\vp \in \lipp(\oom)$ we have
$$
\frac{|\vp(x)-\vp(y)|}{|x-y|}=\frac{|\vp(x)|}{|x-y|}\leq \frac{|\vp(x)|}{\dist(x, \d\om)}= \frac{|\widehat{\vp}(\ve_x)|}{\dist(x, \d\om)}= |\widehat{\vp}(\delta_x)|,
$$
where $\delta_x:=\dist(x, \d\om)^{-1} \ve_x$. Note that if $d(x) \in \d\om$ is such that $|x-d(x)|=\dist(x,\d \om)$,
$$
\|\ve_{x}\|_{KR(\om)}=\sup_{\vp \in \lipp(\oom)} \vp(x)=\sup_{\vp \in \lipp(\oom)} (\vp(x)-\vp(d(x))) \leq  |x-d(x)|=\dist(x,\d \om),
$$
and thus, $\|\delta_{x}\|_{KR(\om)}\leq 1$. Hence,
$$
\sup_{x\in\om, y\in \d\om}\frac{|\vp(x)-\vp(y)|}{|x-y|} \leq \sup\bigl\{|\widehat{\vp}(\nu)|: \nu \in \mathcal{M}(\om), \|\nu\|_{KR(\om)}\leq 1\bigr\}.
$$
Since $\|\cdot\|_W=\|\cdot\|_{KR}$, the result follows by the latter estimate  and the proof of Lemma 7.2. 
\item  To show Lemma 7.3 we should also deal with the boundary points. For $f \in \mathcal{M}(\om)^*$, we define $u(x)=f(\ve_x)$, if $x \in\om$, and notice that $u(x)-u(y)=f(\ve_x)-f(\ve_y)=f(\ve_x -\ve_y)$, which is $1$-Lipschitz by Lemma 7.1. To extend the definition of $u$ to the boundary, for $x \in \d \om$, we define  $u(x)= \lim_{k \to \infty} u(x_k)$, where $x_k$ is a sequence of points in $\om$ so that $x_k \to x$. This is well-defined, since if $y_k \in \om $ is another sequence so that $y_k \to x$, then 
$$
 |u(x_k)-u(y_k)| \leq \|f\| |x_k -y_k| \leq  \|f\| (|x_k-x| + |y_k-x|)
 $$ and so  $\lim_{k \to \infty} u(x_k)= \lim_{k \to \infty}  u(y_k)$. It is easy to show that the extension of $u$, that we still denote by $u$, is  $1$-Lipschitz on $\oom$ and vanishes on $\d \om$, i.e.,  $u \in \lipp(\oom)$. We will only deal with the latter. Fix $x \in \d\om$ and pick $\om \ni x_k \to x$ as $k \to \infty$. Then, in light of the identity $\|\cdot\|_W=\|\cdot\|_{KR}$, we obtain $|u(x_k)|=| f(\ve_{x_k})| \leq \|f\| \|\ve_{x_k}\|_{KR(\om)}$ and since
$$
\|\ve_{x_k}\|_{KR(\om)}=\sup_{\vp \in \lipp(\oom)} \vp(x_k)=\sup_{\vp \in \lipp(\oom)} (\vp(x_k)-\vp(x)) \leq  |x_k-x|,
$$
by taking limits as $k \to \infty$, we infer that $u(x)=0$.   Remark that
$$\widehat u(\ve_x)= u(x)=f(\ve_x), \text{ for } x \in \om,
$$ and so $\widehat u$ and $f$  are bounded linear functionals that agree on a dense subspace of $\mathcal{M}(\om)$ (namely, the space of finite  Borel measures whose support is a finite set). The theorem readily follows.\qedhere
\end{itemize}
\end{proof}

\vv

We move to the last part of this section, which is the identification of the completion of $\mathcal{M}(\om)$ (or, equivalently, the closure of  $\mathcal{M}(\om)$)  as a subspace of $(\lipp(\oom)^*, \| \cdot \|_{KR(\om)} )$. 
Following \cite{BCJ05}, we  denote this completion  by $\widetilde{\mathcal{M}}(\om)$ and emphasize that it is a strict subspace of $(\lipp(\oom)^*, \| \cdot \|_{KR(\om)} )$ since otherwise $\lipp(\oom)$ would  be a reflexive Banach space. If $u \in \widetilde{\mathcal{M}}(\om) \subset \lipp(\oom)^*$, the pairing $\langle u, \vp \rangle$  is well-defined for any $\vp \in \lipp(\oom)$. We define the first order distribution
$$
T_u(\vp)=\langle u,  \vp \rangle, \, \text{ for } \vp \in C^\infty_0(\oom),
$$
where $ C^\infty_0(\oom)$ are the smooth functions in $\oom$ vanishing on $\d\om$.  Then the following theorem holds:

\begin{theorem}\label{thm:completion of M}
Let $\om \subset \rn$ be a bounded  open  set. Then
\begin{equation}\label{eq:Completion}
\Bigl\{ T_u : u \in  \widetilde{\mathcal{M}}(\om) \Bigr\} = \bigl\{ -\dv \sigma : \sigma \in  L^1(\om, \rn) \bigr\}
\end{equation}
as subsets of $\mathcal{D}'(\om)$. Moreover, if $V_0$ is the closed subspace $\{ \sigma \in L^1(\om, \rn): \dv \sigma =0\}$, the linear map $\sigma \in L^1(\om, \rn)/V_0   \mapsto -\dv \sigma \in \wt{\mathcal{M}}(\om)$ is an isometry, i.e.,
$$
\| \sigma\|_{ L^1(\om, \rn)/V_0} =\|- \dv \sigma\|_{KR(\om)}.
$$
\end{theorem}

\begin{proof}
The proof is almost the same as the one of \cite[Theorem 2.4]{BCJ05} and so we shall only make a few comments:
\begin{itemize}
\item In the statement of Lemma 2.1, $\vp_n \in C^\infty_0(\oom)$ and in display (2.1), $\vp_n \to 0$ (i.e., $c=0$).
\item In the proof of Lemma 2.1, one should define $S_\delta\colon \lipp(\oom) \to C^\infty_0(\oom)$ to be the approximation $v_\delta$ given in the proof of \cite[Theorem 3.1, Steps 1 and 2]{BBDP03} setting $\Sigma=\d\om$. It is clear from the proof of \cite[Theorem 3.1]{BBDP03}  that if $u \in \lipp(\oom)$,  then $v_\delta \in C^\infty_0(\oom)$ satisfying the conditions (1)-(3) in the proof of  Lemma 2.1. The rest of the argument is identical.  Let us point out that this approximation works in arbitrary open sets.
\item The proof of Lemma 2.3 is the same. In particular, it is even simpler as we do not need to rule  out the constants.
\item In Lemma 2.5, for $p>n$, one should define $a_p$ over functions $u \in W^{1,p}_0(\om)$ (see  \cite[p. 15, display (21)]{BBDP03}).  
Then the  infimum is attained at a unique point $u_p \in C_0(\oom)$ without the  condition $\int_{\om} u_p =0$ since this is just the unique solution of the Dirichlet problem $\Delta_p u=\mu$  in  $\om$ for $u \in \bigcap_{q < \frac{n(p-1)}{n-1}} W^{1,q}_0(\om)$. For $p$ large enough, since the boundary values are zero in the Sobolev sense, we may extend $u_p$ by zero and use  Morrey's inequality to obtain a H\"older continuous representative on $\oom$ that vanishes on $\d \om$. The rest of the proof is almost the same and is  based on  the one of \cite[Theorem 4.1]{BBDP03}. 
\item Using Theorem \ref{thm:Lip is dual of measure} and following the proof of Theorem 2.4 {\it mutatis mutandis}, we  arrive to the desired conclusion. We record that the  reason why we can extend the arguments to arbitrary domains is the fact that we require for our functions to vanish on the boundary. Thus we can extend them by zero to the exterior without the requirement that the domain is such that we can construct continuous or Sobolev  extensions to its complement. \qedhere
\end{itemize}
\end{proof}

\vv

For the applications in the next section we need the following two lemmas.

\begin{lemma}\label{lem:L2embeddingM}
If $\om$ is a bounded open set, then $L^2(\om)^* \subset \wt{ \mathcal{M}}(\om)$ and the inclusion map  $ i\colon 	L^2(\om)^* \to \wt{ \mathcal{M}}(\om)$ is continuous with constant depending only on the Lebesgue measure of  $\om$. 
\end{lemma}

\begin{proof}
Let  $f \in L^2(\om)$ and let us define 
$$
u_f:=\Delta_\om^{-1} f(x)= \int_\om G_\om (x,y) f(y) \,dy,
$$
and recall that $-\Delta u_f=f$ in the weak sense in $\om$ and   $u_f \in W^{1,2}_0(\om)$ satisfying $\|u_f\|_{W^{1,2}(\om)} \leq  \|f\|_{L^2(\om)}$.
 Define $\sigma_f= \nabla u_f$ and note that $-\div \sigma_f =f$  weakly in $\om$. By Cauchy-Schwarz,
 $$
 \int_\om |\sigma_f| \leq |\om|^{1/2} \|\sigma_f\|_{L^2(\om)} =   |\om|^{1/2} \|\nabla u_f\|_{L^{2}(\om)} \leq |\om|^{1/2} \|f\|_{L^2(\om)}.
 $$
 We conclude  by Theorem \ref{thm:completion of M}. 
\end{proof}

\begin{remark}
It is clear  that $i :  W^{-1,2}(\om) \to  \wt{ \mathcal{M}}(\om)$ is an embedding.
\end{remark}

\vv

\begin{definition}
Let $\om \subset \rn$ be a bounded open set and  $I \subset \R$ be an open interval.  If $W^{1,2}(\om)$ stands for the inhomogeneous Sobolev space with norm $\|u\|_{L^2(\om)} +\|\nabla u\|_{L^2(\om)}$, we define the parabolic Sobolev space 
$$
\mathcal{W}(\om \times I):= \left\{ u \in L^2(I ; W^{1,2}(\om) ): \d_t u \in L^1(I; \widetilde{\mathcal{M}}(\om) ) \right\}
$$
and the semi-norm
$$
\| u \|_{\dot{\mathcal{W}}(\om \times I )}:=  \| \nabla u \|_{L^2(\om \times I)} + \int_{I} \|\d_t u(\cdot, t) \|_{KR(\om)}\,dt.
$$
If we equip $\mathcal{W}(\om \times I)$ with the norm 
$$
\| u \|_{ \mathcal{W}(\om \times I )}:=\|  u \|_{L^2(\om \times I)}+\| u \|_{\dot{\mathcal{W}}(\om \times I )}
$$
it becomes a Banach space.
\end{definition}

\vv 

\begin{lemma}\label{lemma:compact-embed}
Let $\om \subset \rn$ be a bounded  domain so that $W^{1,2}(\om) \hookrightarrow L^2(\om)$ compactly and  let $I \subset \R$ be an open interval. Then the parabolic Sobolev space $\mathcal{W}(\om \times I)$ embeds compactly in $L^2(\om \times I)$.
\end{lemma}

\begin{proof}
Set  $\textbf{X}=W^{1,2}(\om )$, $\textbf{B}=L^2(\om )$, $\textbf{Y}=\wt{\mathcal{M}}(\om)$. By hypothesis, the inclusion map $i:\textbf{X} \to  \textbf{B}$ is compact  and, by Lemma \ref{lem:L2embeddingM}, the inclusion map $i:\textbf{B}^* \to  \textbf{Y}$ is continuous.  Consequently, if the sequence $\{u_j\}_{j \geq 1} \subset \mathcal{W}(\om \times I)$ is uniformly bounded in the $\mathcal{W}(\om \times I )$-norm,   we can apply  \cite[Corollary 4]{Sim87} to deduce that $\{u_j\}_{j \geq 1}$ has a subsequence that converges in $L^2(\om \times I)$-norm.
\end{proof}

\vvv

\section{Blow-ups in time varying  domains } \label{sec8}

\begin{lemma}\label{lem:Sobolev-extension}
Let $\om \subset \Rn1$ be an open set and $C_r:=B_r \times I_r$ be a cylinder of radius $r$  centered on $\d\om$. Set $\om_r:= \om \cap C_r$ and let  $u$ be a caloric function in $\om_{4r}$ that vanishes continuously on $\d \om \cap C_{4r}$. If  $\bar u$ is its extension by zero in $C_{4r} \setminus \om$, then the following hold:
\begin{enumerate}
\item $\bar u \in W^{1,2}(C_{2r})$.
\item $\partial _t \bar u \,\chi_{\om_{2r}}= \partial_t u  \,\chi_{\om_{2r}}$ and $\nabla \bar u \,\chi_{\om_{2r}}= \nabla u \,\chi_{\om_{2r}}$.
\item For  all $t \in I_{r}$, it holds that $\vp u(\cdot,t) \in W_0^{1,2}( (\om_{r})_t)$, for any $\vp \in \lipp(\overline{B_{r}})$, and $\bar u(\cdot,t) \in W^{1,2}(B_{r})$.
\end{enumerate}
\end{lemma}
\begin{proof}
The proof is standard and follows from arguments similar to  those in the proof of  \cite[Lemma 4.12]{AMT16}. See also Theorem 9.17 and Proposition 9.18 in \cite{Bre11}. We omit the details.
\end{proof}

\vv

The following lemma is a boundary Caccioppoli-type inequality for the time derivative of the square of a non-negative solution that is crucial for the blow-up lemmas.


\begin{lemma}\label{lem:bdry-Caccioppoli-time}
Under the same assumptions as in Lemma \ref{lem:Sobolev-extension}, if we set $v:=\bar u^2/2$,    it holds that 
\begin{equation}\label{eq:bdry-Caccioppoli-time}
\int_{ I_r} \| \partial_t v(\cdot,t) \|_{KR(B_r)} \,dt \lesssim_{n} \frac{1}{r}\| v \|_{L^1( C_{2r})}.
\end{equation}
\end{lemma}

\begin{proof}


In view of Lemma \ref{lem:Sobolev-extension}, $\bar u \in W^{1,2}(C_{2r})$, and, by \cite[Theorem 7.20]{Watson}, $\bar u$ is a subtemperature of $H$ and so weakly  subcaloric in $C_{4 r}(\bar \xi)$. Applying the standard parabolic Caccioppoli's inequality for non-negative subcaloric functions  we get
\begin{equation}\label{eq:space-Caccioppoli}
\begin{split}
\int_{C_{r}} |\nabla \bar u (x,t) |^2 \,dxdt &\lesssim \frac{1}{r^2} \int_{C_{2r}}  \bar u (x,t)^2 \,dxdt.
\end{split}
\end{equation}

 Since $\partial_t u \in L^2(C_{2r})$, then for a.e. $t \in I_{2r}$, it holds that $\partial_t u(\cdot,t) \in L^2(B_{2r})$. Fix  such $t \in  I_{r}$  and set  $D:=(\om_{r})_t$, which is an open subset of $\rn \times\{t\}$. If  $\vp \in \lipp(\overline{B_{r}})$,  since  $\vp u (\cdot,t) \in W^{1,2}_0(D)$, for fixed $\ve>0$, there exists $\psi \in C^\infty_c(D)$ such that
$$
\| w(\cdot,t)\|_{W^{1,2}(D)}:=\| \vp u(\cdot,t) - \psi\|_{W^{1,2}(D)} <\ve.
$$
Thus,  
\begin{align*}
\int  \partial_t  v(z,t) \vp(z)& \,dz = \int \partial_t \bar u(z,t)  \bar u(z,t) \vp(z) \,dz\\
&= \int \partial_t \bar u(z,t)  w(z,t) \,dz + \int \partial_t \bar u(z,t)  \psi(z) \,dz \\
&<  \ve \|\partial_t \bar u(z,t)\|_{L^2(B_{r})}  + \int_D \partial_t u(z,t)  \psi(z) \,dz\\
& =  \ve \|\partial_t \bar u(z,t)\|_{L^2(B_{r})} +  \int_{D} \Delta u(z,t)  \psi(z) \,dz=:II_1(\ve) +II_2,
\end{align*}
where we used that $\partial_t u=\Delta u$ in $D$ (pointwisely). In light of the fact that $\Delta u(\cdot,t) \in L^1_{\loc}(D)$ (as $\Delta u$ is continuous in $\om_{4r}$), it is clear  that 
\begin{align*}
II_2 &=-  \int_{D} \nabla_z  u(z,t) \nabla_z \psi(z)\,dz\\
&=\int_{D} \nabla_z  u(z,t) \nabla_z w(\cdot,t)\,dz -\int_{D} \nabla_z u(z,t) \nabla_z (\vp  u(\cdot,t))(z)\,dz\\
& <  \ve \|\nabla \bar  u(\cdot,t)\|_{L^2(B_{a/16})} -\int_D \nabla_z u(z,t) ( \nabla_z \vp(z) u(z,t) +  \vp(z) \nabla_z u(z,t) )\,dz\\
&=:II_{21}(\ve) + II_{22}.
\end{align*}
 Recalling that $|\vp| \leq r$ and $| \nabla \vp | \leq 1$ in $B_{r}$, we obtain
$$
II_{22} \leq  r \int_{B_{r}} |\nabla \bar u(z,t)|^2\,dz + \int_{B_{r}} \bar u(z,t) |\nabla \bar u(z,t)|\,dz.
$$
Therefore, taking limits as $\ve \to 0$, we infer that
$$
\int  \partial_t v(z,t) \vp(z) \,dz \leq r \int_{B_{r}} |\nabla \bar u(z,t)|^2\,dz + \int_{B_{r}} \bar u(z,t) |\nabla \bar u(z,t)|\,dz,
$$
which, by Cauchy-Schwarz's inequality and \eqref{eq:space-Caccioppoli},  implies that 
\begin{align*}
\int_{ I_r} \|  \partial_t v(\cdot,t) \|_{KR(B_r)} \,dt &\leq  r^{-1}\| \bar u\|^2_{L^2(C_{2r})},
\end{align*}
concluding the proof of \eqref{eq:bdry-Caccioppoli-time}.
\end{proof}

\begin{remark}
 It is important to work with $v$ instead of $\bar u$ since we want to exploit the caloricity of $u$ in $D$ as subcaloricity in the whole cylinder is not enough. Thus, taking the $t$-derivative of $v$ gives us $\bar u \d_t \bar u$,  and as $\bar u$  vanishes in  $C_r \setminus D$, we have that  $\bar u \vp$  is  supported in $\overline D$ and vanishes continuously on $\d D$. This way the domain of integration is restricted to the set $D$ where $u$ is caloric, which wouldn't have been achieved if we had just used $\d_t \bar u$. Of course, the same proof works for $u$ in $D$ and $\vp_t \in \lipp(\overline D)$ ($\vp$ depends on $t$ now), but it is not helpful for the blow-up lemmas. Indeed, in that case, we would have a uniformly bounded sequence of functions in a sequence of Sobolev spaces and thus, the compact embedding in $L^2$ would not be  applicable.
\end{remark}

\begin{remark} The need for the introduction of the Sobolev space $\mathcal{W}(B_r \times I_r)$ comes from the term $\int_{B_r}|\nabla \bar u|^2 \vp$ we obtain at the end. If we did not have an $L^\infty$ bound for $\vp$ but rather an $L^p$ bound  for large $p$, one would hope to apply H\"older's inequality and then use higher integrability estimates for  the gradient of $\bar u$ and Caccioppoli. Although, as $\bar u$ is not caloric but merely subcaloric in the whole cylinder, such estimates are out of the question.
\end{remark}


\vv

Now, we are ready to proceed to the first blow-up lemma.

\begin{lemma}\label{lem:blowuplemmaCPC1}
Let $\Omega\subset \Rn1$ be an open set which is quasi-regular for $H$ and regular for $H^*$. Let also  $\omega=\omega^{\bar p}_\Omega$ be the associated caloric measure with pole at $\bar p\in \Omega$. Assume that $F \subset \mathcal P^*\Omega\cap \mathcal S'\Omega$ is compact, $\bar \xi_j$ is a sequence in $F$, and there exists $r_j\to 0$ and $c_j>0$ such that  $C(\bxi_j,ar_j/4) \cap \d\om=C(\bxi_j,ar_j/4) \cap \mathcal{P}^*\om$, where $a \in (0,1/2)$ is as in Lemma \ref{lem:tbcdc}, and 
\[
\omega_j\coloneqq c_j T_{\bar \xi_j, r_j}[\omega]\rightharpoonup \omega_\infty
\]
for some non-zero Radon measure $\omega_\infty$. Let us also assume that there exists a subsequence of $r_j$ and a constant $c>0$ so that
\begin{equation}\label{eq:CPCblowuplemma}
\Cap\bigl(\overline{E(\bar\xi_j; r^2_j)}\setminus \Omega\bigr)\geq c\, r_j^n.
\end{equation}
If   $u=G_\om(\bar p, \cdot)$  in $\om$ and $u=0$ in $\Rn1 \setminus \overline{\om}$, let us denote  
\[
u_j(\bar x)\coloneqq c_j u\bigl(\delta_{r_j}\bar x+ \bar \xi_j\bigr)r^n_j.
\]
 and
\[
C_r\coloneqq C_r(\bar 0) = B_r(0) \times (-r^2, r^2) \eqqcolon B_r\times \mathfrak I_r, \qquad r>0.
\]
Then, if   $\alpha=a/16 \in (0,1/32)$, there exists $R>0$, a subsequence of $\{r_j\}_{j}$ and a non-negative function $u_\infty \in L^2(C_{\alpha R})$ such that
$u_j \to u_\infty$  in $L^2(C_{\alpha R})$-norm. Moreover $u_\infty$ is adjoint caloric in $C_{ \alpha R}\cap \{u_\infty>0\}$,
\begin{equation}\label{eq:bounduinfty1}
\|u_\infty\|_{L^2(C_{ \alpha R})}\lesssim_{ a, R}\omega_\infty\bigl(\overline{C_{MR}}\bigr)
\end{equation}
and 
\[
\int \varphi\, d\omega_\infty=\int_{\Rn1}u_\infty\, H\varphi, \qquad \text{ for all }\varphi\in C^\infty_c(C_{ \alpha R}).
\]
\end{lemma}

\begin{proof}
For simplicity, we will only prove the lemma assuming $\bar\xi_j\equiv\bar \xi$ for some $\bxi \in \mathcal{P}^*\Omega  \cap \mathcal{S}\Omega$, since the proof of the general case is analogous. Let us recall that for $\bar \xi_0\in\Rn1$, $r>0$, and $a>0$ we denote
\[
\widehat R^+_a(\bar \xi_0;r)=B_r(\xi_0)\times(t_0-(ar)^2/2,t_0+r^2).
\]
and it holds that  $C_{ar/2}(\bar \xi_0)\subset \widehat R^+_a(\bar \xi_0; r)$. 

As   \eqref{eq:CPCblowuplemma} is satisfied, by Lemma \ref{lem:tbcdc} we have that there exists $a\in (0,1/2]$ such that 
\[
\Cap\bigl(\overline{R^-_a(\bar \xi;r_j)}\setminus \Omega\bigr)\gtrsim c r_j^2,
\]
and so, by \eqref{eq:Bourgain}, it holds that
\[
\omega^{\bar z}\bigl(C_{Mr_j}(\bar \xi)\bigr)\geq c>0,
\]
for all $\bar z\in C_{ar_j}(\bar \xi)\subset \hat R^+_{a}(\bar \xi; r_j)$.{ Since $\omega_\infty\neq 0$, there exists $R>0$ such that $\omega_\infty(C_R)>0$. Without loss of generality we may assume that $R=a/8$.} By passing to a subsequence, if necessary, we may assume that $\bar p\in \Omega\setminus C_{ar_j}(\bar \xi)$ and combining the latter bound with Lemma \ref{lem:estim_green_caloric_measure} for $ar_j/2$ and $M'=2M/a$, we obtain
\begin{equation}\label{eq:blowuplem1wu1}
\omega^{\bar p}\bigl(C_{M r_j}(\bar \xi)\bigr)=\omega^{\bar p}\bigl(C_{M' ar_j}(\bar \xi)\bigr)\gtrsim (ar_j)^n\, u(\bar y),  \quad \text{ for }\bar y\in C_{ar_j/8}(\bar \xi).
\end{equation}
In particular, 
\begin{equation}\label{eq:Linftybounduj}
\limsup_{j \to \infty}\sup_{C_{a/8}(\bar 0)} u_j(\bar y) \lesssim a^{-n}\limsup_{j \to \infty} \omega_j\bigl(C_{M}(\bar 0)\bigr) \lesssim \omega_\infty\bigl(\overline{C_{M}(\bar 0)}\bigr).
\end{equation}

Recall that $G(\bar p, \cdot)$ is adjoint caloric in $\om \setminus \{\bar p\}$ and $\om$ is regular for $H^*$, and so  $G(\bar p, \cdot)$ is in $C^{2,1}(\om \setminus \{\bar p\}) \cap C(\om \cup \d_e^* \om \setminus \{\bar p\})$ vanishing on $ \d_e^* \om$. Thus, for $j$ large, it holds that $G(\bar p, \cdot)$ is adjoint caloric in $ \om \cap  C_{ar_j/4}(\bar \xi)$ and, by assumption,  $\d \om \cap  C_{ar_j/4}(\bar \xi) =  \mathcal{P}^*\om \cap  C_{ar_j/4}(\bar \xi)$.
By  Lemmas \ref{lem:Sobolev-extension} and \ref{lem:bdry-Caccioppoli-time}, and the standard parabolic Caccioppoli's inequality for non-negative subcaloric functions, along with a simple rescaling argument and   \eqref{eq:Linftybounduj}, we infer that there exists $j_0 > 1$ large such that for  $j \geq j_0$, if $v_j=u_j^2/2$,
\begin{equation}\label{eq:Sobolev-bound}
\begin{split}
\int_{C_{a/16}} |\nabla v_j(x,t) |^2 \,dxdt &+ \int_{ \mathfrak I_{a/16}} \|  \partial_t v_j(\cdot,t) \|_{KR({B_{a/16} } )} \,dt \\
 &\lesssim_{a,n} \left(\omega_\infty\bigl(C_{M}(\bar 0)\bigr)^2+1 \right) \int_{C_{a/8}} u_j (x,t)^2 \,dxdt\\
 & \lesssim_{a,n}  \omega_\infty\bigl(\overline{C_{M}(\bar 0)}\bigr)^2\left(1+\omega_\infty\bigl(C_{M}(\bar 0)\bigr)^2 \right).
\end{split}
\end{equation}

We  apply Lemma \ref{lemma:compact-embed} to find a subsequence of $\{v_j\}_{j \geq j_0}$ and a non-negative  function $v_\infty\in L^2(C_{a/16}) $ such that 
\[
v_j \to v_\infty  \text{ strongly in } L^2(C_{a/16}) \text{ and a.e. in } C_{a/16}.
\]
Note that if $\bz\in C_{a/16}$ is so that $v_\infty(\bz)=0$, then trivially, $u_j(\bz) \to 0$. If $\bz\in C_{a/16}$ is so that $v_\infty(\bz)>0$, then, if we set $ u_\infty=\sqrt{2 v_\infty}$, we have that
$$
|u_j(\bz)-u_\infty(\bz)|= \frac{|u_j^2(\bz)- 2v_\infty(\bz)|}{u_j(\bz)+ u_\infty(\bz)}\leq \frac{2|v_j(\bz)- v_\infty(\bz)|}{ u_\infty(\bz)} \to 0, \text{ as } j \to \infty.
$$
Thus, $u_j \to  u_\infty$ a.e. in $C_{a/16}$ and using \eqref{eq:Linftybounduj} to get a uniform $L^2$ bound on $C_{a/16}$, we deduce that $u_j \to   u_\infty$ weakly in $L^2(C_{a/16})$ and $ u_\infty$ satisfies \eqref{eq:bounduinfty1}. Moreover, by Cauchy–Schwarz, it holds that
$$
\left| \int_{C_{a/16}}u_j^2 -  \int_{C_{a/16}} u_\infty^2  \right| \leq \int_{C_{a/16}} |u_j^2 -  u_\infty^2| \lesssim_{a,n} \|v_j - v_\infty\|_{L^2(C_{a/16})}  \to 0, \text{ as } j \to \infty.
$$
Hence,  $\|u_j\|_{L^2(C_{a/16})} \to \|u_\infty\|_{L^2(C_{a/16})}$ and since $u_j \to   u_\infty$ weakly in $L^2(C_{a/16})$, by the Radon-Riesz Theorem, we conclude that $u_j \to u_\infty$ in  $L^2(C_{a/16})$-norm. 

Now,  a change of variables gives that 
\[
\int \varphi\, d\omega_j=\int u_j\, H\varphi, \qquad \varphi \in C^\infty_c(C_{a/16})
\]
which, if we take  limits as $j\to \infty$, entails
\[
\int \varphi\, d\omega_\infty=\int u_\infty\, H\varphi, \qquad \varphi \in C^\infty_c(C_{a/16}).
\]
In order to complete the proof of the lemma, it suffices to choose $ \alpha =a/16$.
\end{proof}

\vv

The previous lemma applies to the following ``two-phase" blow-up lemma.

\begin{lemma}\label{theorem:propertiesblowup}
Let $\Omega^\pm$ be two disjoint open sets in $\Rn1$ which are quasi-regular for $H$ and regular for $H^*$ and let $\omega^\pm$ be the respective caloric measures with poles $\bar p_\pm\in\Omega^\pm$.  Assume that $F \subset \mathcal P^*\Omega^+\cap \mathcal P^* \Omega^-\cap \mathcal S'\Omega^+\cap \mathcal S'\Omega^-$ is compact, $\bar \xi_j$ is a sequence in $F$ and that $\Omega^+$ and $\Omega^-$  satisfy the joint TBCPC. Let us also assume  that there exists  $r_j\to 0$, $c_j>0$, and a constant   $c>0$, such that  $C(\bxi_j,ar_j/4) \cap \d\om^\pm=C(\bxi_j,ar_j/4) \cap \mathcal{P}^*\om^\pm$, where $a \in (0,1/2)$ is as in Lemma \ref{lem:tbcdc}, and 
\[
\omega^+_j\coloneqq c_j T_{\bar\xi_j, r_j}[\omega^+]\rightharpoonup\omega^+_\infty
\]
and 
\begin{equation}\label{eq:convblowuplemmaCPC2}
\omega^-_j\coloneqq c_j T_{\bar\xi_j, r_j}[\omega^-]\rightharpoonup c\, \omega^+_\infty\eqqcolon \omega^-_\infty.
\end{equation}
Let {$u^\pm\coloneqq G_{\Omega^\pm}( \bar p_\pm ,\cdot)$} on $\Omega^\pm$, $u^\pm \equiv 0$ on $\Rn1\setminus \Omega$ and denote
\[
u_j^\pm(\bar x)\coloneqq c_j u^\pm\bigl(\delta_{r_j}\bar x+ \bar \xi_j\bigr)r^n_j.
\]
The following properties hold:
\begin{enumerate}[(a)]
\item
There exists $\alpha \in (0,1/16)$, $R>0$, a subsequence of $r_j$, $u^\pm_\infty\in L^2(C_{\alpha R})$, and  $u_\infty\in L^2(C_{\alpha R})$, a non-zero adjoint caloric function in $C_{\alpha R}$,  such that $u^\pm_j$ converge in $L^2(C_{\alpha R})$-norm to  $u^\pm_\infty$ and  $u_j\coloneqq u^+_j-c^{-1}u_j^-$ converges in $L^2(C_{\alpha R})$-norm to  $u_\infty=u^+_\infty-c^{-1}u_\infty^-$.

\item
The function  $u_\infty$ extends to an adjoint caloric function in $\Rn1$ and it holds that
\begin{equation}\label{eq:blow-up-bound-global}
\|u_\infty\|_{L^2(C_r)}\lesssim_{a} r^{-n}\omega^+_\infty\bigl(\overline{C_{Mr/a}}\bigr) \qquad  \qquad\text{for all } r>0, 
\end{equation}
and for any $\varphi \in C^\infty_c(\Rn1)$,
\begin{equation}\label{eq:blow-up-caloric measure-global}
\int\varphi\,d\omega^+_\infty= \int u^+_\infty H\varphi = \frac{1}{2}\int_{\Rn1} |u_\infty|\,H \varphi.
\end{equation}
\item We have that
\[
\supp\omega^+_\infty= \partial \om^\pm_\infty:=\partial \{u^\pm_\infty>0\} = \{u_\infty=0\}=:\Sigma_\infty,
\]
and $d\omega^+_\infty=-{\d_{\nu_t^+}  u_\infty}\, d\sigma_{\Sigma_\infty}$, where $\sigma_{\Sigma_{\infty}}$ stands for the parabolic surface measure on ${\Sigma_{\infty}}$ and $\nu_t^+$ is the measure-theoretic outer unit normal on the time-slice $(\om^+_\infty)_t$.
\end{enumerate}
\end{lemma}

\begin{proof}
As in  the proof of the previous lemma, we assume for simplicity that $\bar \xi_j\equiv\bar \xi\in \mathcal P^*\Omega^+\cap \mathcal P^*\Omega^-$ for all $j$ and  $\omega^+_\infty(C_{a/16})>0$.  Since the domains $\Omega^+$ and $\Omega^-$ are disjoint,  the subadditivity property for the thermal capacity \cite[Theorem 7.45-(b)]{Watson} entails
\[
\Cap \bigl(\overline{E(\bar \xi;r_j^2)}\setminus \Omega^+\bigr)+ \Cap \bigl(\overline{E(\bar \xi;r_j^2)}\setminus \Omega^-\bigr)\geq \Cap \bigl(\overline{E(\bar \xi;r_j^2)}\bigr), \qquad j\geq 0,
\]
and so, possibly after passing to a subsequence, without loss of generality, we can assume   that
\[
\mathcal \Cap\bigl(\overline{E(\bar\xi; r_j^2)}\setminus \Omega^+\bigr)\geq \frac{\Cap\bigl(\overline{E(\bar\xi;r_j^2)}\bigr)}{2}\approx r_j^{n},\qquad j\geq 0.
\]
Moreover, as $\Omega^+$ and $\Omega^-$  satisfy the joint TBCPC condition, there exists q subsequence such that
\[
\Cap\bigl(\overline{E(\bar \xi; r_j^2)}\setminus \Omega^-\bigr)\gtrsim r_j^n\, \qquad j\geq 0.
\]
providing us with a common sequence $r_j$ such that Lemma \ref{lem:blowuplemmaCPC1} applies to both $\omega^+$ and $\hm^-$. Thus, for $\alpha=a/16 \in (0,1/32)$ there exist two adjoint subcaloric functions $u^\pm_\infty$ such that $u^\pm_j\to u^\pm_\infty$ in $L^2(C_{\alpha})$ and
\begin{equation}\label{eq:blowuplemmaCPC21212}
\int \varphi \, d\omega^\pm_\infty = \int u^\pm_\infty \, H\varphi, \qquad \varphi \in C^\infty_c(C_{\alpha}).
\end{equation}
In particular, $u_j$ converges in $L^2(C_{\alpha})$-norm to the function $u_\infty\coloneqq u^+_\infty - c^{-1} u^-_\infty$   and, for $\varphi \in C^\infty_c(C_{\alpha})$,
\begin{equation}\label{eq:blowuplemm2adjcal}
\begin{split}
\int u_\infty H\varphi &=\int u^+_\infty H\varphi - c^{-1}\int u^-_\infty H\varphi \overset{\eqref{eq:blowuplemmaCPC21212}}{=}\int \varphi\, d\omega^+_\infty -  c^{-1}\int \varphi\, d\omega^-_\infty\overset{\eqref{eq:convblowuplemmaCPC2}}{=} 0.
\end{split}
\end{equation}
which proves that $u_\infty$ is adjoint caloric in $C_{\alpha}$.

Furthermore, as $\supp u^+_j\cap \supp u^-_j=\varnothing$  for every $j$, it holds that
\[
\int_{C_{\alpha}}u^+_\infty u^-_\infty = \lim_{j\to \infty}\int_{C_{\alpha}} u^+_j u^-_j = 0,
\]
which implies 
\begin{equation}\label{eq:disjoint_support}
\mathcal L^{n+1}(\supp u^+_\infty \cap \supp u^-_\infty\cap C_{\alpha})=0.
\end{equation}
As a result of  \eqref{eq:blowuplemmaCPC21212},  $\{u^+_\infty> 0\}\cap C_\alpha$ is a set of positive $\mathcal{L}^{n+1}$-measure and, thus, $u_\infty \not \equiv 0$. Another consequence is that 
\begin{equation}\label{eq:pos_neg_part_u_infty}
u^+_\infty=u_\infty\chi_{\{u_\infty>0\}}\qquad \text{ and }\qquad u^-_\infty= - c^{-1}\,u_\infty\chi_{\{u_\infty < 0\}}.
\end{equation}
In particular,  $u^\pm_\infty$ are continuous in $C_{\alpha}$ because $u_\infty$ is an adjoint caloric function.

Let us prove (b).
The same argument as the one in the proof Lemma \ref{lem:blowuplemmaCPC1} shows that, given $K \geq 1$,
\[
u^\pm_j(\bar y)\lesssim_a \omega^\pm_j\bigl(C_{KM}(\bar 0)\bigr),\, \qquad \bar y\in C_{Ka/8}(\bar 0)
\]
with the implicit constant independent of $j$.
Moreover
\[
\limsup_{j}\omega^\pm_{j}\bigl(C_{KM}(\bar 0)\bigr)\leq \omega^\pm_\infty\bigl(\overline{C_{KM}(\bar 0)}\bigr), \, \qquad \text{ for all }K\geq 1.
\]
So arguing again as we did in the proof of  Lemma \ref{lem:blowuplemmaCPC1}, we get that for any  fixed $K=1, 2, \ldots$, the sequence $u_{j}$ admits a converging subsequence in $L^2(C_{aK/4})$. By a  standard diagonalization argument we can show that $u_\infty$ extends to a caloric function in $\rrn$  satisfying \eqref{eq:blow-up-bound-global}. As it is easy to see that 
\eqref{eq:blow-up-caloric measure-global} holds,  the proof of (b) is concluded.

Let us turn our attention to the proof of (c). Let $\bar x\in \supp\omega^+_\infty$ and assume that there exists a cylinder $C_r(\bar y)$ such that $\bar x \in C_r(\bar y)\subset \{u^+_\infty>0\}$. In particular, \eqref{eq:pos_neg_part_u_infty} implies that $C_r(\bar y)\subset\{u^-_\infty=0\}$. Then, if $\varphi \in C^\infty_c(C_r(\bar y))$ such that $\varphi(\bar x)>0$, since $u_\infty$ is globally adjoint caloric, we have
\begin{equation}\label{eq:bl12345}
0<\int \varphi\, d\omega^+_\infty=\int u^+_\infty H\varphi=\int u_\infty H\varphi=0,
\end{equation}
so we get a contradiction.
A similar argument shows that $C_r(\bar y)\not\subset\{u^+_\infty=0\}$ for any cylinder such that $\bar x\in C_r(\bar y)$.
So $C_r(\bar y)\cap \{u^+_\infty>0\}\neq \varnothing$ and $C_r(\bar y)\cap \{u^+_\infty=0\}\neq \varnothing$, which shows that
\begin{equation}\label{eq:blup2sub}
\supp\omega_\infty^+\subset\partial \{u^\pm_\infty>0\}.
\end{equation}

We now claim that
\begin{equation}\label{eq:claimcpc}
\partial \{u^\pm_\infty >0\} \subset \supp \omega^+_\infty.
\end{equation}
To prove this, we take $\bar x=(x,t)\in \partial \{u^+_\infty >0\}$ and   assume, by contradiction,  that   $\bar x \in C_r(\bar y)\subset (\Rn1 \setminus \supp \omega^+_\infty)$ for some $\bar y\in (\supp\omega^+_\infty)^c$ and  $r>0$. In particular,
\[
\int u^+_\infty H\varphi \overset{\eqref{eq:blow-up-caloric measure-global}}{=}\int \varphi\, d\omega^+_\infty=0 \qquad \text{ for all }\varphi \in C^\infty_c\bigl(C_r(\bar y)\bigr),
\]
which gives that $u^+_\infty$ is adjoint caloric in $C_r(\bar y)$. 
Since $\bar x \in \partial \{u^+ >0\}\cap C_r(\bar y)$, we have $u^+_\infty(\bar x)=0$.
Hence, the strong minimum principle for adjoint caloric functions (which readily follows from \cite[Theorem 3.11]{Watson}) gives that $u^+_\infty\equiv 0$ on $\{(\zeta, \theta )\in C_r(\bar y):\theta< t\}$.{ In particular, $u^+_\infty=u_\infty=0$ in an open cylinder $C_\theta(\bar \zeta)$, where the first equality follows from \eqref{eq:blow-up-caloric measure-global}. So the fact that $u_\infty$ extends to a globally adjoint caloric function and the unique continuation principle \cite[Theorem 1.2]{Po96} entail $u_\infty\equiv 0$ on the whole $\Rn1$.} This contradicts (a) and proves the claim \eqref{eq:claimcpc}.

An immediate consequence of \eqref{eq:blup2sub} is that
\[
\supp\omega_\infty^+ \subset \{u^+_\infty=0\}\cap \{u^-_\infty=0\}  \subset \{u_\infty=0\} .
\]
In order to prove the converse inclusion, we observe that for $\bar x\in \Rn1$ such that $u_\infty(\bar x)=0$, \eqref{eq:pos_neg_part_u_infty} implies that $u^+_\infty(\bar x)=u^-_\infty(\bar x)=0$. Moreover, by the unique continuation,  $u_\infty$ cannot vanish in any open cylinder containing $\bar x$, so either $u^+_\infty$ or $u^-_\infty$ must be positive in that cylinder. This in turn implies
\[
\{u_\infty=0\}\subset \partial\{u^+_\infty>0\}\cup \partial \{u^-_\infty>0\}\overset{\eqref{eq:claimcpc}}{\subset}\supp\omega^+_\infty.
\]
 Lemma \ref{lem:Poisson-kernel}  concludes the proof.
\end{proof}

\vvv

Before presenting the blow-up result  that we need in order to prove the second part of Theorem \ref{theorem:blowup1}, we recall that the TFCDC condition is invariant under parabolic scaling (see Section \ref{sec:potential theory}).
{Moreover, if $\Omega$ is a domain that satisfies the TFCDC as in \ref{eq:CDC}, $\bar \xi\in \mathcal S\Omega$,  and $\rho>0$, then, if we denote $\tilde{\Omega}\coloneqq T_{\xi,\rho}[\Omega]$,  we have that $\mathcal S\tilde\Omega\subset \partial^*_e\tilde{\Omega}$.}

\vv

The next lemma lists the main properties of the blow-ups of  caloric measure and the adjoint Green's function in a domain satisfying the assumptions of Lemma \ref{lem:blowuplemmaCPC1} that has the TFCDC instead of just being regular for $H^*$. Our proof is inspired  from the one  of the corresponding result for harmonic measure in \cite{AMT16} and although 
it is the essentially the same for items (a)-(c), the proof of (d) is completely different.

\begin{lemma}\label{lemma:main_blowup_lemma}
Let $\Omega\subset \Rn1$ be an open set which is quasi-regular for $H$ and satisfies the TFCDC, and let $F \subset \mathcal P^*\Omega\cap \mathcal S\Omega$ be compact and $\{\bar \xi_j\}_j\subset F$. We denote by $\omega$  the caloric measure for $\Omega$ with pole at $\bar p \in\Omega$,  and assume that there exists $c_{j}> 0$ and $r_{j}\to 0$  such that $\omega_{j}=c_{j}T_{\bar \xi_j,r_{j}}[\omega]\warrow \omega_{\infty}$  for some non-zero Radon measure $\omega_{\infty}$. If we denote $\Omega_{j}\coloneqq T_{\bar \xi_j,r_{j}}(\Omega)$ and assume that there is $c>0$ so that
\begin{equation}\label{eq:CDCblowuplemma_CPC}
\Cap\bigl(\overline{E(\bar\xi_j; r^2_j)}\setminus \Omega\bigr)\geq c\, r_j^n, \qquad j=1,2,\ldots
\end{equation}
then there is a subsequence of $\{r_j\}_{j \geq 1}$ and  a closed set $\Sigma\subset \R^{n+1}$ such that
	\begin{enumerate}[(a)]
		\item For all $R>0$ sufficiently large, $\d\Omega_{j}\cap\cnj{ C_R(\bar 0)} \rightarrow \Sigma\cap \cnj{C_R(\bar 0)}$ in the Hausdorff metric.
		\item $\Rn1\setminus\Sigma=\Omega_{\infty}\cup \widehat\Omega_{\infty}$ where $\Omega_{\infty}$ is a nonempty open set and $\widehat\Omega_{\infty}$ is also open but possibly empty. Moreover, there exists $\lambda>1$ so that, if $C_r(\bar x)$ is a cylinder such that $\cnj{C_r(\bar x)}\subset \Omega_{\infty}$,  then $ C_{\lambda r}(\bar x)$ is contained in $\Omega_{j}$ for all $j$ large enough.
		\item $\supp \omega_{\infty}\subset \Sigma$. 
		\item If we set $u(\bar x)=G_{\Omega}(\bar p_+, \bar x)$ on $\Omega$ and $u\equiv 0$ on $(\Omega)^{c}$, and define
		\[u_{j}(\bar x)=c_{j}\,r_{j}^{n}\,u\bigl(\delta_{r_{j}}(\bar x)+\bar \xi_j\bigr),\]
		then the  sequence $u_{j}$ converges uniformly on compact subsets of $\bR^{n+1}$
to a non-zero function $u_{\infty}$ which is continuous in $\R^{n+1}$, adjoint caloric in $\Omega_\infty$, and vanishes in $(\om_\infty)^c$. Moreover if $a\in (0,1/2)$ and $M>1$ are the constants obtained in Lemmas  \ref{lem:tbcdc} and  \ref{lem:bourgain} respectively,   then for  $\bar x\in \Sigma$ and $r>0$, it holds that
{
		\begin{equation}\label{eq:u<wr}
		{u_{\infty}(\bar y)\lesssim_a r^{-n}\,  \omega_{\infty}\bigl({\overline{ C_{4a^{-1}Mr}(\bar x)}}\bigr),\qquad \bar y\in C_r(\bar x)\cap \Omega^+_{\infty}},
		\end{equation}}
		and
			\begin{align}\label{eq:ibp}
			\int \vphi \,d\omega_{\infty} = \int_{\R^{n+1}} u_{\infty}\,H \varphi  , \quad \textup{for any}\,\, \vphi \in C^\infty_c( \R^{n+1}).
			\end{align}
	\end{enumerate}
\end{lemma}

\begin{proof}
For simplicity, we assume $K=\{\bar \xi\}$, as the general case can be proved similarly.

\vv
\textbf{Proof of (a):} Let $R>0$ and  observe that $\bar 0\in \partial \Omega_j$ for all $j \geq 1$, which entails $\partial \Omega_j\cap \overline{C_R(\bar 0)}\neq \varnothing$.
The Hausdorff distance is a metric on the collection $\mathfrak C\bigl(\overline{C_R(\bar 0)}\bigr)$ of all closed subsets of $\overline{C_R(\bar 0)}$. Since $\overline{C_R(\bar 0)}$ is compact, $\bigl(\mathfrak C\bigl(\overline{C_R(\bar 0)}\bigr), d_H\bigr)$ is compact too. Thus, after passing to a subsequence, $\overline{C_R(\bar 0)}\cap \partial \Omega_j\to \overline{C_R(\bar 0)}\cap \Sigma$ in the Hausdorff distance sense for some $\Sigma \in \mathfrak C\bigl(\overline{C_R(\bar 0)}\bigr)$. A standard diagonalization argument concludes the proof of (a).

\vv

\textbf{Proof of (b):} We will first  prove that there exists a parabolic cylinder $C_\rho (\bar x')$ which is contained in all $\Omega_j$, for $j$ large enough. Arguing by contradiction we assume not. Let $\varphi\in C^\infty_c(\Rn1)$ be a non-negative function such that $\int \varphi\,d\omega^+_\infty\neq 0$ and $\supp \varphi\subset C_K(\bar 0)$ for some $K>0$.
Hence, there exists $\bar x_0$ belonging to $C_K(\bar 0)\cap \supp\omega_\infty$, and our counterassumption implies that 
\begin{equation}\label{eq:lemblowup2}
\rho_j\coloneqq \sup\bigl\{{\dist}_p(\bar x,(\Omega_j)^c): \bar x\in C_{2K}(\bar 0)\bigr\} \to 0, \qquad \text{ as }\, j\to \infty.
\end{equation}
We denote by $\bar \zeta_j(\bar x) \in(\Omega_j)^c$ a point which realizes ${\dist}_p(\bar x,(\Omega_j)^c)$.  In particular, since $\rho_j\to 0$ and $0\in \partial \Omega_j$ for all $j$,  we have that $\|\bar x-\bar \zeta_j(\bar x)\|\leq \rho_j<2K$, for $j$ large enough and $\bx \in C_{2K}(\bar 0)$. Moreover, $\|\bar x-\bar x_0\|	\leq \|\bar x\| + \|\bar x_0\|<3K$.

 Observe that $\Omega_j$ satisfies the TFCDC with the same parameters as $\Omega$ because of Lemma \ref{lem:geom_boundary_CDC}. Moreover, $\bar \zeta_j(\bar x)\in \mathcal S\Omega_j$ for $j$ big enough because $T_{\min}(\Omega_j)=(T_{\min}(\Omega)-\tau)/r_j^2\to -\infty$ and $T_{\max}(\Omega_j)=(T_{\max}(\Omega)-\tau)/r_j^2\to +\infty$ as $j\to \infty$. In fact, $\bar \zeta_j(\bar x)\in \partial^*_e\Omega_j$ (see the discussion before the statement of this lemma).

Note that for $\bx \in C_{2K}(\bar 0) $, we have $C_{2K}\bigl(\bar \zeta_j(\bar x)\bigr) \subset C_{4K}(\bar 0) \cap C_{5K}(\bx_0)$. If we consider $j$ large enough so that $ \bar p \not \in C_{16 K}(\bar 0)$,  then if we combine  \eqref{eq:estim_green_function}, \eqref{eq:Bourgain}, and \eqref{eq:CDCblowuplemma_CPC},  we have that for $\bar y\in C_{4K}(\bar 0) $,
\begin{equation}\label{eq:blowuplem1}
\begin{split}
u_j(\bar y )&= c_j r_j^n\, u\bigl(\delta_{r_j} \bar y + \bar \xi\bigr) \\
&\lesssim c_j r_j^n (4K r_j )^{-n} \omega\Bigl(\delta_{r_j}  C_{4KM}(\bar 0) +\bar \xi\Bigr) \approx_K \omega_j\bigl(C_{4KM}(\bar 0) \bigr).
\end{split}
\end{equation}
Remark that $M$ depends on $a$ and  is chosen so that \eqref{eq:bourgain_estimate}   holds.
Hence, using the $\gamma$-H\"older continuity at the boundary for $u_j$  in $\Omega_j$ (that holds with constants not depending on $j$ because of Lemma \ref{lem:geom_boundary_CDC}), we infer that
\begin{equation*}
\begin{split}
0&<\int \varphi\, d\omega_j = \int_{\Omega_j} u_j\,H\varphi \overset{\eqref{eq:boundaryHolder}}{\lesssim} \int_{\Omega_j}|H \varphi|\biggl(\sup_{C_{3K}(\bar \zeta_j(\bar x)) \cap \Omega_j}u^+_j\biggr)\biggl(\frac{\|\bar x-\bar \zeta_j(\bar x)\|}{2K}\biggr)^\gamma\, d\bar x\\
&\overset{\eqref{eq:blowuplem1}}{\lesssim_K}  \omega_j\bigl(C_{4KM}(\bar 0)\bigr) \Bigl(\frac{\rho_j}{2K}\Bigr)^\gamma  \int |H\varphi|\lesssim_{K, \varphi} \omega_j\bigl(C_{4KM}(\bar 0)\bigr) {\rho_j}^\gamma.
\end{split}
\end{equation*}
So taking the limit in the previous inequalities as $j\to \infty$ we obtain
\begin{equation*}
\begin{split}
0&<\int \varphi\, d\omega_\infty \lesssim_{a,K,\varphi} \limsup_{j\to \infty}\rho_j^\gamma\,\omega_j\bigl(C_{4KM}(\bar 0)\bigr)\leq \omega_\infty\bigl(\overline{C_{4KM}(\bar 0)}\bigr)\lim_{j\to 0}\rho^\gamma \overset{\eqref{eq:lemblowup2}}{=} 0,
\end{split}
\end{equation*}
which is a contradiction. 
Thus, after passing to a subsequence, if necessary, there exists a cylinder $C_\rho(\bar x')\subset \Omega_j$ for $j$ big enough.

For the the construction of $\Omega_\infty$ and, once we define $\widehat \Omega_\infty:=\Rn1\setminus\overline{\Omega_\infty}$, for the  remaining part of the proof of (b), we refer to \cite[p. 2143]{AMT16}.
For future reference, we remark that $\Omega_\infty$ is built as
\begin{equation}\label{eq:defomegainfty}
\Omega_\infty=\bigcup_{C_r(\bar x)\in \mathcal D}C_r(\bar x),
\end{equation}
where $\mathcal D$ is the collection of all open cylinders $C_r(\bar x)$ such that $r$ is rational, $\bar x$ has rational coordinates, $\overline{C_r(\bar x)}\subset \Rn1\setminus \Sigma$ and $C_r(\bar x)$ is contained in all but finitely many $\Omega_j$, for a proper subsequence.

\vv

\textbf{Proof of (c):}
If $C_r(\bar x)\subset \overline{C_r(\bar x)}\subset \Rn1 \setminus \Sigma$, we have that
\[
\omega_\infty\bigl(C_r(\bar x)\bigr)\leq \liminf_{j\to \infty}\omega_j\bigl(C_r(\bar x)\bigr)\leq \liminf_{j\to \infty}\omega_j\bigl(\Rn1 \setminus \partial\Omega_j\bigr)=0,
\]
which implies $\bar x\not \in \supp \omega_\infty$.

\vv
\textbf{Proof of (d):} 
Let us show that $u_j$ converges uniformly on compact subsets of  $\Rn1$ to a continuous function $u_\infty$. Recall that $\bar 0\in \partial \Omega_j$ for $j \geq 1$, and let  $j$ large enough so that $\bar p\in \Omega\setminus C_{400 K r_j}(\bar \xi).$ We can argue as in \eqref{eq:blowuplem1} and, for $K>0$ and $4\leq  \lambda<100$, we get that  for $j$ large, 
\begin{equation}\label{eq:unifboundjd}
u_j(\bar y)\lesssim {(\lambda K)}^{-n}\, \omega_j\bigl(C_{\lambda{K} M}(\bar 0)\bigr)\lesssim (\lambda K)^{-n}\omega_\infty\bigl(\overline{C_{\lambda K M }(\bar 0)}\bigr), \,\, \text{ if }\,\bar y\in C_{\lambda K}(\bar 0)\cap \Omega_j,
\end{equation}
where the last inequality holds because $\limsup_j \omega_j({C_{ \lambda K M}(\bar 0)})\leq \omega_\infty\bigl(\overline{C_{\lambda KM}(\bar 0)}\bigr)$. In particular, $\{u_j\}_j$ is uniformly bounded on $C_{\lambda K}(\bar 0)$ for $j$ large. 

We will show that it  is equicontinuous on $C_K(\bar 0)$ for $j$ large. To do so, we split into cases depending on the  positions of $\bar x, \bar y\in C_K(\bar 0)$.

\noindent{\bf Case 1}. If $\bar x, \bar y\in C_K(\bar 0)\setminus \Omega_j$, then trivially $|u_j(\bar x)-u_j(\bar y)|=0.$

\noindent{\bf Case 2}. Let us study the case $\bar x=(x,t)\in \Omega_j$ and $\bar y=(y,s)\in C_K(\bar 0)\setminus \Omega_j$.
If $t\leq s$, we denote by $\gamma(\bar x, \bar y)$ the (possibly degenerate) parabola with vertex at $\bar x$, whose axis is parallel to the time-axis so that $\bar y \in \gamma(\bar x, \bar y)$. If $\bar z=(z,u)\in \gamma(\bar x, \bar y)$, then, by construction, $|z-x|\leq |y-x|$ and $|u-t|\leq |s-t|$. Hence
\begin{equation}\label{eq:ineqparabd}
\|\bar z-\bar x\|=\max\{|z-x|, |u-t|^{1/2}\}\leq  \max\{|y-x|, |s-t|^{1/2}\}=\|\bar y-\bar x\|.
\end{equation}
Similarly, one can show  that $\|\bar y-\bar z\| \leq \|\bar y-\bar x\|$.

If $s< t$, we set $\tilde \gamma(\bar x, \bar y)\equiv \gamma (\bar y, \bar x)$, where $\gamma(\cdot, \cdot)$ is defined  above.
The parabola $\tilde \gamma(\bar x, \bar y)$ meets $\partial \Omega_j$ at a point, say $\bar z$. Since we choose $j$ big enough, \eqref{eq:boundaryHolder} applies (with constant independent on $j$ because the TFCDC condition holds with uniform constants for $\Omega_j$) and we have
\begin{equation*}
\begin{split}
|u_j(\bar x)-u_j(\bar y)|&=u_j(x)\lesssim_K \|\bar x-\bar z\|^\alpha\,\sup_{C_{2K}(\bar z)}u_j\leq \|\bar x-\bar z\|^\alpha \,\sup_{C_{3K}(\bar 0)}u_j \\
&\overset{\eqref{eq:unifboundjd}}{\lesssim_{K}} \|\bar x-\bar z\|^\alpha\,\omega_\infty\bigl(\overline{C_{3 K M}(\bar 0)}\bigr)\overset{\eqref{eq:ineqparabd}}{\lesssim_{\widetilde K}} \|\bar x-\bar y\|^\alpha.
\end{split}
\end{equation*}

\noindent{\bf Case 3}. Let us assume that $\bar x, \bar y\in \Omega_j\cap C_K(\bar 0)$ and that there exists a point $\bar z\in \Rn1\setminus \Omega_j$ such that $\dist\bigl(\bar z, \gamma(\bar x, \bar y)\bigr)\leq \|\bar x-\bar y\|/3$.
Without loss of generality, we assume $u_j(\bar x)\leq u_j(\bar y)$.
If $\bar \zeta\in \gamma(\bar x, \bar y)$ denotes a point such that $\|\bar z- \bar \zeta\|=\dist\bigl(\bar z, \gamma(\bar x, \bar y)\bigr)$, we have
\begin{equation}\label{eq:triangblowupCDC}
\|\bar y- \bar z\|\leq \|\bar y-\bar \zeta\|+\|\bar \zeta-\bar z\|\leq \|\bar x-\bar y\|+ \frac{1}{3}\|\bar x-y\|=\frac{4}{3}\|\bar x- \bar y\|\leq \frac{8}{3}K.
\end{equation}
Observe that $\bar y\in C_{3K}(\bar z)\subset C_{5K}(\bar 0)$. Since we have $\bar p \in \Omega\setminus C_{400 Kr_j}(\bar \xi),$ we can apply  \eqref{eq:boundaryHolder}  and  get
\begin{equation*}
\begin{split}
|u_j(\bar x)-u_j(\bar y)|&\leq 2 u_j(\bar y)\lesssim \biggl(\sup_{C_{3K}(\bar z)}u_j\biggr)\|\bar y-\bar  z\|^\alpha\leq \biggl(\sup_{C_{5K}(\bar 0)}u_j\biggr)\|\bar y-\bar z\|^\alpha\\
& \overset{\eqref{eq:unifboundjd}}{\lesssim_{K}} \omega_\infty\bigl(\overline{C_{5KM}}\bigr) \|\bar y-\bar z\|^\alpha\overset{\eqref{eq:triangblowupCDC}}{\lesssim_{ K}}\|\bar x-\bar y\|^\alpha.
\end{split}
\end{equation*}

\noindent{\bf Case 4}. Suppose that $\bar x, \bar y\in \Omega_j\cap C_K$ and that $\dist\bigl(\bar z, \gamma(\bar x, \bar y)\bigr)> \|\bar x-\bar y\|/3$ for all $\bar z\in \Rn1\setminus \Omega_j$.
We denote $\delta\coloneqq \|\bar x-\bar y\|/100$ and we cover $\gamma(\bar x, \bar y)$ with $N$ cylinders of the form $C_{\delta}(\bar y_k),$ for $\bar y_k\in \gamma(\bar x, \bar y)$, $k=1,\ldots, N$. The cylinders are chosen so that $C_{\delta}(\bar y_{k+1})\cap C_{\delta}(\bar y_{k})\neq \varnothing$, $\bar x\in C_{\delta}(\bar y_1)$ and $\bar y\in C_{\delta}(\bar y_N)$. For $i=1, \ldots, N-1$ we choose $\bar x_i\in C_{\delta}(\bar y_{i+1}) \cap C_\delta(\bar y_{i})$ and  define $\bar x_0=\bar x$ and $\bar x_N=\bar y$. Remark that  $N$ is independent of $\delta$ and $j$. Since $u_j$ is caloric in $\Omega_j\cap C_K(\bar 0)$, we can use the interior H\"older continuity (see e.g. \cite[Theorem 6.29]{Lieb96}) and write
\begin{equation*}
\begin{split}
|u_j(\bar x)-u_j(\bar y)|&\leq \sum_{i=0}^{N} |u_j(\bar x_{i+1})-u_j(\bar x_i)|\lesssim  \sum_{k=0}^{N} \|u_j\|_{L^\infty(C_{4\delta}(\bar y_i))}\|\bar x_{k+1}-\bar x_k\|^\beta\\
&\lesssim \|u_j\|_{L^\infty(C_{2K}(\bar 0))}  \|\bar x-\bar y\|^\beta  \overset{\eqref{eq:unifboundjd}}{\lesssim_{K}} \omega_\infty\bigl(\overline{C_{2 K M}(\bar 0)}\bigr) \|\bar x-\bar y\|^\beta,
\end{split}
\end{equation*}
where $\beta\in(0,1)$ and the implicit constants are independent  of $j$.

Combining the four cases above,  after passing to a subsequence, we obtain that $u_j$ is equicontinuous in $C_K$ for each $K \geq 1$. Hence, we can apply Ascoli-Arzel\`a's theorem and a standard diagonalization argument to get that $u_j$ has a subsequence that converges uniformly on compact subsets of $\Rn1$  to a continuous function $u_\infty$. Moreover, for  $\bar x\in \Sigma,$ $r>0$, and $ \bar y\in C_r(\bar x)\cap \Omega_\infty$, we have
\[
u_\infty (\bar y) = \lim_{j\to \infty} u_j (\bar y)\lesssim_a \limsup_j r^{-n}\hm_j(C_{4Ma^{-1}}(\bar x)) \leq r^{-n}\hm_\infty(\overline{C_{4Ma^{-1}}(\bar x)})<\infty,
\]
which gives \eqref{eq:u<wr}.
It remains to prove \eqref{eq:ibp}. To this end, for $\varphi\in C^\infty_c(\Rn1)$ it holds 
\begin{align}\label{eq:blowuplemma23}
\int\varphi\, d\omega_j&= c_j\int \varphi \circ T_{\bar\xi,r_j}\,d\omega = c_j\int H\bigl(\varphi\circ T_{\bar\xi,r_j}\bigr)\, u = c_j r_j^{-2}\int \bigr((H\varphi)\circ T_{\bar\xi,r_j}\bigl)\,u \notag\\
&=c_jr_j^{n}\, \int H\varphi(\bar\xi)\,u\bigl(\delta_{r_j}(\bar y)+\bar \xi\bigr)\,d\bar y = \int u_j\, H\varphi.
\end{align}
Taking limits as $j \to \infty$ in  \eqref{eq:blowuplemma23}, in light of the fact  that $\supp\omega_\infty\cap \Omega_\infty\subset \Sigma\cap \Omega_\infty$ by (c) and that $u_j\to u_\infty$ locally uniformly in $\Rn1$,  we have that \eqref{eq:ibp} holds. It is straightforward to see that $u_\infty$ is adjoint caloric in $\Omega_\infty$.

The function $u_\infty$ is not identically zero. Indeed, if $\varphi\in C^\infty_c(\Rn1)$ is a non-negative function  so that $\int\varphi\,d\omega^+_\infty>0$ and plug it in  \eqref{eq:ibp}, we obtain
\[
0<\int \varphi\,d\omega_\infty = \int u_\infty\,  H\varphi.
\]
This readily implies that $u_\infty \not \equiv 0$ in $\rrn$.

In order to finish the proof of (d), we claim that $u_\infty\equiv 0$ on $\Rn1\setminus \Omega_\infty$. Assume that there is $\bx \in \Rn1\setminus \Omega_\infty$ so that $u_\infty(\bar x)>0$. By the continuity of $u_\infty$, there exists $\delta>0$ such that $u_\infty(\bar y)>u_\infty(\bar x)/2$ for all $\bar y\in \overline{C_\delta(\bar x)}$. 
The local uniform convergence of $u_j$ to $u$ implies that $u_j>u_\infty(\bar x)/4$ on $\overline{C_\delta(\bar x)}$ for all $j$ big enough.
Thus, since $u_j$ vanishes outside of $\Omega_j$, we have that $\overline{C_\delta(\bar x)}\subset \Omega_j$ for $j$ large. In particular, it holds that $\overline{C_{\delta/2}(\bar x)}\subset \Rn1 \setminus \Sigma$, so $C_{\delta/2}(\bar x)\subset \Omega_\infty$ by the construction \eqref{eq:defomegainfty} and we can conclude that $\bar x\in \Omega_\infty$, which is a contradiction. This proves our claim and finishes the proof of the lemma.
\end{proof}

\vv

The next lemma studies the two-phase version of the previous result for domains that satisfy the joint TBCPC. Although this proof is  inspired by the one of \cite[Lemma 5.9]{AMT16}, there are several differences as well. Let us remark that an important feature of our proof is that  we do not assume the domains to be complementary (unlike the corresponding lemma in \cite{AMT16}), which gives us significant freedom in the applications in Section \ref{sec9} but also  requires additional care in the proof.

\begin{lemma}\label{lem:blow_up_CDC_two_phase}
Let $\Omega^\pm\subset \Rn1$ be two disjoint open sets that are quasi-regular for $H$ and  satisfy the  joint TBCPC and the TFCDC. Let $F \subset \mathcal P^*\Omega^+\cap \mathcal P^*\Omega^-\cap \mathcal S\Omega^+\cap \mathcal S\Omega^-$ be compact and $\{\bar \xi_j\}_j\subset F$. We denote by $\omega^{\pm}$  the caloric measure for $\Omega^\pm$ with pole at $\bar p_\pm \in\Omega^\pm$,  and   assume that there exists $c_{j}> 0$, $c>0$ and $r_{j}\to 0$  such that
\[
\omega_{j}^+\coloneqq c_{j}T_{\bar \xi_j,r_{j}}[\omega^+]\warrow \omega_{\infty}^+
\]
and
\[
\omega^-_j\coloneqq c_{j}T_{\bar \xi_j,r_{j}}[\omega^+]\warrow c\omega_{\infty}^+\eqqcolon \omega^-_\infty.
\]
We set $u^\pm(\bar x)=G_{\Omega^\pm}(\bar p_\pm, \bar x)$ on $\Omega^\pm$ and $u^\pm\equiv 0$ on $(\Omega^\pm)^{c}$, and  define
\[
u_{j}^\pm(\bar x)=c_{j}\,r_{j}^{n}\,u^\pm\bigl(\delta_{r_{j}}(\bar x)+\bar \xi_j\bigr).
\]
Then there exists a subsequence of $\{r_j\}_{j \geq 1}$,  $u^\pm_\infty$, $\Omega^\pm_\infty$, $\widehat \Omega^\pm_\infty$, and $\Sigma^\pm$ such that  the properties (a)-(d) in Lemma \ref{lemma:main_blowup_lemma} hold; moreover, if $u_{j}\coloneqq u_{j}^+ - c^{-1} u_{j}^-$, then $u_j \to u_{\infty}\coloneqq u_{\infty}^+ - c^{-1} u_{\infty}^-$ uniformly on compact subsets of  $\rrn$ and the following properties hold:

	\begin{enumerate}[(a)]
		\item $u_{\infty}$ is  an adjoint caloric function on $\R^{n+1}$.
		\item  $\widehat\Omega^\pm_\infty=\Omega_{\infty}^\mp$ and $\Sigma^+=\Sigma^-\eqqcolon\Sigma$. Moreover $\Sigma=\{u_{\infty}=0\}$, with $u_\infty>0$ on $\Omega_\infty^+$ and $u_\infty<0$ on $\Omega_\infty^-$. 
		\item $\Sigma$ is a set of Euclidean Hausdorff dimension $n$ and 
\begin{equation}\label{eq:poissonkernel_blowup_lemma}
d\omega_{\infty}^+=-{\partial_{\nu_t} u_\infty}\,d\sigma_{\d\Omega_{\infty}^+},
\end{equation}
where $\sigma_{\d\Omega_{\infty}^+}$ stands for the parabolic surface measure on ${\d\Omega_{\infty}^+}$ and $\nu_t$ is the measure-theoretic outer unit normal on the time-slice $(\om_\infty^+)_t$.
\end{enumerate}
	\label{l:blowup1}
\end{lemma}
\begin{proof}
Let us assume for brevity that $\bar \xi_j\equiv \bar \xi$.
Given $\Omega^\pm$ as in the statement, by the same pigeonholing argument as in Lemma \ref{theorem:propertiesblowup} and in light of the joint TBCPC property, we can find a subsequence of $r_j$ such that
\begin{equation}\label{eq:wlog_joint_TBCPC}
\mathcal \Cap\bigl(\overline{E(\bar\xi; r_j^2)}\setminus \Omega^\pm\bigr)\gtrsim r_j^{n},\qquad j\geq 0.
\end{equation}
So \eqref{eq:CDCblowuplemma_CPC} holds for both $\Omega^+$ and $\Omega^-$ and we can apply Lemma \ref{lemma:main_blowup_lemma} to find $u^\pm_\infty$, $\Omega^\pm_\infty$, $\widehat \Omega^\pm_\infty$, and $\Sigma^\pm$ satisfying the properties (a)-(d) of Lemma \ref{lemma:main_blowup_lemma}.

\textbf{Proof of (a):} Let $\varphi\in C^\infty_c(\Rn1)$. By hypothesis we have that $\omega_j^+\rightharpoonup \omega^+_\infty$ and $\omega^-_j \rightharpoonup c\omega^+_\infty$, and thus, by the local uniform convergence of $u^\pm_j$ on $\Rn1$, we can find a subsequence of $\{r_j\}_{j \geq 1}$ such that
\begin{equation*}
\begin{split}
\int u_\infty\, H\varphi  &= \lim_{j\to \infty}\int(u^+_j- c^{-1}u^-_j)\, H\varphi =\lim_{j\to \infty}\biggl(\int \varphi\, d\omega^+_j- c^{-1}\int\varphi\, d\omega^-_j\biggr)\\
&= \int\varphi\, d\omega^+_\infty - c^{-1}\int\varphi\, d\omega^-_\infty=\int \varphi\, d\omega^+_\infty-\int\varphi\, d\omega^+_\infty=0.
\end{split}
\end{equation*}

So $u_\infty$ is a weakly adjoint caloric function in $\Rn1$ and Lemma 	\ref{lem:weak_caloric_functions} applies. 

\vv

\textbf{Proof of (b):} Let us first prove that $\Rn1\setminus (\Omega^+_\infty\cup\Omega^-_\infty)=\{u_\infty=0\}$. By Lemma \ref{lemma:main_blowup_lemma}-(d) we have that the functions $u^\pm_\infty$ simultaneously vanish on $(\Omega^+_\infty)^c\cap (\Omega^-_\infty)^c=\Rn1\setminus (\Omega^+_\infty\cup\Omega^-_\infty)$. Thus, $\Rn1\setminus (\Omega^+_\infty\cup \Omega^-_\infty)\subset\{u^+_\infty=0\}\cap\{u^-_\infty=0\} \subset \{u_\infty=0\}$.

In order to prove the converse inclusion, we first observe that $u_\infty$ is non-zero since by Lemma \ref{lemma:main_blowup_lemma}-(d) it holds that $u^\pm_\infty\not\equiv 0$ and vanish outside $\om_\infty^\pm$.  By construction, we have that $u^\pm_\infty \geq 0$ on $\Omega^\pm_\infty$. Let us assume by contradiction that $u^+(\bar x)=0$ for some $\bar x\in\Omega^+_\infty$ and let 
\[
E^*(\bar x;r)\coloneqq\{(y,s)\in \Rn1: (y,t-s)\in E(\bar x, r)\}
\]
be a reflected heat ball so  that $\overline{E^*(\bar x;r)}\subset\Omega^+_\infty$. The mean value theorem for adjoint caloric functions (that readily follows from the mean value theorem for caloric functions \cite[Theorem 1.16]{Watson}) gives that
\begin{equation*}
\begin{split}
0=u_\infty(\bar x)&=(4\pi r)^{-n/2}\int_{E^*(\bar x;r)}\frac{|x-y|^2}{4(t-s)^2}u_\infty(y,s)\, dy\,ds\\
&=(4\pi r)^{-n/2}\int_{E^*(\bar x;r)}\frac{|x-y|^2}{4(t-s)^2}u_\infty^+(y,s)\, dy\,ds,
\end{split}
\end{equation*}
which implies that $u_\infty=u^+_\infty\equiv 0$ on $E^*(\bar x;r)$. In particular, $u_\infty$ has infinite order of vanishing at all the points of $E^*(\bar x;r)$ in the sense of \cite[p. 522]{Po96}. Hence, by the unique continuation principle for globally adjoint caloric functions \cite[Theorem 1.2]{Po96}, we deduce that $u_\infty \equiv 0$ on $\Rn1$, which is a contradiction.  The same argument can be applied to show that $u^-_\infty>0$ on $\Omega^-_\infty$.
So since $u^\pm_\infty = 0$ in $\rrn \setminus \om_\infty^\pm$, we get $u_\infty \neq 0$ in $\om_\infty^+ \cup\om_\infty^-$ or equivalently, 
\[\{u_\infty=0\}\subset (\Rn1\setminus \Omega^+_\infty)\cap (\Rn1\setminus \Omega^-_\infty)=\Rn1\setminus (\Omega^+_\infty\cup \Omega^-_\infty).
\]

Let us  remark that, by construction, $\Rn1=\Omega^\pm_\infty\cup \widehat \Omega^\pm_\infty\cup\Sigma^\pm$, where the unions are disjoint.
Hence
\begin{equation}\label{eq:dijunion}
\widehat\Omega^-_\infty=\bigl(\widehat\Omega^-_\infty\cap \Omega^+_\infty\bigr)\cup \bigl(\widehat\Omega^-_\infty\cap \widehat\Omega^+_\infty \bigr)\cup\bigl(\widehat\Omega^-_\infty\cap \Sigma^+\bigr).
\end{equation}

We claim that 
$$ 
\bigl(\widehat\Omega^-_\infty\cap \widehat\Omega^+_\infty \bigr)\cup\bigl(\widehat\Omega^-_\infty\cap \Sigma^+\bigr)= \emptyset.
$$

Firstly, let us show that $\widehat \Omega^+_\infty\cap \widehat \Omega^-_\infty=\varnothing$.
Indeed, since both sets are open, if $\bar x\in \widehat \Omega^+_\infty\cap \widehat \Omega^-_\infty$, there exists $\rho>0$ such that $C_\rho(\bar x)\subset \widehat \Omega^+_\infty\cap \widehat \Omega^-_\infty\subset (\Rn1\setminus\Omega^+_\infty)\cap (\Rn1\setminus \Omega^-_\infty) \subset \Rn1\setminus (\Omega^+_\infty\cup \Omega^-_\infty) \subset \{u_\infty=0\}$.
 Hence, by \cite[Theorem 1.2]{Po96} we have that $u_\infty\equiv 0$, which is a contradiction.

Secondly, we prove that $\widehat \Omega^-_\infty\cap \Sigma^+=\varnothing$. If $\bar x\in \widehat\Omega^-_\infty \subset \Rn1\setminus \Omega^-_\infty$, Lemma \ref{lemma:main_blowup_lemma}-(d) gives that $u^-_\infty(\bar x)=0$. Analogously, if $\bar x\in \Sigma^+\subset \Rn1\setminus \Omega^+_\infty$, we have $u^+_\infty(\bar x)=0$. Hence, if $\bar x\in \widehat \Omega^-_\infty\cap \Sigma^+$ we get $u_\infty(\bar x)=0$.
Moreover, $\widehat \Omega^-_\infty$ is an open set, hence $C_\varepsilon(\bar x)\subset \widehat\Omega^-_\infty$ for $\varepsilon$ small enough. In particular, $u^-_\infty=0$ on $C_\varepsilon(\bar x)$, which entails that $u_\infty\geq 0$ in this cylinder. So since $u_\infty$ is globally adjoint caloric and $u_\infty(\bar x)=0$, the strong minimum principle \cite[Theorem 3.11]{Watson} implies that  $u_\infty=0$ on $C^+_\varepsilon(\bar x)$. Thus, by the unique continuation we get $u_\infty\equiv 0$ on $\Rn1$, which is again a contradiction. 

 Therefore, by  \eqref{eq:dijunion}, we obtain $\widehat \Omega^-_\infty=\widehat \Omega^-_\infty\cap \Omega^+_\infty\subset \Omega^+_\infty$. Since $\Omega^+$ and $\Omega^-$ are  disjoint, we also have $\Omega^+_\infty\subset \Rn1 \setminus \overline{\Omega^-_\infty}=\widehat \Omega^-_\infty$, concluding that $\widehat\Omega^-_\infty=\Omega^+_\infty$ as desired. Symmetrically, we can show that $\Omega^-_\infty=\widehat \Omega^+_\infty$.

In order to finish the proof of (b), it is enough to observe that the construction of $\Sigma^\pm$ and the equality $\widehat\Omega^\pm=\Omega^\mp$ infer
\[
\Sigma^+=\Rn1\setminus \bigl(\Omega^+_\infty\cup\widehat \Omega^+_\infty)=\Rn1\setminus \bigl(\Omega^+_\infty \cup\Omega^-_\infty)
\]
and 
\[
\Sigma^- = \Rn1\setminus \bigl(\Omega^-_\infty\cup\widehat \Omega^-_\infty)=  \Rn1\setminus \bigl(\Omega^+_\infty \cup\Omega^-_\infty).
\]
Hence, the set $\Sigma\coloneqq \Sigma^+=\Sigma^-$ is well-defined and it coincides with $\Rn1\setminus (\Omega^+_\infty\cup\Omega^-_\infty).$ Since we have already proved  that $\Rn1\setminus (\Omega^+_\infty\cup\Omega^-_\infty)=\{u_\infty=0\}$, we obtain that  $\Sigma=\{u_\infty=0\}$ as wished.

\textbf{Proof of (c):} The fact that $\dim_H \Sigma \leq n$ is given by \cite[Theorem 1.1]{HL94}. To show that $\dim_H \Sigma =n$, consider two cylinders $C_+$ and $C_-$ contained in $\Omega^+_\infty$ and $\Omega^-_\infty$ respectively. By (b), $\pm u_\infty > 0$ on $C_\pm$ and, by continuity of $u_\infty$, any line connecting a point in $C_+$ to a point in $C_-$ has to cross $\Sigma$. Finally, the formula \eqref{eq:poissonkernel_blowup_lemma} holds because of Lemma \ref{lem:Poisson-kernel}.
\end{proof}

\vvv

\section{Proofs of the main theorems}\label{sec9}

\subsection{Proofs of Theorems \ref{theorem:blowup1} and \ref{theorem:tsirelson}}
%
%
Let us recall that we use the notation
\begin{equation*}
\begin{split}
\mathscr F&\coloneqq\bigl\{c\mathcal H^{n+1}_{p}|_{\Sigma_h}:c>0,\,h\in F^*(1)\bigr\}\\
&= \bigl\{  c\mathcal H^{n+1}_p|_V: c>0,\,V\,\textup{admissible}\, n \textup{-plane passing through the origin} \bigr\},
\end{split}
\end{equation*}
which  is a $d$-cone of Radon measures.

Given $\omega^+, \omega^-$, and a set $E$ as in the statement of Theorem \ref{theorem:blowup1}, we denote
\[
E^*\coloneqq\Bigl\{\bar \xi\in E:\lim_{r\to 0}\frac{\omega^+(C_r(\bar \xi)\cap E)}{\omega^+(C_r(\bar\xi))}=\lim_{r\to 0}\frac{\omega^-(C_r(\bar\xi)\cap E)}{\omega^-(C_r(\bar\xi))}=1\Bigr\},
\]
which, by Lemma \ref{lem:Besicovitch} and the mutual absolute continuity assumption for $\omega^\pm$, satisfies $\omega^\pm(E\setminus E^*)=0$.

For $\bar\xi\in E^*$, we define the Radon-Nikodym derivative
\[
h(\bar\xi)=\frac{d\omega^+}{d\omega^-}(\bar \xi)=\lim_{r\to 0}\frac{\omega^+(C_r(\bar \xi))}{\omega^-(C_r(\bar \xi))}=\lim_{r\to 0}\frac{\omega^+(C_r(\bar \xi)\cap E)}{\omega^-(C_r(\bar \xi)\cap E)},
\]
the set
\[
\Lambda\coloneqq \{\bar \xi\in E^*: 0<h(\bar\xi)<\infty\}
\]
and 
\[
\Gamma\coloneqq\{\bar\xi\in\Lambda: \bar\xi\text{ is a Lebesgue point for }h \text{ with respect to }\omega^+\}.
\]
Observe that $\omega^\pm(E\setminus \Gamma)=0$ (see Lemma \ref{lem:Besicovitch} and \cite[Remark 2.15 (3)]{Mattila}).

\begin{lemma}\label{lem:blowupabscont}
Let $\bar\xi\in\Gamma$, $c_j> 0$, and $r_j\to 0$ be such that $\omega^+_j=c_j T_{\bar\xi,r_j}[\omega^+]\rightharpoonup \omega^+_\infty$. Then $\omega^-_j=c_j T_{\bar\xi,r_j}[\omega^-]\rightharpoonup h(\bar\xi)\omega^+_\infty$.
\end{lemma}

\begin{proof}
The proof is identical to the one of \cite[Lemma 5.8]{AMT16} and we omit it.
\end{proof}

\vv

\begin{lemma}\label{lem:tan_cap_F_nonempty}
Let $\omega^+$ be as in Theorem \ref{theorem:blowup1}. For $\omega^+$-a.e. $\bar\xi\in\Gamma$ we have that
\[
\Tan(\omega^+,\bar\xi)\cap \mathscr F \neq \varnothing.
\]
\end{lemma}

\begin{proof}
Under the hypotheses of Theorem \ref{theorem:blowup1}, we have that $E\subset\supp \omega^+\cap \supp\omega^-\subset \partial_e\Omega^+\cap \partial_e\Omega^-$ and that $\Omega^+$ is disjoint from $\Omega^-$. By the classification of boundary points in Section \ref{sec:preliminaries},  this implies that $E\cap (\mathcal B\Omega^\pm\cup \partial_s\Omega^\pm)=\varnothing$ and thus, $E\subset \mathcal S'\Omega^+\cap \mathcal S'\Omega^-$.
Furthermore, since $E\subset \bigl(\mathcal P^*\Omega^\pm\bigr)^\circ$ by assumption, $C(\bxi,r) \cap \d\om=C(\bxi,r) \cap \mathcal{P}^*\om$ for all $r$ sufficiently small.

It is not difficult to see that \cite[Theorem 2.5]{Pr87} translates to our  context. In particular, the set $\Tan(\omega^+, \bar \xi)$ is nonempty for $\omega^+$-a.e. $\bar \xi\in \Gamma$. Let $\bar \xi\in \Gamma$ be such that $\Tan(\omega^+, \bar \xi)\neq \varnothing$ and consider $c_j>0$ and $r_j\to 0$ such that $c_jT_{\bar \xi, r_j}[\omega^+]\rightharpoonup \omega^+_\infty$ for some non-zero Radon measure $\omega^+_\infty \in \Tan(\omega^+, \bar \xi)$. Lemma \ref{lem:blowupabscont} implies that $c_jT_{\bar \xi, r_j}[\omega^-]\rightharpoonup h(\bar \xi)\omega^+_\infty$ and so we may apply  Lemma \ref{theorem:propertiesblowup} with $c=h(\bar \xi)$. Thus, we obtain  an adjoint caloric function $u_\infty$ in $\Rn1$ such  that
\[
\int \varphi\, d\omega^+_\infty=\int u^+_\infty H\varphi=\frac{1}{2}\int |u_\infty|\, H\varphi, \qquad \text{ for all }\varphi \in C^\infty_c(\Rn1),
\]
and
\(
\supp \omega^+_\infty=\{u_\infty=0\}\eqqcolon \Sigma.
\)
By Lemma \ref{lem:Poisson-kernel} we get that $\omega^+_\infty$ is  absolutely continuous with respect to the surface measure of $\Sigma=\partial \Omega^+_\infty$ and
\[
d\omega^+_\infty=-\partial_{\nu^+_t}u_\infty\, d\sigma|_{\Sigma},
\]
where $\partial_{\nu^+_t}$ is the derivative along the measure theoretic outer unit normal to the time-slice $\Omega^+_t$.
The set $\Sigma$ is the nodal set of an adjoint caloric function so, by Lemmas \ref{lem:space_regular_set} and \ref{lem:sigma_F},  for $\sigma|_E$-a.e. $\bar \zeta\in\Sigma$, there exists $\varepsilon>0$ such that $\Sigma$ agrees with a smooth admissible graph in $C_\varepsilon(\bar \zeta)$.
Hence, if we combine Lemmas \ref{lem:local-tengents} and  \ref{lem:second-tangents} with Corollary \ref{cor:tangent-rectifiable}, we can find a non-zero Radon measure $\mu\in \mathscr F^*(1)=\mathscr F$ such that $\mu\in \Tan(\omega^+, \bar \xi)$.
\end{proof}

\vv

With the results of Section \ref{sec:caloric_polynomials} at our disposal, the following lemma can be proved following the strategy of \cite[Lemma 6.1]{AM19}. Its proof relies on a corollary of a ``connectivity" lemma (see \cite[Corollary 3.12]{AM19}), which also holds in the parabolic setting as it does not use the Euclidean structure. 

\begin{lemma}\label{lem:dim_caloric_homog}
Let $\Omega$ be an open set in $\Rn1$ and, given $\bar p\in\Omega$, let $\omega\coloneqq\omega_\Omega^{\bar p}$ be its associated caloric measure. Let $\bar \xi\in \supp \omega$, suppose that $\Tan(\omega,\bar\xi)\subset \mathscr H_{\Theta^*}$ and that there exists $\lambda>0$ such that for all $\omega_h\in \Tan(\omega,\bar\xi)$ we have
\[
\|h\|_{L^\infty(C_r(\bar 0))}\lesssim_\lambda r^{-n}\omega_h\bigl({C_{\lambda r}(\bar 0)}\bigr), \qquad \text{ for }r>0.
\]
If $k$ is the smallest integer for which $\Tan(\omega,\bar\xi)\cap \mathscr{F}^*(k)\neq\varnothing$, then $\Tan(\omega, \bar\xi)\subset \mathscr{F}^*(k)$. Moreover,
\begin{equation}\label{eq:dim_caloric_homog}
\lim_{r\to 0}\frac{\log\omega(C_r(\bar \xi))}{\log r}={n+k}.
\end{equation}
\end{lemma} 

\vv

We are now ready to gather all the results we have proved so far in order to prove the first main theorem of the paper.

\begin{proof}[Proof of Theorem \ref{theorem:blowup1}] 
We recall that, arguing as in Lemma \ref{lem:tan_cap_F_nonempty}, we can assume  that $E\subset \mathcal S'\Omega^+\cap \mathcal S'\Omega^-$.
Moreover, Lemma \ref{lem:tan_cap_F_nonempty} infers that $ \Tan(\omega^+, \bar \xi)\cap \mathscr F \neq \varnothing$ for $\omega^+$-a.e. $\bar \xi\in \Gamma$, which, by  Lemma \ref{theorem:propertiesblowup}, gives that $\Tan(\omega^+, \bar \xi)\subset \mathscr H_{\Theta^*}$. Therefore, we can apply Lemma \ref{lem:dim_caloric_homog} with $k=1$ and  obtain
\[
\lim_{r\to 0}\frac{\log \omega(C_r(\bar\xi))}{\log r}=n+1,\qquad \text{ for }\omega^+\text{-a.e. }\bar \xi\in\Gamma.
\]
Since $\omega^+(E\setminus \Gamma)=0$,  the previous formula implies  that $\dim_{\mathcal H_p} \omega^+|_E=\dim_{\mathcal H_p} \omega^-|_E=n+1$.
If, in addition, $\Omega^\pm$ satisfy the TFCDC assumption, we can apply Lemma  \ref{lem:blow_up_CDC_two_phase} and use Lemma \ref{lemma:main_blowup_lemma}-(b) in order to conclude that \(\lim_{r\to 0}\Theta^{\mathscr F_\Sigma}_{\partial\Omega^\pm}(\bar\xi, r)=0\) for $\omega^\pm$-a.e. $\bar \xi\in E$.
\end{proof}

We will now show Theorem \ref{theorem:tsirelson} following the approach in \cite{TV18}.

\begin{proof}[Proof of Theorem \ref{theorem:tsirelson}] 
The proof  follows from  the ones of Lemmas \ref{lem:blowuplemmaCPC1} and \ref{theorem:propertiesblowup},  and so we will only sketch it. Let us assume that $\hm^1(E)>0$, which by mutual absolute continuity implies that $\hm^i(E)>0$ for $i=2,3$ as well. By Lemma \ref{lem:quasireg-subdomain} (or rather its proof), we can find   open sets $\wt \om^i \subset  \om_i$, $i=1,2,3$, which are regular for $H$ and $H^*$, and a set 
$$
\wt E \subset E \cap  \big( \cap_{i=1}^3( \mathcal{P}^* \wt \om_i)^0\big) \cap \supp \hm_{\wt \om^1}, 
$$
 so that $\hm^i(\wt E)>0$ for $i=1,2,3$, and $\hm_{\wt \om^1}, \hm_{\wt \om^2},$ and $\hm_{\wt \om^3}$ are mutually absolutely continuous on $\wt E$.  So it is enough to prove the result  assuming that  $\om_i$   is regular  for $H$ and $H^*$ for $i=1,2,3$.

To this end, by \cite[Theorem 2.5]{Pr87}, for $\hm^1$-a.e. $\bxi \in E$, we have that there exists  $c_j>0$ and $r_j \to 0$ such that $\hm^1_j \rightharpoonup \hm^1_\infty$, for some non-zero Radon measure $\hm^1_\infty$. Let $\Gamma_2$ and $\Gamma_3$ be defined as $\Gamma$ with $\hm^+=\hm^1$ and $\hm^-=\hm^i$, for $i=2,3$.  By Lemma \ref{lem:blowupabscont} and mutual absolute continuity, we obtain that $\hm^2_j \rightharpoonup h^2(\bar \xi) \hm^1_\infty$ and $\hm^3_j \rightharpoonup h^3(\bar \xi)  \hm^1_\infty$ for $\bxi \in \wt \Gamma :=\Gamma_2 \cap \Gamma_3$. Note that $\hm^1( \wt \Gamma)=0$ and so $\hm^i( \wt \Gamma)=0$, for $i=2,3$. By a similar pigeonholing argument as in the proof of Lemma \ref{theorem:propertiesblowup}, after passing to a subsequence, we may assume without loss of generality that 
\[
r_j^2 \Cap\bigl(\overline{E(\bar \xi; r_j^2)}\setminus \Omega^i\bigr)\overset{\eqref{eq:cdc-content}}{\gtrsim} \HH^{n+1}\bigl(\overline{E(\bar \xi; r_j^2)} \setminus \Omega^i\bigr) \gtrsim r_j^{n+2},  \quad j\geq 0 \,\,\textup{and}\,\, i=1,2.
\]
In particular, there exists $F_j \subset \overline{E(\bar \xi; r_j^2)} \setminus (\Omega^1\cap \Omega^2)$ such that $ \HH^{n+1}\bigl(F_j \bigr) \gtrsim r_j^{n+2}$. If we set $G_j=T_{\bxi,r_j}(F_j)\subset \overline{E(\bar 0; 1)} \setminus (\Omega^1_j\cap \Omega^2_j)$, then it is clear that $ \HH^{n+1}\bigl(G_j \bigr) \gtrsim 1$ and we have that 
\begin{equation}\label{eq:tsirelson-Gj}
 \int_{G_j} |u_j| =0, \quad \text{for all}\,\, j\geq 1.
\end{equation}
 Since $G_j$ is compact and  $ \HH^{n+1}\bigl(G_j \bigr) \gtrsim 1$ for all $j \geq 1$, there is a compact set $G$ such that, after passing to a subsequence, $G_j \to G$ in the Hausdorff metric and $ \HH^{n+1}\bigl(G \bigr) \gtrsim 1$. It is easy to see that $\chi_{G_j} \to g$ weakly in $L^2\left(\overline{E(\bar 0; 1)} \right)$ for some non-negative function $g \in L^2\left(\overline{E(\bar 0; 1)} \right)$, and thus,  we must have that $g= \chi_G$. Repeating the proof of  Lemma \ref{theorem:propertiesblowup} for $+=1$ and $-=2$, we  obtain that $u_j \to u_\infty$  in $L^2_{\loc}(\rrn)$ for some globally adjoint caloric function $u_\infty \not \equiv 0$ which vanishes only on a set of zero $\mathcal{H}^{n+1}$-measure. Taking limits as $j \to \infty$ in \eqref{eq:tsirelson-Gj}, we obtain
$$
\int_G |u_\infty| = \lim_{j \to \infty}  \int_{G_j} |u_j| =0,
$$
which, in turn, implies that $u_\infty=0$ on a set of positive Lebesgue measure reaching a contradiction.
\end{proof}

\vv

\subsection{Proofs of Theorems \ref{theorem:main_theorem2} and  \ref{theorem:main3} }\label{sec:thm2_3}

\begin{lemma}\label{lem:weak_conv_omega-}
Let $\mu^\pm$ be two halving Radon measures such that $\mu^- \ll \mu^+$ and $\supp\mu^+=\supp\mu^-\subset \partial_e\Omega^+$. Let us define $f\coloneqq \log\tfrac{d\mu^-}{d\mu^+}$ and assume that there is $r_j\to 0$  and {$\bar \xi_j\in \partial_e\Omega^+$} so that $\mu^+_j\coloneqq\mu^+\bigl(C_{r_j}(\bar \xi_j)\bigr)^{-1}\, T_{\bar \xi_j, r_j}[\mu^+]\rightharpoonup \mu$ for some Radon measure $\mu$ such that $\mu(C_1(\bar 0))>0$. Moreover, we assume that $\mu^+$ is so that for all $M>0$
\begin{equation}\label{eq:cond_lemma_102}
\lim_j\Bigg(\avint_{C_{Mr_j}(\bar \xi_j)}f\, d\mu^+\Bigg)\, \exp\Bigg(-\avint_{C_{Mr_j}(\bar \xi_j)}\log f\, d\mu^+\Bigg)=1.
\end{equation}
Then $\mu^-_j\coloneqq\mu^-\bigl(C_{r_j}(\bar \xi_j)\bigr)^{-1}\, T_{\bar \xi_j, r_j}[\mu^-]\rightharpoonup \mu$.
\end{lemma}
\begin{proof}
The proof of \cite[Lemma 8.1]{AM19} does not use the Euclidean structure  and  can be repeated in the parabolic setting.
\end{proof}

\vv

\begin{proof}[Proof of Theorem \ref{theorem:main_theorem2}]

Since $E$ is nonempty and relatively open in $\supp\omega^+$, we have that $\omega^\pm(E)>0$ and
\[
\lim_{r\to 0}\frac{\omega^+(C_r(\bar \zeta)\setminus E)}{\omega^+(C_r(\bar \zeta))}=0 \qquad \text{ for all }\bar \zeta\in E,
\]
so \cite[Lemma 14.5]{Mattila} gives $\Tan(\omega^+,\bar \xi)=\Tan(\omega^+|_E, \bar \xi)$.
If $\omega\in \Tan (\omega^+|_E,\bar \xi)=\Tan(\omega^+,\bar \xi)$, then by Lemma \ref{lem:tangent_meas_properties}-(1), $\omega=c\,T_{\bar 0, r}[\mu]$, for some constants $c>0$ and $r>0$, and some mesaure $\mu$ of the form
\[
\mu=\lim_{j\to \infty} \frac{1}{\omega^+(C_{r_j}(\bar \xi))}T_{\bar \xi,r_j}\bigl[\omega^+\bigr]=\lim_{j\to \infty} \frac{1}{\omega^+|_E(C_{r_j}(\bar \xi))}T_{\bar \xi,r_j}\bigl[\omega^+|_E\bigr],
\]
for some $r_j\to 0$, is such that $\mu(C_1(\bar 0))>0$ and the second equality holds because $E$ is relatively open in $\supp\omega^+$.
Let us observe that the hypothesis of Lemma \ref{lem:weak_conv_omega-} is satisfied for $\mu^\pm=\omega^\pm(E)^{-1}\omega^\pm|_E$ (the normalization is chosen so that $\mu^\pm$ are probability measures) and for $\bar \xi_j=\bar \xi$ for all $j$. Hence, we can conclude that
\[
\mu=\lim_{j\to \infty} \frac{1}{\omega^-|_E(C_{r_j}(\bar \xi))}\, {T_{\bar \xi,r_j}\bigl[\omega^-|_E\bigr]}=\lim_{j\to \infty}\frac{1}{\omega^-(C_{r_j}(\bar \xi))}T_{\bar \xi,r_j}\bigl[\omega^-\bigr].
\]

By Lemma \ref{theorem:propertiesblowup} there exists $g\in \Theta^*$ such that $\mu=\omega_g$ and  since the caloric measures associated with adjoint caloric functions form a $d$-cone and $\omega=c T_{\bar 0, r}[\mu]$, we have that  $\omega\in \mathscr H_{\Theta^*}$ too. More specifically there exists $h\in\Theta^*$ such that $\omega=\omega_h$. Let us define
\[
h_j(\bar x)\coloneqq \sum_{|\alpha|+2\ell=k}\frac{D^{\alpha,\ell}h(\bar 0)}{\alpha!\ell!}x^\alpha t^\ell,\qquad\qquad j>0,
\]
and let $k$ be the smallest number for which we have $h_k\not \equiv 0$. 
By Lemmas \ref{lem:second-tangents} and \ref{lem:tangent_caloric_function}, it follows that $\Tan(\omega_{h},\bar 0)\subset \Tan(\omega^+,\bar \xi)$ and 
\[
\Tan(\omega_h,\bar 0)=\{c\omega_{h_k}:c>0\}\subset \mathcal F^*(k).
\]
This shows that $\Tan(\omega^+,\bar \xi)\cap \mathcal F^*(k)\neq\varnothing$, which, in turn, by Proposition \ref{prop:connect714} implies $\Tan(\omega^+,\bar \xi)\subset \mathcal F^*(k)$. If we additionally assume that $\om^\pm$ have TFCDC,  then we can argue as in the proof of Theorem \ref{theorem:blowup1} to show \eqref{eq:theta-main2}.
\end{proof}

\vv

\begin{proof}[Proof of Theorem \ref{theorem:main3}]
Since  $F$ is compact, there exists $\rho>0$ such that $C_\rho(\bar \xi)\cap \supp\omega^+\subset E$ for all $\bar \xi \in F$
Arguing by contradiction, we assume that for every $d \in \NN$ we can find  sequences $\{\bar \xi_j\}_{j \geq 1} \subset  F$ and $r_j\to 0$ such that
\begin{equation}\label{eq:lem_dist_d}
d_1\bigl(T_{\bar{\xi}_j,r_j}[\omega^+], \mathcal P^*(d)\bigr)\geq \varepsilon >0.
\end{equation}
By the doubling assumption on $\omega^+$, after passing to a subsequence,  we have that
\[
\lim_{j\to \infty}\frac{1}{\omega^+(C_{r_j}(\bar \xi_j))}T_{\bar \xi_j,r_j}[\omega^+]=\lim_{j\to \infty}\frac{1}{\omega^+|_E(C_{r_j}(\bar \xi_j))}T_{\bar \xi_j,r_j}[\omega^+|_E]= \omega
\]
for some measure $\omega$, where the equality holds because $E$ is relatively open in $\supp\omega^+$. We claim that $\omega^-(C_{r_j}(\bar \xi_j))^{-1}T_{\bar \xi_j, r_j}[\omega^-]\rightharpoonup \omega$ as well. Indeed, since $\log f\in \VMO(\omega^+|_E)$ and $\omega^+|_E$ is doubling by assumption, $f$ is an $A_p$-weight on small enough cylinders by the John-Nirenberg theorem (see \cite[Chapter 6.2]{BAF} for the proof of this implication in the Euclidean setting), which still holds for spaces of homogeneous type. Hence, $\omega^-|_E\in VA(\omega^+|_E)$ by \cite[Corollary 7.8]{AM19} and the hypothesis of Lemma \ref{lem:weak_conv_omega-} is satisfied, so
\[
\omega=\lim_{j\to \infty}\frac{1}{\omega^-|_E(C_{r_j}(\bar \xi_j))}T_{\bar \xi_j,r_j}[\omega^-|_E]=\lim_{j\to \infty}\frac{1}{\omega^-(C_{r_j}(\bar \xi_j))}T_{\bar \xi_j,r_j}[\omega^-].
\]

Lemma \ref{theorem:propertiesblowup} gives that $\omega=\omega_h$ for some $h\in \Theta^*$, i.e.
\[
\int \varphi\, d\omega=\int h\, H\varphi, \qquad \varphi\in C^\infty_c(\Rn1).
\]
The measure $\omega_h$ is doubling because $\omega^+|_E$ is. Thus, there exists $\lambda\in \mathbb N$ such that for every  $k \in \NN$ and $N \in \NN$,
\begin{equation}\label{eq:doub_appl}
\omega_h\bigl(C_{2^{k+N}}(\bar 0)\bigr)\lesssim 2^{k\lambda} \omega_h\bigl(C_{2^N}(\bar 0)\bigr).
\end{equation}
{
Let $a$ and $M$ be as in Lemma \ref{theorem:propertiesblowup} and denote by $N$ a positive integer  such that $4a^{-1}M\leq 2^N$.
Inequality \eqref{eq:doub_appl} together with the Cauchy estimates \eqref{eq:cauchy_estimates}, gives that, for $\alpha\in\mathbb N^n$, $k\in \mathbb N$ and $\ell>0$ such that $|\alpha|+2\ell=m$,
\begin{equation*}
\begin{split}
|D^{\alpha,\ell}h(\bar 0)|&\lesssim 2^{-mk}\|h\|_{L^\infty\bigl(C_{2^k}(\bar 0)\bigr)}\\
&{\lesssim}_{a,M}2^{-k(m+n)}\omega_h\bigl(C_{2^{k+N}}(\bar 0)\bigr)\overset{\eqref{eq:doub_appl}}{\lesssim}2^{-k(m+n-\lambda)}\omega_h\bigl(C_{2^N}(\bar 0)\bigr),
\end{split}
\end{equation*}
}
where the second inequality follows from the proof of Lemma \ref{theorem:propertiesblowup}.
Thus, for $m+n-\lambda>0$ we have that the right hand side in the previous expression converges to $0$ as $k\to \infty$, which implies that $D^{\alpha,\ell}h(\bar 0)=0$. In particular, $h$ is a polynomial of degree at most $d=m+n-\lambda$, which contradicts \eqref{eq:lem_dist_d}. In order to prove \eqref{eq:beta_tilde_d}, it suffices to use  Lemma \ref{lemma:main_blowup_lemma} and argue as in the proof of Theorem \ref{theorem:blowup1}.
\end{proof}

\vv

\begin{proof}[Proof of Corollary \ref{cor:reifenberg}]
One can modify the proof above according to the original proof in \cite[pp.34-35]{KT06} to conclude the corollary bearing in mind that $\hm^- \ll \hm ^+$ implies that $t_-  \leq t_+$. We leave the details to the interested reader.
\end{proof}

\frenchspacing
\bibliographystyle{alpha}

\newcommand{\etalchar}[1]{$^{#1}$}
\def\cprime{$'$}

\end{document}